\documentclass[11pt]{amsart}
\usepackage{amsmath,amsthm,amssymb,amsfonts,mathtools}
\usepackage{hyperref}
\usepackage{cleveref}
\usepackage[margin=1in]{geometry}
\usepackage{mathrsfs}
\usepackage{tikz-cd}
\usepackage{pdflscape}
\usepackage{makecell}
\usepackage{xcolor}
\usepackage{multirow}
\usepackage{hyperref}
\usepackage{booktabs,longtable,array}
\usepackage{needspace}
\usepackage{enumitem}
\hypersetup{
    colorlinks=true,
    linkcolor=red,     
    citecolor=blue,     
    urlcolor=blue       
}

\newtheorem{theorem}{Theorem}[section]
\newtheorem{lemma}[theorem]{Lemma}
\newtheorem{proposition}[theorem]{Proposition}
\newtheorem{corollary}[theorem]{Corollary}
\newtheorem{conjecture}[theorem]{Conjecture}

\newtheorem{convention}[theorem]{Convention}

\theoremstyle{definition}
\newtheorem{definition}[theorem]{Definition}
\newtheorem{example}[theorem]{Example}

\theoremstyle{definition}
\newtheorem{remark}[theorem]{Remark}
\newtheorem*{notation}{Notation}

\newenvironment{itemlist}
  {\begin{itemize}[leftmargin=3.4em,labelsep=0.6em,itemsep=1.5pt,topsep=2pt]}%
  {\end{itemize}}

\newcommand{\triv}{{\unskip\nobreak\hfil\penalty50\hskip2em\hbox{}\nobreak\hfil$(\ast)$\parfillskip=0pt\finalhyphendemerits=0\par}}
\newcommand{\kk}{\mathbf{k}}

\newcommand{\Z}{\mathbb{Z}}
\newcommand{\Q}{\mathbb{Q}}
\newcommand{\R}{\mathbb{R}}
\newcommand{\A}{\mathbb{A}}
\newcommand{\PP}{\mathbb{P}}
\newcommand{\OO}{\mathcal{O}}
\newcommand{\half}{\frac{1}{2}}
\newcommand{\eps}{\epsilon}
\newcommand{\NE}{\overline{\mathrm{NE}}}
\newcommand{\cL}{\mathcal{L}}
\newcommand{\cN}{\mathcal{N}}
\newcommand{\cM}{\mathcal{M}}
\newcommand{\cX}{\mathcal{X}}
\newcommand{\cD}{\mathcal{D}}
\newcommand{\cE}{\mathfrak{E}}
\DeclareMathOperator{\Supp}{Supp}
\DeclareMathOperator{\Pic}{Pic}
\newcommand{\sq}{\sim_{\Q}}
\newcommand{\Aut}{\operatorname{Aut}}
\newcommand{\uX}{\underline{X}}
\newcommand{\Deps}{\Delta_{\epsilon}}
\DeclareMathOperator{\Nef}{Nef}
\DeclareMathOperator{\chr}{char}

\begin{document}

\title[Del Pezzo orbifolds]{Classification of low degree del Pezzo orbifolds II}

\author{Saptarshi Dandapat}
\address{Graduate School of Mathematics, Nagoya University, Furocho Chikusa-ku, Nagoya, 464-8601, Japan}
\email{saptarshi.dandapat.d2@math.nagoya-u.ac.jp}

\date{\today}

\setcounter{tocdepth}{1}

\begin{abstract}
    Del Pezzo orbifolds are del Pezzo surfaces equipped with a strict normal crossing divisor, called boundary, such that the log anticanonical divisor 
    is ample.   
    We classify del Pezzo orbifolds of degree less than $6$ in terms of components of the boundary, in the sense of Campana. This completes our previous
    work on classification of del Pezzo orbifolds with irreducible boundaries.     
\end{abstract}

\maketitle

\tableofcontents

\section{Introduction}\label{sec:intro}

Campana orbifolds were introduced by Campana in his work on special manifolds and on the birational classification of compact K\"ahler manifolds \cite{campana2004orbifolds, Campana2011, Campana2011Survey}. A Campana orbifold consists of a variety together with a boundary, that is a $\Q$-divisor whose coefficients are of the form $1-\frac{1}{m}$, for some integer $m$, and it interpolates between a projective variety, where the boundary is empty, and a quasi-projective variety, where the boundary is reduced. Such pairs arise naturally as bases of fibrations with multiple fibres, and Campana's decomposition of projective varieties by canonical fibrations, into pieces whose canonical class has a definite sign in this enlarged category precisely in order to take multiple fibres into account \cite{Campana2011Survey}.

Over the last few years, Campana orbifolds have moved to the centre of two closely related lines of research: in arithmetic geometry, through \emph{Campana points}, which interpolate between rational and integral points \cite{campana2005fibres, abramovich2009birational, campana2015special, pieropan2021campana}, and in the geometry of rational curves, through \emph{Campana curves} and \emph{Campana rational connectedness} \cite{Campana2011Survey, chen2026campana, dandapat2026campana}. In both settings the orbifold analogues of Fano varieties play a distinguished role, and the two-dimensional ones whose underlying surface is a del Pezzo surface : the \emph{del Pezzo orbifolds}.

This paper is devoted to the classification of del Pezzo orbifolds of low degree. In \cite{dandapat2026classification, dandapat2026campana} we classified the del Pezzo orbifolds whose boundary is irreducible. Here we treat the complementary case of \emph{reducible} boundaries, thereby completing the classification of del Pezzo orbifolds of degree at most five.

\subsection{Campana orbifolds and del Pezzo orbifolds}\label{subsec:definitions}

Throughout the paper we work over an algebraically closed field $\kk$, in this introduction we assume for simplicity that $\chr\kk=0$.

We follow the conventions of \cite{chen2026campana} and \cite{dandapat2026campana, dandapat2026classification}.

\begin{definition}[Campana orbifolds]\label{def:campana-orbifold}
Let $\uX$ be a smooth projective variety and let $\Delta=\sum_{i=1}^{r}D_i$ be a reduced strict normal crossings (SNC) divisor on $\uX$ with irreducible 
components $D_1,\dots,D_r$. We denote by $X$ the log scheme associated with the pair $(\uX,\Delta)$. To each component $\Delta_i$ we assign a 
\emph{multiplicity} $m_i\in\Z_{\ge 1}$ and the weight $\epsilon_i=1-\frac{1}{m_i}$, and we put
\[
\Deps=\sum_{i=1}^{r}\epsilon_i D_i=\sum_{i=1}^{r}\Bigl(1-\frac{1}{m_i}\Bigr)D_i.
\]
The pair $(X,\Deps)$ is called a \emph{\textup{(}klt\textup{)} Campana orbifold} \cite[Section~1.2.1]{campana2004orbifolds}, \cite[Definition~5.1]{chen2026campana}.
\end{definition}

Since every coefficient of $\Deps$ is smaller than one, the pair $(\uX,\Deps)$ has klt singularities. One may also allow $m_i=\infty$, that is 
$\epsilon_i=1$, which leads to non-klt Campana orbifolds and, in the arithmetic setting, to integrality conditions along $D_i$ 
\cite{pieropan2021campana, chen2026campana}. As in \cite{dandapat2026campana, dandapat2026classification}, all multiplicities in this paper are finite. 
We shall always assume that $m_i\ge 2$ for every $i$.

\begin{definition}\label{def:dP-orbifold}
A \emph{del Pezzo orbifold} is a Campana orbifold $(X,\Deps)$ such that $\uX$ is a smooth del Pezzo surface and the $\Q$-divisor
\[
-(K_{\uX}+\Deps)
\]
is ample. The \emph{degree} of $(X,\Deps)$ is the degree $d=K_{\uX}^{2}$ of the underlying del Pezzo surface. We say that the boundary of $(X,\Deps)$ is \emph{irreducible} if $r=1$ and \emph{reducible} if $r\ge 2$.
\end{definition}

In other words, del Pezzo orbifolds are the two-dimensional klt Fano orbifolds in the sense of \cite[Conjecture~1.4]{chen2026campana} whose underlying surface is itself a del Pezzo surface.

\subsection{The classification problem}\label{subsec:problem}

Let $\uX$ be a del Pezzo surface of degree $d\le 7$. Its cone of curves is generated by the classes of $(-1)$-curves 
(see \cite[Lemma~2.1]{dandapat2026classification} and the references therein), so by Kleiman's criterion a $\Q$-divisor on $\uX$ is ample if and only 
if it has positive degree on every $(-1)$-curve. Consequently, whether a Campana orbifold $(X,\Deps)$ on $\uX$ is a del Pezzo orbifold depends only 
on the classes of the components $D_i$ in $\Pic(\uX)$, on their intersection numbers with the $(-1)$-curves, and on the multiplicities $m_i$.

For $d\ge 6$ the surface $\uX$ is toric, and boundaries contained in the toric boundary can be studied by toric methods, see, e.g., \cite[Section~8]{chen2026campana} and \cite{pieropan2024hyperbola}. If $d\le 5$, then $\uX$ is the blow-up of $\PP^2$ in $9-d\ge 4$ points in general position: its automorphism group is finite, it contains $10$, $16$, $27$, $56$ and $240$ $(-1)$-curves for $d=5,4,3,2,1$, respectively, whose configuration is governed by the root systems $A_4$, $D_5$, $E_6$, $E_7$ and $E_8$, and for $d\le 4$ it varies in a moduli space of dimension $10-2d$ \cite{manin1986cubic, dolgachev2012classical}. In this range no group action is available to organise the possible boundaries, which therefore have to be determined directly. At the other extreme, log del Pezzo surfaces with a non-zero reduced SNC boundary (formally, $m_i=\infty$ for all $i$) were classified by Maeda \cite{maeda1986classification}, and the regime in which all $m_i$ are large was studied by Cheltsov and Rubinstein \cite{cheltsov2015asymptotically}, see \S\ref{subsec:motivation}.

In \cite{dandapat2026campana, dandapat2026classification} we classified all del Pezzo orbifolds of with irreducible boundary. Reducible boundaries behave quite differently. The multiplicities of different components interact through the $(-1)$-curves meeting several of them, which produces new arithmetic constraints on the multiplicities, moreover, the SNC condition is no longer automatic, so the answer may depend on the position of $\uX$ in its moduli space. The following example illustrates both phenomena.

\begin{example}[A triangle of lines on a cubic surface]\label{ex:triangle}
Let $\uX$ be a smooth cubic surface and let $\ell_1,\ell_2,\ell_3\subset\uX$ be three coplanar lines, 
so that $\ell_1+\ell_2+\ell_3\in|-K_{\uX}|$. Put $D_i=\ell_i$. The divisor $\ell_1+\ell_2+\ell_3$ is SNC if and only if the three lines are not concurrent, i.e.\ they do not meet in an Eckardt point of $\uX$. In this case
\[
-(K_{\uX}+\Deps)\equiv\frac{1}{m_1}\,\ell_1+\frac{1}{m_2}\,\ell_2+\frac{1}{m_3}\,\ell_3 .
\]
Every other line of $\uX$ meets exactly one of the $\ell_i$, and the right-hand side has degree $\frac{1}{m_i}>0$ on it, whereas its degree on $\ell_i$ equals $\frac{1}{m_j}+\frac{1}{m_k}-\frac{1}{m_i}$ for $\{i,j,k\}=\{1,2,3\}$. Hence $(X,\Deps)$ is a del Pezzo orbifold if and only if $\frac{1}{m_1},\frac{1}{m_2},\frac{1}{m_3}$ satisfy the strict triangle inequalities. For example, the multiplicity vector $(2,2,m)$ is admissible for every $m\ge 2$, whereas $(2,3,m)$ is admissible if and only if $m\le 5$. Note also that the set of admissible multiplicity vectors is not closed under lowering multiplicities: $(3,4,11)$ is admissible, but $(2,4,11)$ is not.
\end{example}

\subsection{Main results}\label{subsec:results}

In this paper we classify all del Pezzo orbifolds of degree at most five with reducible boundary. An elementary observation (Remark~\ref{rem:bounds}) shows that their reduced boundary has anticanonical degree at most $2d-1$. 
For each $d\in\{1, 2,3,4,5\}$ we then determine:
\begin{enumerate}
\item the curves that can occur as components of the boundary, by the bound just mentioned they have low anticanonical degree, and we describe them using the 
classification of curves of low anticanonical degree obtained in \cite{dandapat2026classification},
\item the possible configurations of the boundary, recorded by the classes of its components and their mutual intersection numbers,
\item for each configuration, the set of all choices of multiplicities $m_i\ge 2$ for which the divisor $-(K_{\uX}+\Deps)$ is ample. (See Theorem \ref{thm:deg1},
\ref{thm:main-deg2}, \ref{thm:main3}, \ref{thm:dP}, \ref{thm:main5}.)
\end{enumerate}

\begin{remark}\label{rem:bounds}
Let $(X,\Deps)$ be a del Pezzo orbifold of degree $d$. Since $-K_{\uX}$ and $-(K_{\uX}+\Deps)$ are both ample, their intersection number is positive, 
that is,
\[
\sum_{i=1}^{r}\epsilon_i\,\bigl(-K_{\uX}\cdot D_i\bigr)<d .
\]
As $\epsilon_i\ge\frac12$ and $-K_{\uX}\cdot D_i\ge 1$ for all $i$, the reduced boundary satisfies $-K_{\uX}\cdot\Delta\le 2d-1$, in particular 
$r\le 2d-1$. Hence, for fixed $d$, the pairs $(\uX,\Delta)$ underlying del Pezzo orbifolds of degree $d$ form a bounded family, although the 
multiplicities need not be bounded (Example~\ref{ex:triangle}).  
\end{remark}

\subsection{Motivation: connections with other conjectures}\label{subsec:motivation}

We now explain how the classification enters several conjectures on Fano orbifolds, both geometric and arithmetic.

\subsubsection*{Campana's rational connectedness conjecture}
Let $(X,\Deps)$ be a Campana orbifold. A \emph{Campana curve} on $(X,\Deps)$ is a morphism $f\colon C\to\uX$ from a smooth projective curve whose image is not contained in $\Delta$ and such that, for every $i$ and every point $p\in f^{-1}(\Delta_i)$, the multiplicity of $f^{*}\Delta_i$ at $p$ is at least $m_i$. Chen, Lehmann and Tanimoto \cite{chen2026campana} developed the deformation theory of Campana curves by means of log geometry and moduli stacks of stable log maps. Following \cite[Definition~1.2]{chen2026campana}, $(X,\Deps)$ is \emph{Campana rationally connected} if two general points of $\uX$ are connected by genus-zero Campana curves belonging to a family that dominates $\uX$, this notion is equivalent to Campana's orbifold rational connectedness \cite{Campana2011Survey}, see \cite[Remark~7.3]{chen2026campana}. The following conjecture is the orbifold counterpart of the rational connectedness of Fano manifolds \cite{campana1992connexite, kollar1992rational} and of klt log Fano pairs \cite{zhang2006rational, hacon2007shokurov}.

\begin{conjecture}[Campana, see {\cite[Section~9]{Campana2011Survey}}, {\cite[Conjecture~1.4]{chen2026campana}}]\label{conj:campana}
Assume that $\chr\kk=0$. Every klt Campana orbifold $(X,\Deps)$ such that $-(K_{\uX}+\Deps)$ is ample is Campana rationally connected.
\end{conjecture}

Conjecture~\ref{conj:campana} is known only in a few situations: when all multiplicities are equal to some $m$ and $\Delta$ is divisible by $m$ in 
$\Pic(\uX)$, by a cyclic covering argument due to Campana \cite[Example~1.5]{chen2026campana}, for toric Campana orbifolds with torus-invariant 
boundary \cite[Theorem~1.8]{chen2026campana} (see \cite{feng2025campana} for non-klt toric orbifolds), and in some cases admitting $\PP^1$-fibrations 
\cite[Section~7]{chen2026campana}. In \cite{dandapat2026campana} we prove Campana rational connectedness for del Pezzo orbifolds of degree less than $7$, 
amd higher degree del Pezzo orbifolds with irreducible boundaries. In general, even the existence of covering families of orbifold rational curves is unknown, which led Campana to introduce a slope-theoretic substitute for rational connectedness of orbifolds \cite{campana2016orbifold}. Reducible boundaries are a genuine source of new phenomena: already in the logarithmic limit $m_i=\infty$, the projective plane with the union of two lines as boundary is log Fano but not $\A^1$-connected \cite[Section~1.3]{chen2026campana}, cf.\ \cite{keel1999rational, ChenZhu2019}.

For del Pezzo orbifolds, Conjecture~\ref{conj:campana} we have a good strategy. Since $\uX$ is rationally connected, 
\cite[Corollary~1.7]{chen2026campana} shows that $(X,\Deps)$ is Campana rationally connected as soon as it is \emph{strongly Campana uniruled}, 
that is, as soon as it carries a free Campana curve whose class lies in the interior of the cone $\Nef_1(\uX)$ of nef curve classes 
\cite[Definition~5.9]{chen2026campana}. Constructing such curves requires precise knowledge of the boundary and its components, their configuration 
and the multiplicities. Moreover, a Campana curve for $(X,\Deps)$ remains a Campana curve after any of the multiplicities is lowered, so 
Campana rational connectedness passes to every orbifold structure on the same boundary with smaller multiplicities, even to those that are not 
del Pezzo (cf.\ Example~\ref{ex:triangle}). For a fixed configuration it therefore suffices to treat the multiplicity vectors that are maximal 
among the admissible ones, whenever such maximal vectors exist.

\subsubsection*{Weak approximation for Campana sections and higher-dimensional orbifolds}
Over the function field of a complex curve $B$, Chen, Lehmann and Tanimoto proved a Campana analogue of the theorems of Graber-Harris-Starr and 
Hassett-Tschinkel \cite{graber2003families, hassett2005}: for a klt Campana fibration over $B$ whose general fibre is rationally connected and 
strongly Campana uniruled, every finite collection of Campana jets supported in distinct fibres of good reduction is induced by a Campana section 
\cite[Theorem~1.6]{chen2026campana}. Fibrations whose general fibre is a del Pezzo orbifold, for instance, families of del Pezzo surfaces of 
degree at most five equipped with a SNC divisor with standard coefficients are natural next cases after the toric one 
\cite[Corollary~1.10]{chen2026campana}, and the classification determines precisely which orbifold structures can occur on their general fibres.
\cite{dandapat2026campana} proves weak approximation for Campana sections at all places for del Pezzo orbifolds of degree $3 \leq d \leq 7$, and higher
degree del Pezzo orbifolds with irreducible boundaries. 
More generally, one approach to Conjecture~\ref{conj:campana} in higher dimensions is to run the minimal model program for the pair $(\uX,\Deps)$, 
which produces Mori fibre spaces whose general fibres are orbifold Fano pairs of lower dimension. Whenever such a fibre is a smooth del Pezzo surface 
on which the boundary restricts to an SNC divisor, it is a del Pezzo orbifold in the sense of Definition~\ref{def:dP-orbifold}.

\subsubsection*{Campana points of bounded height}
Let $F$ be a number field, let $S$ be a finite set of places of $F$ containing the archimedean places, and let $(\cX,\cD_\epsilon)$ be a model of 
$(X,\Deps)$ over the ring of $S$-integers $\OO_{F,S}$. A rational point of $\uX$ outside $\Delta$ is a \emph{Campana point} if, at every place 
$v\notin S$, its local intersection multiplicity with each $\cD_i$ is either zero or at least $m_i$ \cite[Definition~3.4]{pieropan2021campana}, see 
also \cite{campana2005fibres, abramovich2009birational, abramovich2018campana, campana2015special}. Pieropan, Smeets, Tanimoto and V\'arilly-Alvarado 
formulated a Manin-type conjecture for klt Campana points on Fano orbifolds \cite[Conjecture~1.1]{pieropan2021campana}: for the height associated with 
$-(K_{\uX}+\Deps)$ and under suitable hypotheses, it predicts that, after the removal of a thin set, the number of Campana points of height at most 
$T$ grows like $cT(\log T)^{b-1}$, where $b$ is an explicit geometric invariant. The conjecture has been established for equivariant compactifications 
of vector groups \cite{pieropan2021campana}, for toric varieties \cite{pieropan2024hyperbola, shute2024semi} and for biequivariant compactifications 
of the Heisenberg group \cite{xiao2022campana}, a variant of it has been proved for wonderful compactifications of semisimple groups of adjoint type 
\cite{chow2026campana}, and it is supported by results on hyperplanes and norm forms 
\cite{browning2012sums, van2012squareful, browning2021arithmetic, streeter2022campana}. At the same time, the leading constant and the exceptional 
thin set are subtle \cite{shute2022leading, chow2026campana}. Many of these results rely on the action of an algebraic group with a dense orbit, 
whereas a del Pezzo surface of degree at most five has a finite automorphism group. Even for rational points, Manin's conjecture for smooth del Pezzo 
surfaces of degree at most five is known only in special cases (see, e.g., \cite{dlB02, Browning2008ManinsCF}) and remains open for smooth cubic 
surfaces. Del Pezzo orbifolds of low degree are therefore natural next test cases. Our classification determines exactly for which boundaries and 
multiplicities the log-anticanonical version of the conjecture applies to a del Pezzo surface of degree at most five, and it is the starting point 
for identifying the relevant exceptional sets (cf.\ \cite{lehmann2022geometric} for rational points), such as the contribution of Campana points 
lying on $(-1)$-curves, which depends on how these curves meet the boundary. Del Pezzo orbifolds also provide natural two-dimensional examples of 
Fano Deligne-Mumford stacks, for which Batyrev-Manin-type conjectures \cite{ellenberg2023heights, darda2024batyrev} and a theory of Brauer groups 
\cite{loughran2024malle} have recently been developed, see also \cite{mitankin2024semi} for the closely related Darmon points.

\subsubsection*{Campana curves over function fields}
Over global function fields $\mathbb{F}_q(t)$, Campana points correspond to Campana sections, and counting them is closely related to the geometry 
of moduli spaces of Campana curves. Manin's conjecture for Campana rational curves has recently been established for split toric varieties 
\cite{chen2026maninsconjecturesemiintegralcurves} (see \cite{faisant2025motivic} for a motivic version), while, for rational curves without boundary 
conditions, the homological sieve of \cite{das2025homological} yielded a version of Manin's conjecture for rational curves on quartic del Pezzo surfaces, 
see \cite{lehmann2019geometric, lehmann2022geometric} for Geometric Manin's Conjecture and \cite{testa2009, Tanimoto2025} for the spaces of rational 
curves on del Pezzo surfaces. The del Pezzo orbifolds of degree four classified here and in \cite{dandapat2026classification} are natural non-toric 
test cases for extending these methods to Campana curves. \cite{lehmann2026fano} proves that Fano orbifolds admit free Campana curves.

\subsubsection*{Local-global principles and special orbifolds}
Qualitative arithmetic questions are equally open. Campana conjectured that an orbifold is special if and only if its Campana points are potentially 
dense over number fields, and if and only if its orbifold Kobayashi pseudometric vanishes identically 
\cite{campana2004orbifolds, Campana2011, campana2015special}, see also \cite{rousseau2021nonspecial}. Del Pezzo orbifolds lie at the positive extreme 
of this set up, and their Campana points are expected to be abundant. For the Manin-type conjecture above we need that the Campana points are not thin. 
Nakahara and Streeter showed that Campana weak weak approximation implies this Campana Hilbert property, and exhibited Campana orbifolds whose Campana 
points are not thin \cite{nakahara2024weak}. Mitankin, Nakahara and Streeter developed a semi-integral version of the Brauer-Manin obstruction and 
determined the status of local-global principles for two families of orbifolds associated with quadrics \cite{mitankin2024semi}, see also 
\cite{moerman2024generalized} for toric varieties and \cite{loughran2024malle} for Brauer groups of stacks. For rational points on del Pezzo surfaces, 
the Hasse principle and weak approximation always hold in degree at least five, but may fail in degree at most four, and Colliot-Th\'el\`ene conjectured 
that all such failures are explained by the Brauer-Manin obstruction (see \cite{colliot2003points, colliot2021brauer}). It is 
therefore natural to ask to what extent the semi-integral Brauer-Manin obstruction controls the Hasse principle and weak approximation for Campana 
points on del Pezzo orbifolds. The del Pezzo orbifolds of degree at most four classified here and in \cite{dandapat2026classification} provide a 
non-toric test case.

\subsection*{Conventions}
A \emph{del Pezzo surface} is a smooth projective surface $\uX$ with ample anticanonical divisor, its degree is $d=K_{\uX}^{2}$. A \emph{$(-1)$-curve} 
is a smooth rational curve with self-intersection $-1$, and a \emph{conic} is a smooth rational curve with self-intersection $0$, that is, a smooth 
fibre of a conic fibration $\uX\to\PP^1$. The \emph{anticanonical degree} of a curve $C\subset\uX$ is $-K_{\uX}\cdot C$. When $d\le 7$ we fix a 
birational morphism $\beta\colon\uX\to\PP^2$ realising $\uX$ as the blow-up of $\PP^2$ in $9-d$ points in general position, and write $H$ for the 
pullback of the class of a line and $E_1,\dots,E_{9-d}$ for the exceptional curves, so that $-K_{\uX}=3H-\sum_i E_i$. For a Campana orbifold 
$(X,\Deps)$ we write $\Delta=\sum_i D_i$ for the reduced boundary and $\epsilon_i=1-\frac{1}{m_i}, m_i \in \Z_{\geq 2}$.

\subsection*{Acknowledgements}

The author would like to thank his supervisor Prof. Sho Tanimoto for his enthusiastic comments on this project and sharing the problem in 2025. 
He would also like to thank Qile Chen, Brian Lehmann and Sho Tanimoto for their comments and sharing their unpublished works on classification of del Pezzo orbifolds of 
degree $3$. The author was supported by MEXT Scholarship for Doctoral studies from the Government of Japan, recommended by the Embassy of Japan in New 
Delhi, India. Claude Opus 5 and 5.5 was used to generate Python scripts to crosscheck the results in this paper, to organise the representation of the 
tables used in this paper, and to proofread the document.  However, the idea of writing the programs is due to the author. The paper itself was written 
by the author who is responsible for correctness.   

\section{Surfaces of degree 1}
\label{sec:dp1}

Let $S$ be a del Pezzo surface of degree $1$ and $\Delta_{\epsilon}$ be a $\Q$-divisor on it. We are interested in the case where $-(K_{S} + \Delta_{\epsilon})$ is ample. In this case, $\Delta_{\epsilon}$ can have at most $2$ components and each component has an anti-canonical degree at most $2$. 

\subsection{The possible components.} 
We record the intersection numbers with $(-1)$-curves. Let $D'$ be an irreducible curve which is a component of $\Delta$.
\begin{itemize}
    
    \item If $D'$ is a $(-1)$-curve. Then $D'^2 = -1$, $D' \cdot E = 0$ for all disjoint $(-1)$-curves (there are 56 such curves), $D' \cdot E = 1$ for all $(-1)$-curves which it meets at one point (there are 126 such curves), $D' \cdot E = 2$ for all $(-1)$-curves which it meets twice (there are 56 such curves), and there is a unique $(-1)$-curve such that $D' \cdot E = 3$.

    \item If $D'$ is a member of the anti-canonical class. Then $D' \cdot E = 1$ for any $(-1)$-curve $E$.

    \item If $D'$ is a smooth fiber of a conic fibration. Then there are 14 $(-1)$-curves $E$ such that $D' \cdot E = 0$, there are 64 $(-1)$-curves $E$ such that $D' \cdot E = 1$, there are 84 $(-1)$-curves $E$ such that $D' \cdot E = 2$, there are 64 $(-1)$-curves such that $D' \cdot E = 3$, and there are 14 $(-1)$-curves such that $D' \cdot E = 4$.

    \item If $D'$ is pullback of a member in the anti-canonical class of a degree 2 del Pezzo surface via a birational morphism. Then $D' \cdot E = 0$ for the contracted $(-1)$-curve. There are 56 $(-1)$-curves disjoint from the contracted curve such that $D' \cdot E = 1$, there are 126 $(-1)$-curves that meet the contracted curve at one point such that $D' \cdot E = 2$, there are 56 $(-1)$-curves that meet the contracted curve twice such that $D' \cdot E = 3$, and there is an unique $(-1)$-curve $E$ which meet the contracted curve thrice such that $D' \cdot E = 4$.

    \item If $D'$ which is member of $\lvert -2K_{S}\rvert$. Then $D' \cdot E = 2$ for all $(-1)$-curves $E$.

\end{itemize}

Now we are going to classify when $-(K_{S} + \Delta_{\epsilon})$ is nef.

\begin{theorem}
\label{thm:deg1}
    Let $S$ be a del Pezzo surface of degree 1 and $\Delta_{\epsilon}$ be a $\Q$-divisor on it. Then $-(K_{S} + \Delta_{\epsilon})$ is nef if and only if $\Delta_{\epsilon} = \frac{1}{2}D_1 + \frac{1}{2}D_2$, where $D_1$ and $D_2$ both are member of the anti-canonical class $\lvert -K_{S}\rvert$.
\end{theorem}

\begin{proof}
We have the following cases.
    \begin{itemize}
        \item \textbf{Case 1.} Let $D_1$ be a component that is a member of the anti-canonical class.  
        
        Suppose we have another component $D_2$ which is also a member of the anti-canonical class. Looking at the numerology, the only option is $\frac{1}{2}D_1 + \frac{1}{2}D_2$. 
    
        Suppose we have another component $D_2$ which is either a $(-1)$-curve, or a smooth fiber of a conic fibration, or is a member of $\lvert -2K_{S}\rvert$, then there is at least one $(-1)$-curve $E$ such that $D_2 \cdot E = 2$. Hence in these cases $-(K_{S} + \Delta_{\epsilon})$ is not nef.
    
        Suppose we have another component $D_2$ which is pullback of a member in the anti-canonical class of a del Pezzo surface of degree 2. Then there are more than one $(-1)$-curves such that $D_2 \cdot E \geq 2$. Hence in this case $-(K_{S} + \Delta_{\epsilon})$ is not nef.   
    
        \item \textbf{Case 2.} Let $D_1$ be a component that is a $(-1)$-curve. 
    
        If $D_2$ is another component which is either a $(-1$-curve, or a smooth fiber of a conic fibration, or pullback of a member in the anti-canonical class of a del Pezzo surface of degree 2, or is a member of $\lvert-2K_{S}\rvert$, there are more than one $(-1)$-curve such that $D_2 \cdot E \geq 2$. Hence $-(K_{S} + \Delta_{\epsilon})$ is not nef. 
        
        \item \textbf{Case 3.} Let $D_1$ is a component which is smooth fiber of a conic fibration. Since there are 78 $(-1)$-curves such that $D_1 \cdot E \geq 3$, this case is not possible for any $D_2$.
        
        \item \textbf{Case 4.} Let $D_1$ be a component which is pullback of a member in the anti-canonical class of a del Pezzo surface of degree $2$. Since there are 57 $(-1)$-curves such that $D_1 \cdot E \geq 3$, this case is also not possible for any $D_2$. 
        
        \item \textbf{Case 5.} Let $D_1$ be a component which is a member of $\lvert-2K_{S}\rvert$. 
        
        If $D_2$ is any of the above type of curve, then there exists at least one $(-1)$-curve $E$ such that $D_2 \cdot E \geq 1$. So this case is also not possible.
    \end{itemize}
\end{proof}

\section{Surfaces of degree 2}
\label{sec:dp2}

Let $S$ be a smooth del Pezzo surface with $K_S^2=2$ and 
$$\Delta_{\epsilon}\,=\,\sum_{i=1}^{n}\epsilon_i D_i, \qquad \epsilon_i=1-\frac{1}{m_i},\quad m_i\in\mathbb Z_{\ge 2},$$
is a SNC $\Q$-divisor: the $D_i$ are pairwise distinct irreducible curves on $S$. In particular $\tfrac12\le \epsilon_i<1$ for every $i$.
We call $\Delta_\epsilon$ \emph{reducible} if $n\ge 2$. We write
$$\mathcal L=\{\,E\subset S \,:\, E \text{ is a } (-1)\text{-curve}\,\},$$
and $d_i:=-K_S\cdot D_i$, $\deg\Delta_\epsilon:=-K_S\cdot\Delta_\epsilon=\sum_i\epsilon_id_i$. 

\begin{lemma}\label{lem:nef-test}
    $-(K_S+\Delta_\epsilon)$ is nef if and only if
    $$\Delta_\epsilon\cdot E\,\le\,1\qquad\text{for every }E\in\mathcal L.$$
\end{lemma}
 
\begin{proof}
    Since $K_S^2\le 7$, the Mori cone $\overline{\mathrm{NE}}(S)$ is generated by the classes of the $(-1)$-curves.  Hence $-(K_S+\Delta_\epsilon)$ is 
    nef if and only if $(-K_S-\Delta_\epsilon)\cdot E\ge 0$ for all $E\in\mathcal L$, and $-K_S\cdot E=1$.
\end{proof}

We are interested in the case when 
$-(K_S + \Delta_{\epsilon})$ is ample. 

\begin{lemma}[Geiser conjugation]\label{lem:geiser}
For $E\in\mathcal L$ set $\overline E:=-K_S-E$.  Then:
\begin{enumerate}
\item $\overline E\in\mathcal L$, $\overline{\overline E}=E$ and $E\cdot\overline E=2$.
      Thus $\mathcal L$ is partitioned into $28$ pairs $\{E,\overline E\}$, namely the
      orbits of the Geiser involution, and $E+\overline E\sim -K_S$.
\item For distinct $E,E'\in\mathcal L$ one has $E\cdot E'\in\{0,1,2\}$, and
      $E\cdot E'=2$ if and only if $E'=\overline E$.
\item For $E'\in\mathcal L$, $E'\neq E$, one has $E'\cdot\overline E=1-E\cdot E'$.
\end{enumerate}
\end{lemma}

\begin{proof}
(1) $\overline E{}^2=(-K_S-E)^2=K_S^2+2K_S\cdot E+E^2=2-2-1=-1$ and
$-K_S\cdot\overline E=2-1=1$.  On a del Pezzo surface every class $D$ with
$D^2=D\cdot K_S=-1$ is represented by a unique irreducible $(-1)$-curve, so
$\overline E\in\mathcal L$. The remaining assertions are immediate.

(2) $E\cdot E'\ge 0$ because the curves are distinct and irreducible.  Applying the
Hodge index theorem to the ample class $-K_S$ and to $E+E'$ gives
$(E+E')^2\cdot K_S^2\le\bigl(K_S\cdot(E+E')\bigr)^2$, i.e.
$2(-2+2E\cdot E')\le 4$, whence $E\cdot E'\le 2$, equality forces
$E+E'\equiv -K_S$, i.e. $E'=\overline E$.

(3) $E'\cdot\overline E=E'\cdot(-K_S-E)=1-E\cdot E'$.
\end{proof}

\begin{lemma}\label{lem:degree}
Assume $-(K_S+\Delta_\epsilon)$ is nef.  Then
\[
  \sum_{i}\epsilon_i d_i\,\le\,2 ,
\]
hence $\sum_i d_i\le 4$ and $n\le 4$, if moreover $n\ge 2$ then every component
satisfies $d_i\le 3$.  Furthermore the following are equivalent:
\begin{enumerate}
\item $\sum_i\epsilon_i d_i=2$,
\item $\Delta_\epsilon\sim -K_S$,
\item $-(K_S+\Delta_\epsilon)\sim 0$.
\end{enumerate}
\end{lemma}

\begin{proof}
Put $N:=-(K_S+\Delta_\epsilon)$.  As $N$ is nef and $-K_S$ is ample,
$0\le N\cdot(-K_S)=2-\sum_i\epsilon_id_i$.  Since $\epsilon_i\ge\half$ and
$d_i\ge1$ this gives $\half\sum_id_i\le 2$, so $\sum_id_i\le4$ and $n\le4$,
if $n\ge2$ then $d_i\le 4-(n-1)\le 3$.

If $N\cdot(-K_S)=0$ then $N^2=0$: otherwise $N$ would be nef and big, forcing
$N\cdot(-K_S)>0$.  Now $N^2=0$ and $N\cdot(-K_S)=0$ give $N\equiv 0$ by the Hodge
index theorem, and $N\sim 0$ because numerical and linear equivalence agree on a
rational surface.  The converse implications are clear.
\end{proof}

\begin{lemma}\label{lem:key}
Assume $-(K_S+\Delta_\epsilon)$ is nef and let $E\in\mathcal L$ be a component of
$\Delta_\epsilon$, with coefficient $\epsilon_E$.  If $\overline E$ is \emph{not}
a component of $\Delta_\epsilon$, then
\[
  \epsilon_E=\tfrac12
  \qquad\text{and}\qquad
  D\cdot\overline E=0 \quad\text{for every other component } D \text{ of } \Delta_\epsilon .
\]
\end{lemma}

\begin{proof}
By Lemma \ref{lem:nef-test},
\[
  1\,\ge\,\Delta_\epsilon\cdot\overline E
   \,=\,\epsilon_E\,(E\cdot\overline E)+\sum_{D\neq E}\epsilon_D\,(D\cdot\overline E)
   \,=\,2\epsilon_E+\sum_{D\neq E}\epsilon_D\,(D\cdot\overline E).
\]
Since $\overline E$ is not a component, every $D\neq E$ is an irreducible curve
different from $\overline E$, so $D\cdot\overline E\ge0$.  Hence $2\epsilon_E\le1$,
which forces $\epsilon_E=\half$, and then the sum must vanish termwise.
\end{proof}

\subsection{The possible components}

We recall from \cite{dandapat2026classification} the irreducible curves of small anticanonical degree.

\begin{lemma}\label{lem:components}
Let $D\subset S$ be an irreducible curve and $d=-K_S\cdot D$.
\begin{enumerate}
\item $d=1$ if and only if $D\in\mathcal L$, then $D^2=-1$.
\item $d=2$: then $D^2\in\{0,2\}$.  If $D^2=0$, then $D$ is a fibre of a conic
      fibration, if $D^2=2$, then $D\in|-K_S|$.
\item $d=3$: then $D^2\in\{1,3\}$.  If $D^2=1$, then $D=\pi^{*}\ell$ for a
      birational morphism $\pi\colon S\to\mathbb P^2$ and a line $\ell$,
      if $D^2=3$, then $D\sim -K_S+E_0=\pi^{*}(-K_{S'})$, where
      $\pi\colon S\to S'$ contracts a single $E_0\in\mathcal L$ onto a smooth
      cubic surface $S'$.
\end{enumerate}
\end{lemma}

For later use we record the intersection numbers with $(-1)$-curves.  Let
$D'$ be an irreducible curve.
\begin{itemize}
    \item If $D'\in\mathcal L$, then $D'^2=-1$, $D'\cdot E=0$ for $27$ curves,
        $D'\cdot E=1$ for $27$ curves, and $D'\cdot E=2$ for the unique curve
        $E=\overline{D'}$.
    \item If $D'$ is a fibre of a conic fibration, then $D'\cdot E=0$ for $12$
        curves, $D'\cdot E=1$ for $32$ curves and $D'\cdot E=2$ for $12$ curves.
    \item If $D'\in|-K_S|$, then $D'\cdot E=1$ for every $E\in\mathcal L$.
    \item If $D'=\pi^{*}\ell$ is the pullback of a line from $\mathbb P^2$, then
        $D'\cdot E=0,1,2,3$ for respectively $7,21,21,7$ curves.
    \item If $D'=\pi^{*}(-K_{S'})$ is the pullback of an anticanonical divisor from a
        smooth cubic surface, contracting $E_0\in\mathcal L$, then
        $D'\cdot E_0=0$, $D'\cdot E=1$ for the $27$ curves disjoint from $E_0$,
        $D'\cdot E=2$ for the $27$ curves meeting $E_0$ once, and $D'\cdot E=3$ for
        the unique curve $\overline{E_0}$.
\end{itemize}

\begin{remark}\label{rem:deg4}
    Irreducible curves of anticanonical degree $4$ (pullbacks of $(1,1)$ from
    $\mathbb P^1\times\mathbb P^1$, pullbacks of anticanonical divisor from a quartic del Pezzo
    surface, curves $\sim -K_S+F$ with $F$ a conic, and members of $|-2K_S|$) cannot
    occur as components of a \emph{reducible} boundary, by Lemma \ref{lem:degree}.
    Of these only a member $B\in|-2K_S|$ occurs at all: $\Delta_\epsilon=\half B$
    gives $\Delta_\epsilon\sim -K_S$, and this boundary is irreducible.  Each of the
    other three classes $D$ satisfies $D\cdot E\ge 3$ for at least two $E\in\mathcal L$,
    so $\half D\cdot E\ge\frac32>1$ and $D$ never occurs.
\end{remark}

\subsection{Classification}

\begin{theorem}\label{thm:main-deg2}
    Let $S$ be a smooth del Pezzo surface of degree $2$ and let
    $\Delta_\epsilon=\sum_{i=1}^n\epsilon_iD_i$ be a Campana boundary with
    $n\ge 2$.  Then $-(K_S+\Delta_\epsilon)$ is nef if and only if $\Delta_\epsilon$
    appears in the following list.
    \begin{enumerate}
    \item[\rm(A)] $\Delta_\epsilon=\epsilon_1E+\epsilon_2\overline E$ for some
    $E\in\mathcal L$, with
    \[
        2\epsilon_1-\epsilon_2\le 1
        \quad\text{and}\quad
        2\epsilon_2-\epsilon_1\le 1,
        \qquad\text{equivalently}\qquad
        \tfrac12\,\le\,\tfrac{m_1}{m_2}\,\le\,2 .
    \]
    \item[\rm(B)] $\Delta_\epsilon=\tfrac12(E_1+E_2)$ with $E_1,E_2\in\mathcal L$ and
    $E_1\cdot E_2=1$.
    \item[\rm(C)] $\Delta_\epsilon=\tfrac12 B$, where $B=D_1+\dots+D_n\in|-2K_S|$ is
    reduced with $n\ge2$ irreducible components.  Explicitly, $B$ is of exactly one
    of the following seven forms, indexed by the multiset $(d_1,\dots,d_n)$ (where $d_i$ is anticanonical degree of $D_i$, for all $i$):
    \begin{enumerate}
    \item[\rm(C1)] $(3,1)$: $B=D+E$ with $E\in\mathcal L$ and
        $D\sim -2K_S-E=-K_S+\overline E$, i.e. $D=\pi^{*}(-K_{S'})$ for the
        contraction $\pi\colon S\to S'$ of $\overline E$ onto a cubic surface.
    \item[\rm(C2)] $(2,2)$: $B=A_1+A_2$ with $A_1\neq A_2$ irreducible members of
        $|-K_S|$.
    \item[\rm(C3)] $(2,2)$: $B=F_1+F_2$ with $F_1,F_2$ irreducible conics such that
        $F_1+F_2\sim -2K_S$ (then $F_1\cdot F_2=4$).
    \item[\rm(C4)] $(2,1,1)$: $B=E_1+E_2+F$ with $E_1\cdot E_2=1$ and $F$ an
        irreducible conic in $|-2K_S-E_1-E_2|$.
    \item[\rm(C5)] $(2,1,1)$: $B=E+\overline E+A$ with $E\in\mathcal L$ and $A$ an
        irreducible member of $|-K_S|$.
    \item[\rm(C6)] $(1,1,1,1)$: $B=E_1+E_2+E_3+E_4$ with $E_i\cdot E_j=1$ for all
        $i<j$.
    \item[\rm(C7)] $(1,1,1,1)$: $B=E_1+\overline{E_1}+E_2+\overline{E_2}$ for two
        distinct Geiser pairs.
    \end{enumerate}
    \end{enumerate}
    All seven forms in {\rm(C)} occur on every smooth del Pezzo surface of degree $2$.
\end{theorem}

\begin{proof}
    By Lemma \ref{lem:degree} we have $\sum_i d_i\le 4$ and $d_i\le 3$ for all $i$,
    so by Lemma \ref{lem:components} each component is a $(-1)$-curve, a conic, a
    member of $|-K_S|$, a pullback $\pi^*\ell$ of a line, or a pullback
    $\pi^*(-K_{S'})$ from a cubic surface.  We argue according to $\sum_id_i\in\{2,3,4\}$.

    \medskip
    \noindent\textbf{Step 1: $\sum_i d_i=2$.}
    Then $n=2$ and $\Delta_\epsilon=\epsilon_1E_1+\epsilon_2E_2$ with
    $E_1,E_2\in\mathcal L$ distinct.

    Suppose first $E_2=\overline{E_1}$.  For $E\in\mathcal L\setminus\{E_1,E_2\}$ we
    have $E\cdot E_1+E\cdot E_2=E\cdot(-K_S)=1$ with both summands $\ge0$, so
    $\Delta_\epsilon\cdot E\in\{\epsilon_1,\epsilon_2\}\subset[\tfrac12,1)$.  The only
    remaining conditions are
    \[
    \Delta_\epsilon\cdot E_1=-\epsilon_1+2\epsilon_2\le 1,
    \qquad
    \Delta_\epsilon\cdot E_2=-\epsilon_2+2\epsilon_1\le 1 .
    \]
    Substituting $\epsilon_i=1-1/m_i$ turns these into $m_2\le 2m_1$ and
    $m_1\le 2m_2$.  This is case (A).

    Suppose now $E_2\neq\overline{E_1}$, then neither $\overline{E_1}$ nor
    $\overline{E_2}$ is a component.  Lemma \ref{lem:key} applied to $E_1$ gives
    $\epsilon_1=\half$ and $E_2\cdot\overline{E_1}=0$, i.e. $E_1\cdot E_2=1$ by
    Lemma \ref{lem:geiser}(3), applied to $E_2$ it gives $\epsilon_2=\half$.
    Conversely, assume $E_1\cdot E_2=1$ and put $\Delta_\epsilon=\half(E_1+E_2)$.
    For $E\in\mathcal L$ we claim $(E_1+E_2)\cdot E\le2$.  Indeed, if $E=E_i$ then
    $(E_1+E_2)\cdot E=-1+1=0$, if $E=\overline{E_1}$ then
    $(E_1+E_2)\cdot E=2+(1-E_1\cdot E_2)=2$, and symmetrically for
    $E=\overline{E_2}$, in all other cases $E_i\cdot E\le 1$ for $i=1,2$ by
    Lemma \ref{lem:geiser}(2).  Hence $\Delta_\epsilon\cdot E\le 1$ always and
    $-(K_S+\Delta_\epsilon)$ is nef.  This is case (B).

    \medskip
    \noindent\textbf{Step 2: $\sum_i d_i=3$.}
    We show this is impossible.  There are two forms: $(d_1,d_2)=(1,2)$ and
    $(d_1,d_2,d_3)=(1,1,1)$.

    \emph{Shape $(1,2)$.}  Write $\Delta_\epsilon=\epsilon_1E+\epsilon_2D$ with
    $E\in\mathcal L$ and $-K_S\cdot D=2$.  As $\overline E$ has degree $1\neq 2$ it
    is not a component, so Lemma \ref{lem:key} gives $\epsilon_1=\half$ and
    $D\cdot\overline E=0$.

    If $D\in|-K_S|$ then $D\cdot\overline E=1\neq0$, a contradiction.

    If $D=F$ is a conic, set $F':=-2K_S-F$.  Then $F'^2=8+4K_S\cdot F+F^2=0$ and
    $-K_S\cdot F'=2$, and $F'$ is nef: it is effective by Riemann-Roch, and
    $-K_S-\half F=\half F'$ is nef because $F\cdot E\le 2$ for all $E\in\mathcal L$.
    Now
    \[
    0\,\le\,-(K_S+\Delta_\epsilon)\cdot F'
    \,=\,\bigl(\tfrac12F'-\tfrac12E-(\epsilon_2-\tfrac12)F\bigr)\cdot F'
    \,=\,-\tfrac12\,E\cdot F'-4\bigl(\epsilon_2-\tfrac12\bigr),
    \]
    using $F\cdot F'=F\cdot(-2K_S-F)=4$.  Since $F'$ is nef and $\epsilon_2\ge\half$,
    both terms are $\le0$, so $\epsilon_2=\half$ and $E\cdot F'=0$.  Thus $E$ is
    contained in a fibre of the conic fibration $|F'|$, being a $(-1)$-curve it is a
    component of a reducible fibre, say $F'\sim E+E''$ with $E''\in\mathcal L$,
    $E\cdot E''=1$ and $F'\cdot E''=0$.  But then
    \[
    -(K_S+\Delta_\epsilon)\cdot E''
    =\bigl(\tfrac12F'-\tfrac12E\bigr)\cdot E''
    =0-\tfrac12<0,
    \]
    contradicting nefness.

    \emph{Shape $(1,1,1)$.}  Write $\Delta_\epsilon=\sum_{i=1}^{3}\epsilon_iE_i$ with
    $E_i\in\mathcal L$ pairwise distinct.

    Assume first that $E_2=\overline{E_1}$.  Then $\overline{E_3}\notin\{E_1,E_2,E_3\}$:
    indeed $\overline{E_3}=E_1$ would give $E_3=\overline{E_1}=E_2$, and
    $\overline{E_3}=E_2$ would give $E_3=E_1$.  Lemma \ref{lem:key} applied to $E_3$
    yields $E_1\cdot\overline{E_3}=E_2\cdot\overline{E_3}=0$, whereas
    $(E_1+E_2)\cdot\overline{E_3}=-K_S\cdot\overline{E_3}=1$, a contradiction.

    Hence no $\overline{E_i}$ is a component.  Lemma \ref{lem:key} applied to each
    $E_i$ gives $\epsilon_i=\half$ and $E_j\cdot\overline{E_i}=0$ for $j\ne i$,
    i.e. $E_i\cdot E_j=1$ for all $i\ne j$.  Put
    $E_0:=E_1+E_2+E_3+K_S$.  Then
    \[
    E_0^2=3+2K_S\cdot(E_1+E_2+E_3)+K_S^2=3-6+2=-1,
    \qquad -K_S\cdot E_0=3-2=1,
    \]
    so $E_0\in\mathcal L$.  Moreover $E_i\cdot E_0=E_i\cdot(E_1+E_2+E_3)+E_i\cdot K_S=0$,
    hence $E_i\cdot\overline{E_0}=1$ and in particular
    $\overline{E_0}\notin\{E_1,E_2,E_3\}$.  Finally
    \[
    \Delta_\epsilon\cdot\overline{E_0}
    =\tfrac12(-K_S+E_0)\cdot(-K_S-E_0)
    =\tfrac12\bigl(K_S^2-E_0^2\bigr)=\tfrac32>1,
    \]
    contradicting Lemma \ref{lem:nef-test}.

    \medskip
    \noindent\textbf{Step 3: $\sum_i d_i=4$.}
    Since $\epsilon_i\ge\half$ and $\sum_i\epsilon_id_i\le2$, all $\epsilon_i=\half$
    and $\sum_i\epsilon_id_i=2$.  By Lemma \ref{lem:degree},
    $\Delta_\epsilon\sim-K_S$ and $-(K_S+\Delta_\epsilon)\sim0$, conversely any
    reduced $B=\sum_iD_i\in|-2K_S|$ with $n\ge2$ components yields
    $\Delta_\epsilon=\half B\sim-K_S$, hence a nef (indeed trivial)
    $-(K_S+\Delta_\epsilon)$.  So it remains to enumerate the possible forms of $B$,
    according to the partition of $4=\sum_id_i$.

    \emph{$(3,1)$.}  Here $B=D+E$ with $E\in\mathcal L$ and $D\sim -2K_S-E$, so
    \[
    D^2=(-2K_S-E)^2=8+4K_S\cdot E+E^2=8-4-1=3 .
    \]
    By Lemma \ref{lem:components}(3), $D=\pi^{*}(-K_{S'})=-K_S+E_0$ with
    $E_0=D+K_S=-K_S-E=\overline E$.  This is (C1), note that the case $D^2=1$
    (pullback of a line) is excluded, which disposes of that class of components
    entirely.  Conversely, for any $E\in\mathcal L$ the linear system
    $|-K_S+\overline E|=|\pi^{*}(-K_{S'})|$ contains irreducible members (pull back a
    general member of $|-K_{S'}|$), so (C1) occurs.

    \emph{$(2,2)$.}  Here $D_2\sim-2K_S-D_1$ and $D_1^2\in\{0,2\}$ by
    Lemma \ref{lem:components}(2).  If $D_1^2=2$ then $D_1\sim-K_S$ and hence
    $D_2\sim-K_S$: this is (C2), which occurs because $|-K_S|$ is a net whose general
    member is irreducible.  If $D_1^2=0$ then
    $D_2^2=8+4K_S\cdot D_1+D_1^2=0$, so both are conics with
    $D_1+D_2\sim-2K_S$ and $D_1\cdot D_2=D_1\cdot(-2K_S-D_1)=4$: this is (C3), which
    occurs since for any conic class $F$ the class $-2K_S-F$ is again a conic class
    and both pencils have irreducible general members.  A mixed pair is impossible,
    since $D_1\sim-K_S$ forces $D_2\sim-K_S$.

    \emph{$(2,1,1)$.}  Here $B=E_1+E_2+D$ with $E_i\in\mathcal L$ and
    $D\sim-2K_S-E_1-E_2$, so
    \[
    D^2=8+4K_S\cdot(E_1+E_2)+(E_1+E_2)^2=-2+2\,E_1\cdot E_2 .
    \]
    As $D$ is irreducible of degree $2$ we need $D^2\in\{0,2\}$, i.e.
    $E_1\cdot E_2\in\{1,2\}$, the value $E_1\cdot E_2=0$ is excluded.  If
    $E_1\cdot E_2=1$ then $D$ is a conic, giving (C4), if $E_1\cdot E_2=2$, i.e.
    $E_2=\overline{E_1}$, then $D\sim-2K_S-(E_1+\overline{E_1})=-K_S$, giving (C5).
    Both occur: for (C4) take any two $(-1)$-curves meeting once and a general member
    of the conic pencil $|-2K_S-E_1-E_2|$, for (C5) take any Geiser pair and a general
    member of $|-K_S|$.

    \emph{$(1,1,1,1)$.}  Here $B=E_1+E_2+E_3+E_4$ with $\sum_iE_i\sim-2K_S$.
    Intersecting with $E_i$ gives, for each $i$,
    \[
    \sum_{j\neq i}E_i\cdot E_j=E_i\cdot(-2K_S)-E_i^2=2+1=3 ,
    \]
    a sum of three elements of $\{0,1,2\}$.  Hence for each $i$ the multiset
    $\{E_i\cdot E_j\}_{j\ne i}$ is either $\{1,1,1\}$ or $\{0,1,2\}$.
    If some $E_i\cdot E_j=2$, say $E_2=\overline{E_1}$, then
    $E_1+E_2\sim-K_S$ forces $E_3+E_4\sim-K_S$, i.e. $E_4=\overline{E_3}$: this is
    (C7).  Otherwise all $E_i\cdot E_j=1$: this is (C6).

    Conversely, two distinct Geiser pairs always sum to $-2K_S$, so (C7) occurs.  For
    (C6), if $E_1,E_2,E_3\in\mathcal L$ pairwise meet in one point, then as in Step 2
    the class $E_0=E_1+E_2+E_3+K_S$ lies in $\mathcal L$, and
    $E_4:=\overline{E_0}$ satisfies $E_i\cdot E_4=1$ for $i=1,2,3$ and
    $\sum_{i=1}^4E_i=(-K_S+E_0)+(-K_S-E_0)=-2K_S$.  Such triples exist: writing
    $S$ as the blow-up of $\mathbb P^2$ at $p_1,\dots,p_7$ in general position with
    $\mathrm{Pic}(S)=\langle L,E_1,\dots,E_7\rangle$, the curves
    \[
    E_7,\qquad L-E_6-E_7,\qquad 2L-E_1-E_2-E_3-E_4-E_7,\qquad
    3L-E_1-\dots-E_7-E_5
    \]
    are pairwise meeting $(-1)$-curves summing to $6L-2\sum_{i}E_i=-2K_S$.
    Alternatively, (C6) is exactly the configuration obtained from (C3) by replacing
    both conics by reducible members of their pencils, and (C4) by replacing one of
    them.  Similarly (C5) and (C7) arise from (C2) by replacing one, resp. both,
    anticanonical members by a Geiser pair.

    This exhausts all cases and proves the theorem.
\end{proof}

\subsection{Positivity in each case}

\begin{corollary}\label{cor:positivity}
In the notation of Theorem \ref{thm:main-deg2}:
\begin{enumerate}
\item In case {\rm(A)},
  \[
    -(K_S+\Delta_\epsilon)=\tfrac1{m_1}E+\tfrac1{m_2}\overline E
  \]
  is big and nef, with
  $\bigl(-(K_S+\Delta_\epsilon)\bigr)^2=\frac{4}{m_1m_2}-\frac1{m_1^2}-\frac1{m_2^2}>0$,
  it is ample if and only if $m_1<2m_2$ and $m_2<2m_1$.  Moreover
  $\lceil\Delta_\epsilon\rceil=E+\overline E\sim-K_S$, so
  $-(K_S+\lceil\Delta_\epsilon\rceil)$ is trivial.
\item In case {\rm(B)},
  $-(K_S+\Delta_\epsilon)=\half\bigl(-2K_S-E_1-E_2\bigr)$ is nef with
  self-intersection $0$, it is neither big nor ample.  Here
  $-(K_S+\lceil\Delta_\epsilon\rceil)=-K_S-E_1-E_2$ has anticanonical degree $0$
  and is nonzero, hence is not pseudo-effective.
\item In case {\rm(C)}, $-(K_S+\Delta_\epsilon)$ is trivial and
  $-(K_S+\lceil\Delta_\epsilon\rceil)=K_S$ is antiample.
\end{enumerate}
In particular, $-(K_S+\Delta_\epsilon)$ is ample for a reducible boundary exactly
in case {\rm(A)} with $m_1<2m_2$ and $m_2<2m_1$, and then
$-(K_S+\lceil\Delta_\epsilon\rceil)$ is trivial.
\end{corollary}

\begin{proof}
(1) For $E'\in\mathcal L\setminus\{E,\overline E\}$ we have
$E\cdot E'+\overline E\cdot E'=1$ with both terms $\ge0$, so
$-(K_S+\Delta_\epsilon)\cdot E'\in\{1/m_1,1/m_2\}>0$, and
$-(K_S+\Delta_\epsilon)\cdot E=-\frac1{m_1}+\frac2{m_2}$,
$-(K_S+\Delta_\epsilon)\cdot\overline E=-\frac1{m_2}+\frac2{m_1}$.
Writing $a=1/m_1$, $b=1/m_2$, the self-intersection is $-a^2+4ab-b^2$, which is
positive exactly when $2-\sqrt3<a/b<2+\sqrt3$. The nefness constraints
$\half\le m_1/m_2\le2$ give $a/b\in[\half,2]$, so bigness always holds and
by Nakai-Moishezon criterion we get the ampleness.

(2) $-(K_S+\Delta_\epsilon)=\half F'$ with $F'=-2K_S-E_1-E_2$ a conic class, so
$F'^2=0$.  A nonzero pseudo-effective class has positive degree against the ample
class $-K_S$, whence the last assertion.

(3) Immediate from $\Delta_\epsilon\sim-K_S$ and
$\lceil\Delta_\epsilon\rceil=B\sim-2K_S$.
\end{proof}

\begin{remark}
    Cases {\rm(A)} and {\rm(B)} are precisely the reducible boundaries with
    $\deg\Delta_\epsilon<2$, every other reducible boundary is half of a reduced
    member of $|-2K_S|$.  This is the exact analogue, for $K_S^2=2$, of
    Proposition \ref{thm:deg1} for $K_S^2=1$, where by the inequality $\deg\Delta_\epsilon\le1$
    we have $\half(D_1+D_2)$ with $D_i\in|-K_S|$.
\end{remark}

\section{Surfaces of degree 3}

\begin{convention}\label{conv:setup}
Throughout, $S$ is a smooth cubic surface, i.e.\ a del Pezzo surface of degree
$K_S^2=3$. Since $S$ is rational, $\Pic S\cong\Z^7$ and numerical and
linear equivalence coincide on $S$, we write $\sim$ for both, and $\sq$ for
$\Q$-divisors. For a divisor (class) $D$ we put $\deg D:=-K_S\cdot D$, every curve has
positive degree because $-K_S$ is ample. We consider boundaries
\[
  \Delta=\sum_{i=1}^{r}\eps_iD_i,\qquad
  \eps_i=1-\frac{1}{m_i}\in\cE:=\Bigl\{1-\frac1m : m\in\Z_{\ge2}\Bigr\}\cup\{1\},
\]
where $D_1,\dots,D_r$ are pairwise distinct irreducible curves and
$m_i\in\Z_{\ge2}\cup\{\infty\}$. Thus $\half\le\eps_i\le1$, and $\half$ is the only
element of $\cE$ smaller than $\tfrac23$. We put $\deg\Delta:=\sum_i\eps_i\deg D_i$.
We classify the boundaries with $r\ge2$ for which $-(K_S+\Delta)$ is nef, for irreducible
boundaries see \cite{dandapat2026classification}.
\end{convention}

\begin{notation}[Lines]\label{not:lines}
$\cL$ is the set of the $27$ lines on $S$, i.e.\ of the $(-1)$-curves. Two distinct
lines $E\neq E'$ satisfy $E\cdot E'\in\{0,1\}$, accordingly they are
\emph{disjoint}, or they \emph{meet}. For $E\in\cL$ let $\cN(E):=\{F\in\cL : E\cdot F=1\}$ be the set of
\emph{neighbours} of $E$. We say a line meeting two given lines is a \emph{common neighbour} of
them. Also, we say that three lines with $L_1+L_2+L_3\sim-K_S$ form a \emph{tritangent triangle}. If $E,E'$ meet, the plane spanned by $E\cup E'$ cuts out
$E+E'+E''$ for a line $E''$, the \emph{third line} of $E,E'$. Thus $E+E'+E''\sim-K_S$.
\end{notation}

\begin{notation}\label{not:model}
A standard model of $S$ is a birational morphism $\pi\colon S\to\PP^2$. It is the
blow-up of six points $p_1,\dots,p_6$ in general position. We denote
$H:=\pi^*\mathcal{O}_{\PP^2}(1)$ and denote by $E_1,\dots,E_6$ the exceptional curves, so
that $-K_S \sim 3H-\sum_iE_i$. The $27$ lines are
\[
  E_i,\qquad F_{ij}\sim H-E_i-E_j\ \ (1\le i<j\le 6),\qquad
  G_j\sim 2H-\textstyle\sum_{i\neq j}E_i ,
\]
the strict transforms of the points, of the lines $\overline{p_ip_j}$ and of the conics
through $\{p_i\}_{i\ne j}$. For distinct lines we have
\begin{gather*}
  E_i\cdot E_j=G_i\cdot G_j=0,\qquad E_i\cdot F_{jk}=1\iff i\in\{j,k\},\qquad
  E_i\cdot G_j=1\iff i\neq j,\\
  F_{ij}\cdot F_{kl}=1\iff \{i,j\}\cap\{k,l\}=\emptyset,\qquad
  F_{ij}\cdot G_k=1\iff k\in\{i,j\},
\end{gather*}
and $H\cdot E_i=0$, $H\cdot F_{ij}=1$, $H\cdot G_j=2$.
\end{notation}

The following lemma is known to the experts, we state it for completion.

\begin{lemma}\label{lem:nef}
The cone $\NE(S)$ is spanned by the classes of the $27$ lines. Hence a $\Q$-divisor
$N$ on $S$ is nef \textup{(}resp.\ ample\textup{)} if and only if $N\cdot E\ge0$
\textup{(}resp.\ $N\cdot E>0$\textup{)} for every $E\in\cL$. In particular $-(K_S+\Delta)$
is nef \textup{(}resp.\ ample\textup{)} if and only if $\Delta\cdot E\le1$
\textup{(}resp.\ $\Delta\cdot E<1$\textup{)} for every $E\in\cL$.
\end{lemma}

\begin{lemma}\label{lem:deg}
If $-(K_S+\Delta)$ is nef, then $\deg\Delta\le3$, and $\deg\Delta=3$ if and only if
$K_S+\Delta\sq0$. 
\end{lemma}

\begin{proof}
Put $N:=-(K_S+\Delta)$. Since $N$ is nef and $-K_S$ is ample,
$0\le N\cdot(-K_S)=3-\deg\Delta$. If equality holds, then $N^2\ge0$ because $N$ is nef,
while $N^2\le0$ by the Hodge index theorem because $N\cdot(-K_S)=0$, with equality only if
$N\equiv0$. Hence $N\sq0$. The converse is clear.
\end{proof}

In particular, since $\eps_i\ge\half$, we always have $\sum_i\deg D_i\le 6$, and when
$\deg\Delta=3$ the nefness condition is equivalent to $\Delta\sq-K_S$.

\begin{lemma}\label{lem:local}
$-(K_S+\Delta)$ is nef if and only if for every line $E$:
\begin{enumerate}[leftmargin=3em]
\item[\textup{(N1)}] if $E\not\subset\Supp\Delta$, then
      $\sum_i\eps_i\,(D_i\cdot E)\le1$,
\item[\textup{(N2)}] if $E=D_j$, then $\sum_{i\neq j}\eps_i\,(D_i\cdot E)\le1+\eps_j$.
\end{enumerate}
Consequently, if $-(K_S+\Delta)$ is nef and $E\in\cL$ is not a component of $\Delta$, then:
\begin{enumerate}[label=\textup{(\alph*)}]
\item $\sum_iD_i\cdot E\le2$,
\item if $\sum_iD_i\cdot E=2$, every component meeting $E$ has coefficient $\half$,
\item if $\eps_j(D_j\cdot E)>\half$ for some $j$ \textup{(}e.g.\ if $D_j\cdot E\ge2$, or
      if $D_j\cdot E\ge1$ and $\eps_j\ge\frac23$\textup{)}, then $E$ meets no component
      other than $D_j$.
\end{enumerate}
\end{lemma}

\begin{proof}
The first part is Lemma~\ref{lem:nef} together with $E^2=-1$. Since every $\eps_i\ge\half$,
each component meeting $E$ contributes at least $\half(D_i\cdot E)$ to the left hand side
of (N1), this gives (a)-(c).
\end{proof}

\subsection{Lines and curves of low degree}\label{sec:curves}

\begin{lemma}[Configurations of lines]\label{lem:lines3}
The following facts are classical and known to experts. We include
short proofs for completion.
\begin{enumerate}[label=\textup{(\alph*)}]
\item Any $k\le4$ pairwise disjoint lines are among the exceptional curves of a standard
      model, and any six pairwise disjoint lines are the exceptional curves of a standard
      model.
\item Every line $E$ meets exactly ten lines. They form five pairs $\{F,F'\}$ with
      $E+F+F'\sim-K_S$ \textup{(}the five tritangent planes through $E$\textup{)}, and two
      lines of $\cN(E)$ meet if and only if they form such a pair. There are $45$
      tritangent triangles.
\item If $L_1+L_2+L_3\sim-K_S$, then every line $F\notin\{L_1,L_2,L_3\}$ meets exactly one
      of $L_1,L_2,L_3$.
\item Two meeting lines have exactly one common neighbour, namely their third line. In
      particular no four lines meet pairwise, three pairwise meeting lines form a
      tritangent triangle, and two distinct tritangent triangles share at most one line.
\item Two disjoint lines $E,E'$ have exactly five common neighbours $F_1,\dots,F_5$. These
      are pairwise disjoint and lie in the five different tritangent planes through $E$.
      Moreover $E$ and $E'$ are the only lines meeting at least four of $F_1,\dots,F_5$.
\item Three pairwise disjoint lines have exactly three common neighbours, and these are
      pairwise disjoint. Six pairwise disjoint lines have no common neighbour. For a
      standard model and each $j$, the only line meeting all $G_i$ with $i\neq j$ is
      $E_j$.
\end{enumerate}
In particular, any two distinct lines have a common neighbour.
\end{lemma}

\begin{proof}
(a) Contracting pairwise disjoint lines $L_1,\dots,L_k$ gives a birational morphism
$\sigma\colon S\to S'$ onto a smooth surface with $K_{S'}^2=3+k$ and
$\sigma^*(-K_{S'})=-K_S+\sum_iL_i$. For a curve $\Gamma\subset S'$ we get
$-K_{S'}\cdot\Gamma=(-K_S+\sum_iL_i)\cdot\sigma^*\Gamma=-K_S\cdot\sigma^*\Gamma>0$, so
$-K_{S'}$ is ample by Nakai-Moishezon criterion and $S'$ is a del Pezzo surface of degree $3+k$.
If $k\le4$, then $3+k\le7$, so $S'$ is the blow-up of $\PP^2$ in $6-k$ points. Let $\Gamma_1,\dots,\Gamma_{6-k}$ be the exceptional curves. The
strict transform $\tilde\Gamma_l\subset S$ is irreducible, so
$\tilde\Gamma_l^2\ge-1=\Gamma_l^2$. Hence $\Gamma_l$ avoids the points $\sigma(L_i)$, and
$\tilde\Gamma_l=\sigma^*\Gamma_l$ is a line disjoint from $L_1,\dots,L_k$ and from the other
$\tilde\Gamma_{l'}$. The composite $S\to S'\to\PP^2$ is a standard model with
exceptional curves $L_1,\dots,L_k,\tilde\Gamma_1,\dots,\tilde\Gamma_{6-k}$. If $k=6$,
then $S'$ is a del Pezzo surface of degree $9$, i.e.\ $S'\cong\PP^2$.

(b) By (a) we may assume $E=E_1$ in a standard model. By Notation~\ref{not:model},
$\cN(E_1)=\{F_{1j},G_j : 2\le j\le6\}$, $E_1+F_{1j}+G_j\sim-K$, and among these ten
lines $F_{1j}\cdot G_k=1$ iff $k=j$, while $F_{1j}\cdot F_{1k}=G_j\cdot G_k=0$ for
$j\neq k$. Counting gives $27\cdot5/3=45$ tritangent triangles.

(c) $F\cdot(L_1+L_2+L_3)=-K_S\cdot F=1$ and $F\cdot L_i\ge0$.

(d) If $E,E'$ meet with third line $E''$, then by (c) every line
$F\notin\{E,E',E''\}$ meets only one of $E,E',E''$, so $E''$ is the only common
neighbour. The remaining assertions follow at once.

(e) By (a) we may assume $E=E_1$, $E'=E_2$. By Notation~\ref{not:model} the common
neighbours are $F_{12},G_3,G_4,G_5,G_6$. They are pairwise disjoint and lie in the
tritangent planes $E_1+F_{12}+G_2$ and $E_1+F_{1j}+G_j$ ($3\le j\le 6$). A line $E_k$,
$k\ge3$, meets exactly three of them (the $G_j$ with $j\neq k$). A line $F_{ab}$ meets
$G_j$ iff $j\in\{a,b\}$, and meets $F_{12}$ iff $\{a,b\}\cap\{1,2\}=\emptyset$, hence
meets at most three of them. A line $G_k$ meets at most $F_{12}$. So only $E_1,E_2$ meet
four or more of them.

(f) By (a), three pairwise disjoint lines can be taken to be $E_1,E_2,E_3$, whose common
neighbours are $G_4,G_5,G_6$. Six pairwise disjoint lines can be taken to be
$E_1,\dots,E_6$, and no line meets all of them, since $F_{ab}$ meets two and $G_j$ meets
five of them. Finally, for $i\neq j$ the line $E_i$ does not meet $G_i$, a line $F_{ab}$
meets only $G_a,G_b$, and $G_k$ meets no $G_i$, while $E_j$ meets every $G_i$, $i\ne j$.
\end{proof}

We recall the following result from \cite{dandapat2026classification}. 

\begin{lemma}[Curves of low degree]\label{lem:curves3}
Let $D\subset S$ be an irreducible curve of degree $d$. Then
\begin{equation}\label{eq:bounds}
  d-2\ \le\ D^2\ \le\ \frac{d^2}{3},\qquad D^2\equiv d\pmod 2,
\end{equation}
and $3D^2=d^2$ only if $D\sim\frac d3(-K_S)$. More precisely, for $d\le5$:
\begin{enumerate}[label=\textup{(\roman*)}]
\item $d=1$ if and only if $D\in\cL$, then $D^2=-1$.
\item If $d=2$ \textup{(}$D$ is a \emph{conic}\textup{)}, then $D^2=0$ and $D\sim-K_S-R$ for a
      unique line $R$, the \emph{residual line} of $D$ \textup{(}the plane of $D$ cuts out
      $D+R$\textup{)}. The linear system $\lvert-K_S-R\rvert$ is a base point free pencil \textup{(}the
      residual conics of the planes through $R$\textup{)}, whose singular members are
      the five sums $L+L'$ with $R+L+L'\sim-K_S$. We write $C_R$ for an irreducible member
      of $\lvert-K-R\rvert$.
\item If $d=3$, then $D^2\in\{1,3\}$. If $D^2=3$, then $D\in\lvert-K_S\rvert$ is a hyperplane section.
      If $D^2=1$, then $D\sim H$ for a standard model, i.e.\ $D$ is the pullback of a
      line under a birational morphism $S\to\PP^2$.
\item If $d=4$, then $D^2\in\{2,4\}$. If $D^2=4$, then $D\sim-K_S+R$ for a line $R$,
      i.e.\ $D$ is the pullback of $-K_{S'}$ under the contraction $S\to S'$ of $R$ onto
      a del Pezzo surface of degree $4$. If $D^2=2$, then $D\sim-2K_S-B_1-B_2$ for two
      disjoint lines $B_1,B_2$, i.e.\ $D$ is the pullback of the hyperplane class under
      the birational morphism $\psi=(f_1,f_2)\colon S\to\PP^1\times\PP^1$ defined by the
      conic bundles $f_i$ given by $\lvert-K_S-B_i\rvert$.
\item If $d=5$, then $D^2\in\{3,5,7\}$, and if $D^2=7$, then $D\sim-2K-R$ for a line $R$.
\end{enumerate}
The descriptions in \textup{(ii), (iv), (v)} hold for all divisor classes with the given
degree and self-intersection. Conversely, each of the classes $-K_S-R$, $H$, $-K_S$, $-K_S+R$,
$-2K_S-B_1-B_2$ \textup{(}$B_1\cdot B_2=0$\textup{)} and $-2K_S-R$ contains smooth irreducible
curves.
\end{lemma}

\begin{lemma}[Intersection numbers with lines]\label{lem:intersections}
Let $R,B_1,B_2\in\cL$ with $B_1\cdot B_2=0$, and let $E\in\cL$. Then:
\begin{center}
\renewcommand{\arraystretch}{1.2}
\begin{tabular}{@{}lll@{}}
\toprule
curve & class & $D\cdot E$\\
\midrule
conic $C_R$ & $-K_S-R$ & $2$ if $E=R$,\ $0$ if $E\in\cN(R)$ \textup{(}$10$ lines\textup{)},\
                        $1$ otherwise \textup{(}$16$ lines\textup{)}\\
$A$ & $-K_S$ & $1$ for all $E$\\
$T$ & $H$ & $0$ for $E=E_i$,\ $1$ for $E=F_{ij}$,\ $2$ for $E=G_j$\\
$P$ & $-K_S+R$ & $0$ if $E=R$,\ $2$ if $E\in\cN(R)$,\ $1$ otherwise\\
$Q$ & $-2K_S-B_1-B_2$ & $3$ if $E\in\{B_1,B_2\}$,\ $2-B_1\cdot E-B_2\cdot E$ otherwise\\
$D_5$ & $-2K_S-R$ & $3$ if $E=R$,\ $1$ if $E\in\cN(R)$,\ $2$ otherwise\\
\bottomrule
\end{tabular}
\end{center}
Thus $Q\cdot E=0$ for the five common neighbours of $B_1,B_2$, $Q\cdot E=1$ for the ten
lines meeting exactly one of $B_1,B_2$, and $Q\cdot E=2$ for the remaining ten lines.
For conics, $C_R\cdot C_{R'}=1+R\cdot R'$ if $R\neq R'$, and $C\cdot C'=0$ for
$C,C'\in\lvert-K_S-R\rvert$.
\end{lemma}

\begin{proof}
Immediate from Lemma~\ref{lem:curves3}, $-K_S \cdot E=1$, $E^2=-1$,
Notation~\ref{not:model} and Lemma~\ref{lem:lines3}(b),(e).
\end{proof}

\subsection{The classification}\label{sec:main}

\begin{theorem}\label{thm:main3}
Let $\Delta=\sum_{i=1}^r\eps_iD_i$ be as in Convention~\ref{conv:setup} with $r\ge2$.
Then $-(K_S+\Delta)$ is nef if and only if, up to renumbering, $\Delta$ is one of the
following. Here $R,R_i,B_i,L_i$ denote lines, $C_R$ denotes an irreducible conic in
$\lvert-K_S-R\rvert$ \textup{(}Lemma~\textup{\ref{lem:curves3}(ii))}, and all curves occurring in one
item are distinct. The items marked $(\ast)$ are exactly those with $K_S+\Delta\sq0$.

{\upshape
\medskip\noindent\textbf{Case 0} (a component of degree $5$).
\begin{itemlist}
\item[(0.a)] $\half D+\half R$ with $D\in\lvert-2K_S-R\rvert$. \triv
\end{itemlist}

\smallskip\noindent\textbf{Case 1} (a component $P\in\lvert-K_S+R\rvert$ of degree $4$ with $P^2=4$,
i.e.\ the pullback of the anticanonical class of the quartic del Pezzo surface obtained
by contracting $R$).
\begin{itemlist}
\item[(1.a)] $\half P+\half C_R$. \triv
\item[(1.b)] $\half P+\half L_1+\half L_2$ with $R+L_1+L_2\sim-K_S$. \triv
\end{itemlist}

\smallskip\noindent\textbf{Case 2} (a component $Q\in\lvert-2K_S-B_1-B_2\rvert$, $B_1\cdot B_2=0$,
of degree $4$ with $Q^2=2$, i.e.\ the pullback of the hyperplane class of a smooth
quadric).
\begin{itemlist}
\item[(2.a)] $\half Q+\half B_1+\half B_2$. \triv
\end{itemlist}

\smallskip\noindent\textbf{Case 3} (an anticanonical component $A\in\lvert-K_S\rvert$).
\begin{itemlist}
\item[(3.a)] $\half A+\half A'$ with $A'\in\lvert-K_S\rvert$. \triv
\item[(3.b)] $\half A+\half C_R+\half R$. \triv
\item[(3.c)] $\half A+\half L$, with $L$ an arbitrary line.
\item[(3.d)] $\half A+\half(L_1+L_2+L_3)$ with $L_1+L_2+L_3\sim-K_S$. \triv
\end{itemlist}

\smallskip\noindent\textbf{Case 4} (a component $T$ of degree $3$ with $T^2=1$, i.e.\ the
pullback of a line under a birational morphism $S\to\PP^2$, and no anticanonical
component). Choose a standard model with $H\sim T$ (Lemma~\ref{lem:curves3}(iii)) and
use its lines $E_i,F_{ij},G_j$ (Notation~\ref{not:model}).
\begin{itemlist}
\item[(4.a)] $\half T+\half T'$ with $T'\in\lvert-2K_S-T\rvert=|5H-2\sum_iE_i|$. \triv
\item[(4.b)] $\half T+\half G_j+\half C_{E_j}$ for some $j$. \triv
\item[(4.c)] $\half(T+G_i+G_j+F_{ij})$ for some $i\neq j$. \triv
\item[(4.d)] $\half T+\half G_j$ for some $j$.
\end{itemlist}

\smallskip\noindent\textbf{Case 5} (all components of degree $\le2$, at least one of them
a conic).
\begin{itemlist}
\item[(5.a)] $\eps_1C_R+\eps_2R$ with $2\eps_1-\eps_2\le1$, this is $(\ast)$ only for
             $\eps_1=\eps_2=1$.
\item[(5.b)] $\half C_R+\half L$ with $L\neq R$ and $L\cdot R=0$.
\item[(5.c)] $\half C_R+\half R+\half L$ with $L\cdot R=1$.
\item[(5.d)] $\half C_R+\half R+\half L$ with $L\cdot R=0$.
\item[(5.e)] $\half C_R+\half L_1+\half L_2$ with $L_1\cdot L_2=1$ and
             $L_1\cdot R=L_2\cdot R=0$ \textup{(}then the third line of $L_1,L_2$ meets $R$,
             by Lemma~\textup{\ref{lem:lines3}(c)}\textup{)}.
\item[(5.f)] $\half C_R+R+\half L_1+\half L_2$ with $R+L_1+L_2\sim-K_S$. \triv
\item[(5.g)] $\half C_R+\half R+\half(L_1+L_2+L_3)$ with $L_1+L_2+L_3\sim-K_S$. \triv
\item[(5.h)] $\half C+\half(L_1+L_2+L_3+L_4)$, where $L_1,\dots,L_4$ form a
             \emph{$4$-cycle} \textup{(}$L_i\cdot L_{i+1}=1$ for $i\in\Z/4$ and
             $L_1\cdot L_3=L_2\cdot L_4=0$\textup{)} and
             $C\in\lvert-2K_S-L_1-L_2-L_3-L_4\rvert$, a conic class. \triv
\item[(5.i)] $\half C_{R_1}+\half C_{R_2}$ with $R_1\cdot R_2=1$.
\item[(5.j)] $\half(C_{R_1}+C_{R_2}+C_{R_3})$ with $R_1+R_2+R_3\sim-K_S$. \triv
\item[(5.k)] $\half C+\half C'+R$ with $C\neq C'$ in $\lvert-K_S-R\rvert$. \triv
\item[(5.l)] $\half(C_{R_1}+C_{R_2}+R_1+R_2)$ with $R_1\neq R_2$ \textup{(}meeting or
             disjoint\textup{)}. \triv
\item[(5.m)] $\half(C_{R_1}+C_{R_2}+L_1+L_2)$ with $R_1+R_2+R_3\sim L_1+L_2+R_3\sim-K_S$
             and $\{L_1,L_2\}\neq\{R_1,R_2\}$. \triv
\end{itemlist}

\smallskip\noindent\textbf{Case 6} (all components are lines).
\begin{itemlist}
\item[(6.a)] $\half L_1+\half L_2$ \textup{(}meeting or disjoint\textup{)}.
\item[(6.b)] $\eps_1L_1+\eps_2L_2+\eps_3L_3$ with $L_1+L_2+L_3\sim-K_S$ and
             $\eps_i+\eps_j-\eps_k\le1$ for $\{i,j,k\}=\{1,2,3\}$, this is $(\ast)$ only
             for $\eps_1=\eps_2=\eps_3=1$.
\item[(6.c)] $\half(L_1+L_2+L_3)$ with $L_1\cdot L_2=L_2\cdot L_3=1$ and $L_1\cdot L_3=0$.
\item[(6.d)] $\half(L_1+L_2+L_3+L_4)$ with $L_1+L_2+L_3\sim-K_S$ and $L_3\cdot L_4=1$.
\item[(6.e)] $\half(L_1+L_2+L_3+L_4)$ with $L_1,\dots,L_4$ a $4$-cycle.
\item[(6.f)] $R+\half(L_1+L_2+L_3+L_4)$ with $R+L_1+L_2\sim R+L_3+L_4\sim-K_S$. \triv
\item[(6.g)] $\half(L_1+\dots+L_6)$, where either $L_1,L_2,L_3$ are pairwise disjoint and
             $L_4,L_5,L_6$ are their three common neighbours, or
             $L_1+L_2+L_3\sim L_4+L_5+L_6\sim-K_S$. \triv
\end{itemlist}}

\smallskip\noindent
In every item other than \textup{(5.a)} and \textup{(6.b)} the coefficients are uniquely
determined, and every item occurs on every smooth cubic surface.
\end{theorem}

\begin{center}
\small
\renewcommand{\arraystretch}{1.05}
\begin{longtable}{@{}llllll@{}}
\caption{Summary of Theorem~\ref{thm:main3}. Component types: $L$ line, $C$ conic,
$A\in\lvert-K_S\rvert$, $T$ \textup{(}degree $3$, $T^2=1$\textup{)}, $P$ \textup{(}degree $4$,
$P^2=4$\textup{)}, $Q$ \textup{(}degree $4$, $Q^2=2$\textup{)}, $D_5$ \textup{(}degree $5$,
$D^2=7$\textup{)}. The column $\#$ is the number of configurations of classes
$\{[D_1],\dots,[D_r]\}$ of the given type, each entry, and each summand of an entry $a+b$,
is a single $W(E_6)$-orbit. In total there are $13,105$ configurations in $35$ orbits.}
\label{tab:summary}\\
\toprule
Item & Components & Coefficients & $\deg\Delta$ & $K+\Delta\sq0$ & $\#$\\
\midrule
\endfirsthead
\toprule
Item & Components & Coefficients & $\deg\Delta$ & $K+\Delta\sq0$ & $\#$\\
\midrule
\endhead
\bottomrule
\endfoot
(0.a) & $D_5,L$      & $\half,\half$ & $3$ & yes & $27$\\
(1.a) & $P,C$        & all $\half$   & $3$ & yes & $27$\\
(1.b) & $P,L,L$      & all $\half$   & $3$ & yes & $135$\\
(2.a) & $Q,L,L$      & all $\half$   & $3$ & yes & $216$\\
(3.a) & $A,A$        & all $\half$   & $3$ & yes & $1$\\
(3.b) & $A,C,L$      & all $\half$   & $3$ & yes & $27$\\
(3.c) & $A,L$        & all $\half$   & $2$ & no  & $27$\\
(3.d) & $A,L,L,L$    & all $\half$   & $3$ & yes & $45$\\
(4.a) & $T,T$        & all $\half$   & $3$ & yes & $36$\\
(4.b) & $T,L,C$      & all $\half$   & $3$ & yes & $432$\\
(4.c) & $T,L,L,L$    & all $\half$   & $3$ & yes & $1080$\\
(4.d) & $T,L$        & all $\half$   & $2$ & no  & $432$\\
(5.a) & $C,L$        & $2\eps_1-\eps_2\le1$ & $2\eps_1+\eps_2$ & iff $\eps_1=\eps_2=1$ & $27$\\
(5.b) & $C,L$        & all $\half$   & $\tfrac32$ & no & $432$\\
(5.c) & $C,L,L$      & all $\half$   & $2$ & no  & $270$\\
(5.d) & $C,L,L$      & all $\half$   & $2$ & no  & $432$\\
(5.e) & $C,L,L$      & all $\half$   & $2$ & no  & $1080$\\
(5.f) & $C,L,L,L$    & $\half$ except $1$ on $R$ & $3$ & yes & $135$\\
(5.g) & $C,L^4$      & all $\half$   & $3$ & yes & $1080$\\
(5.h) & $C,L^4$      & all $\half$   & $3$ & yes & $1080$\\
(5.i) & $C,C$        & all $\half$   & $2$ & no  & $135$\\
(5.j) & $C,C,C$      & all $\half$   & $3$ & yes & $45$\\
(5.k) & $C,C,L$      & $\half,\half,1$ & $3$ & yes & $27$\\
(5.l) & $C,C,L,L$    & all $\half$   & $3$ & yes & $135+216$\\
(5.m) & $C,C,L,L$    & all $\half$   & $3$ & yes & $540$\\
(6.a) & $L,L$        & all $\half$   & $1$ & no  & $135+216$\\
(6.b) & $L,L,L$      & $\eps_i+\eps_j-\eps_k\le1$ & $\eps_1+\eps_2+\eps_3$ & iff all $\eps_i=1$ & $45$\\
(6.c) & $L,L,L$      & all $\half$   & $\tfrac32$ & no & $1080$\\
(6.d) & $L^4$        & all $\half$   & $2$ & no  & $1080$\\
(6.e) & $L^4$        & all $\half$   & $2$ & no  & $1080$\\
(6.f) & $L^5$        & $\half$ except $1$ on $R$ & $3$ & yes & $270$\\
(6.g) & $L^6$        & all $\half$   & $3$ & yes & $360+720$\\
\end{longtable}
\end{center}

\needspace{8\baselineskip}
\begin{remark}[Structure of the list]\label{rem:structure}
We record some features of Theorem~\ref{thm:main3}.
\begin{enumerate}[label=\textup{(\alph*)}]
\item The items (5.k), (5.f), (6.f) are all of the form $R+\half(Z_1+Z_2)$, where $Z_1\neq
      Z_2$ are members of the conic bundle $\lvert-K_S-R\rvert$, both smooth, one singular, or both
      singular. They are exactly the items in which some coefficient is forced to be $1$.
\item Consequently, if only finite multiplicities $m_i<\infty$ are allowed, the items
      (5.f), (5.k), (6.f) disappear, and in (5.a), (6.b) the extreme points
      $\eps=(1,1)$, $(1,1,1)$ are excluded. In terms of multiplicities (with
      $1/\infty=0$), the condition in (5.a) implies $\frac1{m_2}\le\frac2{m_1}$, i.e.\
      $m_1\le2m_2$ for finite multiplicities ($m_1$ belonging to the conic, $m_2$ to
      $R$), and the condition in (6.b) says that $\frac1{m_1},\frac1{m_2},\frac1{m_3}$
      satisfy the triangle inequalities.
\item Only (5.a) and (6.b) allow a non-trivial family of coefficients, in all other items
      every coefficient is $\half$, except for the coefficient $1$ of $R$ in the three
      items listed in (a).
\item Apart from (5.a) and (6.b), one has $K_S+\Delta\not\sim_{\Q}0$ only in the items
      (3.c), (4.d), (5.b)-(5.e), (5.i), (6.a), (6.c)-(6.e), and in all of them
      $\deg\Delta\le2$.
\end{enumerate}
\end{remark}

\begin{corollary}[Ampleness]\label{cor:ample}
Let $r\ge2$. Then $-(K_S+\Delta)$ is ample if and only if either
\begin{enumerate}[label=\textup{(\roman*)}]
\item $\Delta=\eps_1C_R+\eps_2R$ as in \textup{(5.a)} with $2\eps_1-\eps_2<1$ and
      $\eps_2<1$, i.e.\ $m_1<2m_2<\infty$, or
\item $\Delta=\eps_1L_1+\eps_2L_2+\eps_3L_3$ as in \textup{(6.b)} with all $\eps_i<1$ and
      $\eps_i+\eps_j-\eps_k<1$, i.e.\ all $m_i<\infty$ and
      $\frac1{m_k}<\frac1{m_i}+\frac1{m_j}$ for $\{i,j,k\}=\{1,2,3\}$.
\end{enumerate}
In particular, $\Supp\Delta$ is then a hyperplane section of $S\subset\PP^3$.
\end{corollary}

\begin{proof}
By Lemma~\ref{lem:nef}, $-(K+\Delta)$ is ample iff $\Delta\cdot E<1$ for all $E\in\cL$.
In the items marked $(\ast)$, $\Delta\cdot E=1$ for every line. In the other items
with fixed coefficients, a line $E$ with $\Delta\cdot E=1$ is: any $E\in\cN(L)$ in
(3.c), $G_i$, $i\neq j$, in (4.d), $R$ in (5.b), (5.c), (5.e), $R_1$ in (5.i), a common
neighbour of $R$ and $L$ in (5.d), a common neighbour of $L_1,L_2$ in (6.a), the third
line of $L_1,L_2$ in (6.c), (6.e), the third line of $L_3,L_4$ in (6.d). In (5.a) one
has $\Delta\cdot R=2\eps_1-\eps_2$, $\Delta\cdot E=\eps_2$ for $E\in\cN(R)$ and
$\Delta\cdot E=\eps_1\le\frac{1+\eps_2}2$ otherwise, in (6.b), $\Delta\cdot
L_i=\eps_j+\eps_k-\eps_i$, and $\Delta\cdot E=\eps_i$ for the lines $E\notin\{L_1,L_2,
L_3\}$ meeting $L_i$. These computations are justified in
Propositions~\ref{prop:case5a}-\ref{prop:case6} below.
\end{proof}

\begin{remark}[Computer verification]\label{rem:computer}
Independently of the proof in \S\ref{sec:proof}, the classification was re-derived by
Python 3 scripts at the level of divisor classes, in terms of rational arithmetic.
Taking Lemmas~\ref{lem:nef}
and~\ref{lem:deg} and adjunction as input, and writing classes as $aH-\sum_ib_iE_i$, the
computation
\begin{enumerate}[label=\textup{(\roman*)}]
\item lists the $27$ lines and the classes of all irreducible curves of degree $2,\dots,5$,
      namely the $1018$ nef classes with $D^2\ge\deg D-2$ \textup{(}cf.\
      Lemma~\ref{lem:curves3}\textup{)},
\item runs through all $1,523,161$ multisets of these classes with $r\ge2$ and total degree
      $\le6$ \textup{(}a line at most once\textup{)}, and discards those for which some line
      $E$ has $\Delta\cdot E>1$ for \emph{every} $\eps\in[\half,1]^r$, this leaves $13,105$
      supports,
\item exhibits for each of them a coefficient vector with all $m_i\le10$ or $m_i=\infty$
      satisfying the $27$ inequalities of Lemma~\ref{lem:nef},
\item computes the $W(E_6)$-orbits of these supports: there are $35$, and they are exactly
      the items of Theorem~\ref{thm:main3}, generated from their descriptions, with the sizes
      given in Table~\ref{tab:summary},
\item determines, by exact linear programming, the polytope
      $\{\eps\in[\half,1]^r : \Delta\cdot E\le1 \text{ for all } E\in\cL\}$ of each orbit: it
      is the stated point in all items except (5.a) and (6.b), and the stated region there.
\end{enumerate}
Since (ii) and (v) concern arbitrary coefficients in $[\half,1]$, this covers all $m_i$.
The same scripts also check Lemmas~\ref{lem:lines3}-\ref{lem:intersections}, the marks
$(\ast)$ and Corollary~\ref{cor:ample}.
\end{remark}

\subsection{Proof of Theorem~\ref{thm:main3}}\label{sec:proof}

\subsubsection{Sufficiency and first reductions}
For each item marked $(\ast)$ we check directly, using Lemma~\ref{lem:curves3} and
Notation~\ref{not:lines}, that $\Delta\sq-K_S$, hence $-(K_S+\Delta)\sq0$ is nef. For
instance, in (4.b) $H+G_j+(-K_S-E_j)\sim H+(2H-\sum_iE_i+E_j)+(3H-\sum_iE_i-E_j)=-2K_S$, in
(4.c) $H+G_i+G_j+F_{ij}\sim6H-2\sum_iE_i=-2K_S$, in (5.h) this is the definition of $C$,
which is a conic class by Lemma~\ref{lem:curves3}(ii) because
$(-2K_S-\sum_iL_i)^2=12-16+4=0$ and $\deg(-2K_S-\sum_iL_i)=2$, in (5.m)
$L_1+L_2\sim R_1+R_2\sim-K_S-R_3$, for (6.g) see the end of the proof of
Proposition~\ref{prop:case6}. The remaining items are shown to be nef in
Propositions~\ref{prop:case3}-\ref{prop:case6}.

It remains to show that every $\Delta$ with $-(K_S+\Delta)$ nef occurs in the list. From
now on we assume that $-(K_S+\Delta)$ is nef and $r\ge2$, we write $d_i:=\deg D_i$, and we
number the components so that $d_1=\max_id_i$.

\begin{lemma}\label{lem:reduction}
Under these assumptions the following hold.
\begin{enumerate}[label=\textup{(\alph*)}]
\item $\sum_id_i\le6$, hence $d_1\le5$.
\item If $d_1\ge4$, then $\eps_1=\half$ and $\sum_{i\ge2}\eps_id_i\le3-\frac{d_1}2$.
\item Exactly one of the following holds:
  \textbf{Case 0}: $d_1=5$,
  \textbf{Case 1}: $d_1=4$ and $D_1^2=4$,
  \textbf{Case 2}: $d_1=4$ and $D_1^2=2$,
  \textbf{Case 3}: $d_1=3$ and some component lies in $\lvert-K_S\rvert$,
  \textbf{Case 4}: $d_1=3$ and no component lies in $\lvert-K_S\rvert$ \textup{(}so every component of
  degree $3$ has self-intersection $1$\textup{)},
  \textbf{Case 5}: $d_1=2$,
  \textbf{Case 6}: $d_1=1$.
\end{enumerate}
\end{lemma}

\begin{proof}
(a) $\half\sum_id_i\le\deg\Delta\le3$ by Lemma~\ref{lem:deg}, and $d_1\le6-1$ since
$r\ge2$. (b) The other components contribute at least $\half$ to $\deg\Delta\le3$, so
$\eps_1\le\frac5{2d_1}\le\frac58<\frac23$, whence $\eps_1=\half$, the second assertion
follows. (c) follows from Lemma~\ref{lem:curves3}, in Cases~1 and~2 the component of degree
$4$ is unique by (b).
\end{proof}

\subsubsection{A component of degree at least \texorpdfstring{$4$}{4}}

\begin{proposition}[Case 0]\label{prop:case0}
If $d_1=5$, then $\Delta$ is as in \textup{(0.a)}.
\end{proposition}

\begin{proof}
By Lemma~\ref{lem:reduction}(b), $\eps_1=\half$ and $\sum_{i\ge2}\eps_id_i\le\half$. Hence
$r=2$, $R:=D_2$ is a line and $\eps_2=\half$. Then $\deg\Delta=3$, so $\Delta\sq-K_S$ by
Lemma~\ref{lem:deg}, i.e.\ $D_1\sim-2K_S-R$.
\end{proof}

\begin{proposition}[Case 1]\label{prop:case1}
If $d_1=4$ and $D_1^2=4$, then $\Delta$ is as in \textup{(1.a)} or \textup{(1.b)}.
\end{proposition}

\begin{proof}
By Lemma~\ref{lem:curves3}(iv), $P:=D_1\sim-K_S+R$ for a line $R$, by
Lemma~\ref{lem:reduction}(b), $\eps_1=\half$ and $\Gamma:=\Delta-\half P$ satisfies
$\deg\Gamma\le1$. Thus $\Gamma$ is $\eps C$ for a conic $C$, or $\eps L$ for a line $L$, or
$\eps L_1+\eps'L_2$ for lines $L_1,L_2$. By Lemma~\ref{lem:intersections}, $P\cdot F=2$
for every $F\in\cN(R)$, so by Lemma~\ref{lem:local}(c):
\begin{equation}\label{eq:case1}
  \text{no line } F\in\cN(R) \text{ with } F\not\subset\Supp\Delta
  \text{ meets a component of }\Gamma.
\end{equation}
If $\Gamma=\eps L$, let $F$ be a common neighbour of $R$ and $L$ (Lemma~\ref{lem:lines3},
any $F\in\cN(R)$ if $L=R$). Then $F$ is not a component, $F\in\cN(R)$, and $F$ meets
$L$, contradicting \eqref{eq:case1}. In the other two cases $\deg\Gamma\ge1$, so
$\deg\Gamma=1$, all coefficients of $\Gamma$ are $\half$, $\deg\Delta=3$, and by
Lemma~\ref{lem:deg}, $\Gamma\sq-K_S-\half P\sq\half(-K_S-R)$. Thus either $\Gamma=\half C$
with $C\sim-K_S-R$, which is (1.a), or $\Gamma=\half(L_1+L_2)$ with $L_1+L_2\sim-K_S-R$, which
is (1.b).
\end{proof}

\begin{proposition}[Case 2]\label{prop:case2}
If $d_1=4$ and $D_1^2=2$, then $\Delta$ is as in \textup{(2.a)}.
\end{proposition}

\begin{proof}
By Lemma~\ref{lem:curves3}(iv), $Q:=D_1\sim-2K_S-B_1-B_2$ with $B_1\cdot B_2=0$, and
$\eps_1=\half$. Since $\half(Q\cdot B_i)=\frac32>1$, condition (N1) forces
$B_1,B_2\subset\Supp\Delta$. Then $\deg\Delta\ge2+\half+\half=3$, so equality holds and
$\Delta=\half(Q+B_1+B_2)$.
\end{proof}

\subsubsection{A component of degree \texorpdfstring{$3$}{3}}

\begin{proposition}[Case 3]\label{prop:case3}
If $d_1=3$ and $A:=D_1\in\lvert-K_S\rvert$, then $\Delta$ is as in \textup{(3.a)}-\textup{(3.d)},
conversely, the boundaries in \textup{(3.c)} are nef.
\end{proposition}

\begin{proof}
Put $\Gamma:=\Delta-\eps_1A\neq0$. Since $\deg\Delta\le3$, we get $\eps_1\le\frac56$ and
$\deg\Gamma\le3-3\eps_1\le\frac32$. As $A\cdot E=1$ for every line, (N1) gives
\begin{equation}\label{eq:case3}
  \Gamma\cdot E\le1-\eps_1\le\half\qquad\text{for every line }E\not\subset\Supp\Delta.
\end{equation}
(i) If $\Gamma$ has a component $D$ of degree $3$, then $\deg\Gamma\ge\frac32$ forces
$\eps_1=\half$, $\Gamma=\half D$ and $\deg\Delta=3$, by Lemma~\ref{lem:deg}, $D\sim-K_S$:
(3.a).

(ii) If $\Gamma$ has a conic component $C\sim-K_S-R$, then $C\cdot R=2$ and
\eqref{eq:case3} force $R\subset\Supp\Delta$, now $\deg\Gamma\ge1+\half$ forces
$\Delta=\half(A+C+R)$: (3.b).

(iii) Otherwise $\Gamma=\sum_{j\ge2}\eps_jL_j$ with lines $L_j$, and $r-1\le3$ because
$\sum_{j\ge2}\eps_j\le\frac32$. If $r-1=3$, then $\deg\Gamma\ge\frac32$ forces all
coefficients to be $\half$ and $\deg\Delta=3$, so $L_2+L_3+L_4\sim-K_S$: (3.d). If
$r-1=2$, a common neighbour $F$ of $L_2,L_3$ is not a component, and
$\Gamma\cdot F\ge\eps_2+\eps_3\ge1$ contradicts \eqref{eq:case3}. If $r-1=1$, then
\eqref{eq:case3} for $F\in\cN(L_2)$ gives $\eps_2\le1-\eps_1$, so $\eps_1=\eps_2=\half$:
(3.c).

Conversely, for $\Delta=\half(A+L)$ we have $\Delta\cdot F=\half+\half L\cdot F\le1$ for
all lines $F\neq L$, and $\Delta\cdot L=0$.
\end{proof}

\begin{proposition}[Case 4]\label{prop:case4}
If $d_1=3$ and no component lies in $\lvert-K_S\rvert$, then $\Delta$ is as in
\textup{(4.a)}-\textup{(4.d)}, conversely, the boundaries in \textup{(4.d)} are nef.
\end{proposition}

\begin{proof}
By Lemma~\ref{lem:curves3}(iii), $T:=D_1\sim H$ for a standard model. We use its lines
$E_i,F_{ij},G_j$, so that $T\cdot E_i=0$, $T\cdot F_{ij}=1$ and $T\cdot G_j=2$. Put
$\Gamma:=\Delta-\eps_1T$. Then $\deg\Gamma\le3-3\eps_1\le\frac32$, so $\Gamma$ has at
most three line components. Hence some $G_j$ is not a component, and (N1) at $G_j$ gives
$2\eps_1\le1$, i.e.\ $\eps_1=\half$. Now by (N1) 
\begin{equation}\label{eq:case4}
  \Gamma\cdot G_j=0\quad\text{for every } G_j\not\subset\Supp\Delta,
  \ \text{i.e.\ such a } G_j \text{ meets no component of }\Gamma.
\end{equation}
(i) If $\Gamma$ has a component $T'$ of degree $3$, then $\deg\Gamma\le\frac32$ forces
$\Gamma=\half T'$, and $\deg\Delta=3$ gives $T'\sim-2K_S-T$ by Lemma~\ref{lem:deg}: (4.a).
Note that $-2K_S-H\sim5H-2\sum_iE_i$ is the class $H'$ of the standard model contracting
$G_1,\dots,G_6$ (Lemma~\ref{lem:lines3}(a)), indeed $-K_S\sim3H'-\sum_jG_j$.

From now on all components of $\Gamma$ are lines or conics.

(ii) No $E_k$ is a component: otherwise, by \eqref{eq:case4}, the five lines $G_j$,
$j\neq k$, would be components.

(iii) If some $F_{ab}$ is a component, then by \eqref{eq:case4} so are $G_a$ and $G_b$,
as $\deg\Gamma\le\frac32$, we get $\Gamma=\half(F_{ab}+G_a+G_b)$: (4.c).

(iv) Otherwise every line component of $\Gamma$ is some $G_k$, and at most one occurs:
if $G_a,G_b$ \textup{(}$a\ne b$\textup{)} were components, then $F_{ab}$ would not be a
component and $\Delta\cdot F_{ab}\ge\half+\half+\half$. Suppose $\Gamma$ has a conic
component $C\sim-K_S-R$. By \eqref{eq:case4}, $0=C\cdot G_j=1-R\cdot G_j$, i.e.\ $R$ meets
$G_j$, for every $G_j$ that is not a component. By Lemma~\ref{lem:lines3}(f), $R$ cannot
meet all six $G_j$, hence exactly one $G_k$ is a component, $R$ meets the other five, and
$R=E_k$ by Lemma~\ref{lem:lines3}(f). Since $\deg\Gamma\le\frac32$, we obtain
$\Gamma=\half(C+G_k)$ with $C\in\lvert-K_S-E_k\rvert$: (4.b). If $\Gamma$ has no conic component,
then $\Gamma=\eps G_k$, and (N1) at $F_{kj}$ ($j\ne k$) gives $\half+\eps\le1$, so
$\eps=\half$: (4.d).

Conversely, for $\Delta=\half(H+G_k)$ we have $\Delta\cdot E_i\le\half$,
$\Delta\cdot F_{ij}=\half+\half F_{ij}\cdot G_k\le1$, $\Delta\cdot G_j=1$ for $j\neq k$,
and $\Delta\cdot G_k=\half$.
\end{proof}

\subsubsection{Conics and lines}
In Propositions~\ref{prop:case5a} and~\ref{prop:case5b} all components have degree
$\le2$ and at least one is a conic. For a conic $C$ with residual line $R$ recall from
Lemma~\ref{lem:intersections} that $C\cdot R=2$, $C\cdot E=0$ for $E\in\cN(R)$, and
$C\cdot E=1$ for the $16$ lines $E$ disjoint from $R$. By Convention~\ref{conv:setup}, a
conic component has coefficient $\half$ unless its coefficient is $\ge\frac23$.

\begin{proposition}[Case 5, a conic with coefficient $\ge\frac23$]\label{prop:case5a}
Suppose all components have degree $\le2$ and some conic component $C$, with residual
line $R$, has coefficient $\eps_C\ge\frac23$. Then $\Delta=\eps_CC+\eps R$ with
$2\eps_C-\eps\le1$, as in \textup{(5.a)}. Conversely, for every conic $C$ with residual
line $R$ and all $\eps_1,\eps_2\in\cE$, the divisor $-(K_S+\eps_1C+\eps_2R)$ is nef if and
only if $2\eps_1-\eps_2\le1$.
\end{proposition}

\begin{proof}
As $\eps_C(C\cdot R)=2\eps_C>1$, (N1) forces $R\subset\Supp\Delta$, let $\eps$ be its
coefficient and $\Gamma:=\Delta-\eps_CC-\eps R$. Then
$\deg\Gamma\le3-\frac43-\frac12=\frac76$, so $\Gamma$ is zero, or a single conic with its
coefficient, or it has at most two components, both lines. Let $\cM$ be the set of the
$16$ lines disjoint from $R$. Every $F\in\cM$ satisfies $\eps_C(C\cdot F)=\eps_C>\half$,
so by Lemma~\ref{lem:local}(c):
\begin{equation}\label{eq:case5a}
  \text{no line } F\in\cM \text{ that is not a component of } \Gamma
  \text{ meets a component of }\Gamma.
\end{equation}
If $\Gamma=\eps'C'$ for a conic $C'$ with residual line $R'$, then $C'\cdot F\ge1$ for
every $F\in\cM\backslash \cN(R')$, a set with at least $16-10=6$ elements,
contradicting \eqref{eq:case5a}. If $\Gamma$ has a line component $L$, then $L$ meets at
least five lines of $\cM$: if $L\in\cN(R)$, the eight lines of $\cN(L)$ other than $R$
and the third line of $R,L$ lie in $\cM$ (Lemma~\ref{lem:lines3}(b)), if $L\cdot R=0$, the
five lines of $\cN(L)$ that are not common neighbours of $L$ and $R$ lie in $\cM$
(Lemma~\ref{lem:lines3}(e)). At most one of them is a component of $\Gamma$, which
contradicts \eqref{eq:case5a}. Hence $\Gamma=0$.

Finally, for $\Delta=\eps_1C+\eps_2R$ we have $\Delta\cdot R=2\eps_1-\eps_2$,
$\Delta\cdot F=\eps_2$ for $F\in\cN(R)$, and $\Delta\cdot F=\eps_1$ for $F\in\cM$. So
$-(K+\Delta)$ is nef if and only if $2\eps_1-\eps_2\le1$.
\end{proof}

\begin{proposition}[Case 5, all conics with coefficient $\half$]\label{prop:case5b}
Suppose all components have degree $\le2$, at least one is a conic, and all conic
components have coefficient $\half$. Then $\Delta$ is as in \textup{(5.a)} \textup{(}with
$\eps_1=\half$\textup{)} or \textup{(5.b)}-\textup{(5.m)}, conversely, the boundaries in
\textup{(5.b)}-\textup{(5.e)} and \textup{(5.i)} are nef.
\end{proposition}

\begin{proof}
Let $C_1,\dots,C_s$ ($s\ge1$) be the conic components, with residual lines
$R_1,\dots,R_s$, and $L_1,\dots,L_t$ the line components, with coefficients
$\lambda_1,\dots,\lambda_t$. Then $\deg\Delta=s+\sum_j\lambda_j\le3$. Since
$\half(C_k\cdot R_k)=1$, By Lemma~\ref{lem:local}(c):
\begin{equation}\label{eq:case5b}
  \parbox{0.86\textwidth}{if $R_k$ is not a component, then no component other than
  $C_k$ meets $R_k$, in particular $R_l$ meets $R_k$ for every $l\neq k$ \textup{(}as
  $C_l\cdot R_k=1-R_l\cdot R_k$ if $R_l\neq R_k$, and $C_l\cdot R_k=2$ if $R_l=R_k$\textup{)}.}
\end{equation}

\emph{$s=3$.} Then $t=0$ and $\deg\Delta=3$, so $\half(C_1+C_2+C_3)\sq-K_S$ by
Lemma~\ref{lem:deg}, i.e.\ $R_1+R_2+R_3\sim-K_S$: (5.j).

\emph{$s=2$.} Then $\sum_j\lambda_j\le1$, so $t\le2$.

\emph{$t=0$}: By \eqref{eq:case5b}, $R_1\cdot R_2=1$: (5.i). Conversely,
$\Delta\cdot R_k=1+\half C_l\cdot R_k=1$ for $\{k,l\}=\{1,2\}$, and
$\Delta\cdot F=\half(2-R_1\cdot F-R_2\cdot F)\le1$ for all other lines $F$.

\emph{$t=1$}: Write $L:=L_1$, $\lambda:=\lambda_1$. Suppose first $L\notin\{R_1,R_2\}$.
Then by \eqref{eq:case5b}, $R_1\cdot R_2=1$ and $L\cdot R_1=L\cdot R_2=0$. Let $R_3$ be
the third line of $R_1,R_2$, then $L\neq R_3$, so $L$ meets $R_3$ by
Lemma~\ref{lem:lines3}(c). By Lemma~\ref{lem:lines3}(e), $\cN(L)\cap\cN(R_i)$ has five
elements for $i=1,2$, and $R_3$ lies in both, hence some $F\in\cN(L)$ meets neither $R_1$
nor $R_2$. Then $F$ is not a component and $\Delta\cdot F=\half+\half+\lambda>1$, a
contradiction. So $L=R_1$, say. If $R_2\neq R_1$, then \eqref{eq:case5b} at $R_2$ gives
$C_1\cdot R_2=0$ and $L\cdot R_2=0$, i.e.\ $R_1\cdot R_2=1$ and $R_1\cdot R_2=0$, which
is absurd. Hence $R_1=R_2=:R$, so $C_1,C_2$ are distinct members of $\lvert-K_S-R\rvert$, and by (N2) at
$R$ we get $2\le1+\lambda$, thus $\lambda=1$: (5.k).

\emph{$t=2$}: Now $\lambda_1+\lambda_2\le1$ forces $\lambda_1=\lambda_2=\half$ and
$\deg\Delta=3$, so $\Delta\sq-K_S$, i.e.
\begin{equation}\label{eq:case5b2}
  L_1+L_2\sim R_1+R_2.
\end{equation}
If $\{L_1,L_2\}=\{R_1,R_2\}$, this is (5.l). Otherwise we can assume
$L_1\notin\{R_1,R_2\}$, then also $L_2\notin\{R_1,R_2\}$ by \eqref{eq:case5b2}.
Intersecting \eqref{eq:case5b2} with $L_i$ gives $-1+L_1\cdot L_2=L_i\cdot R_1+L_i\cdot
R_2\ge0$, so $L_1\cdot L_2=1$ and $L_i\cdot R_k=0$ for all $i,k$, intersecting with $R_1$
then gives $0=-1+R_1\cdot R_2$. Thus $R_1\cdot R_2=1$ and, with $R_3$ the third line of
$R_1,R_2$, we get $L_1+L_2\sim R_1+R_2\sim-K_S-R_3$: (5.m).

\emph{$s=1$.} Write $C:=C_1$ and $R:=R_1$, now $\sum_j\lambda_j\le2$, so $t\le4$. If $R$
is a component, with coefficient $\lambda_R$, then for every line $F\neq R$
\begin{equation}\label{eq:case5b3}
  (\half C+\lambda_RR)\cdot F=\half+(\lambda_R-\half)\,R\cdot F\ \ge\ \half .
\end{equation}

\emph{$t=1$}: If $L_1=R$, we are in (5.a) with $\eps_1=\half$, which is nef for every
$\lambda_1$ by Proposition~\ref{prop:case5a}. Otherwise $L_1\cdot R=0$ by
\eqref{eq:case5b}, choose $F\in\cN(L_1)\backslash\cN(R)$ (Lemma~\ref{lem:lines3}(e)).
Then $F$ is not a component and $\Delta\cdot F=\half+\lambda_1$, so $\lambda_1=\half$:
(5.b). Conversely, for $\Delta=\half(C+L_1)$ we have $\Delta\cdot R=1$,
$\Delta\cdot L_1=0$, and $\Delta\cdot F\le\half+\half$ for all other lines $F$.

\emph{$t=2$, $R$ a component}: Say $L_1=R$ and put $L:=L_2$. By \eqref{eq:case5b3},
$\Delta\cdot F\ge\half+\lambda_2$ for $F\in\cN(L)\backslash\{R\}$, so
$\lambda_2=\half$, and a common neighbour $F$ of $R$ and $L$ has $C\cdot F=0$ and
$\Delta\cdot F=\lambda_1+\lambda_2$, so $\lambda_1=\half$. This is (5.c) if $L\cdot R=1$
and (5.d) if $L\cdot R=0$. Conversely, for $\Delta=\half(C+R+L)$ we have
$\Delta\cdot F=\half+\half L\cdot F\le1$ for $F\notin\{R,L\}$,
$\Delta\cdot R=\half+\half L\cdot R\le1$ and $\Delta\cdot L=0$.

\emph{$t=2$, $R$ not a component}: By \eqref{eq:case5b}, $L_1\cdot R=L_2\cdot R=0$. Every
common neighbour $F$ of $L_1,L_2$ is not a component, and $\Delta\cdot F=\half C\cdot
F+\lambda_1+\lambda_2\le1$ gives $\lambda_1=\lambda_2=\half$ and $C\cdot F=0$, i.e.\ $F$
meets $R$. If $L_1\cdot L_2=0$, then $R$ meets all five common neighbours of $L_1,L_2$,
so $R\in\{L_1,L_2\}$ by Lemma~\ref{lem:lines3}(e), which is absurd. Hence $L_1\cdot
L_2=1$ and their third line meets $R$: (5.e). Conversely, let $\Delta$ be as in (5.e). A
line $F\notin\{L_1,L_2\}$ meeting both $L_1$ and $L_2$ is their third line, for which
$C\cdot F=0$, hence $\Delta\cdot F\le1$ for all lines $F\notin\{R,L_1,L_2\}$, while
$\Delta\cdot R=1$ and $\Delta\cdot L_1=\Delta\cdot L_2=\half$.

\emph{$t=3$, $R$ a component}: Say $L_1=R$. The lines $L_2,L_3$ have a common neighbour
$F\neq R$ unless $R+L_2+L_3\sim-K_S$ (if $L_2\cdot L_3=0$ they have five common
neighbours, and if $L_2\cdot L_3=1$ their only common neighbour is their third line).
Such an $F$ is not a component, and $\Delta\cdot F\ge\half+\lambda_2+\lambda_3>1$ by
\eqref{eq:case5b3}. Hence $R+L_2+L_3\sim-K_S$, and by (N2) at $R$ we get
$1+\lambda_2+\lambda_3\le1+\lambda_1$, so $\lambda_1=1$ and $\lambda_2=\lambda_3=\half$:
(5.f).

\emph{$t=3$, $R$ not a component}: By \eqref{eq:case5b}, $L_1,L_2,L_3$ are disjoint from
$R$, so by Lemma~\ref{lem:lines3}(c) they do not form a tritangent triangle, and by
Lemma~\ref{lem:lines3}(d) two of them, say $L_1$ and $L_3$, are disjoint. At least four of
their five common neighbours are not components (only $L_2$ could be one), and, as in the
case $t=2$, each of these meets $R$. By Lemma~\ref{lem:lines3}(e), $R\in\{L_1,L_3\}$, a
contradiction.

\emph{$t=4$}: Now $\sum_j\lambda_j\le2$ forces $\lambda_j=\half$ and $\deg\Delta=3$, so
$\Delta\sq-K_S$, i.e.
\begin{equation}\label{eq:case5b4}
  L_1+L_2+L_3+L_4\sim-2K_S-C\sim-K_S+R.
\end{equation}
Intersecting with $R$ gives $\sum_jL_j\cdot R=0$. If $R$ is one of the $L_j$, say
$L_4=R$, then $L_1+L_2+L_3\sim-K_S$: (5.g). Otherwise all $L_j$ are disjoint from $R$, and
intersecting \eqref{eq:case5b4} with $L_j$ shows that $L_j$ meets exactly two of the
other three lines, so, after renumbering, $L_1,\dots,L_4$ form a $4$-cycle and
$C\in\lvert-2K_S-\sum_jL_j\rvert$: (5.h).
\end{proof}

\subsubsection{Lines only}

\begin{proposition}[Case 6]\label{prop:case6}
Suppose all components are lines, $\Delta=\sum_{j=1}^r\eps_jL_j$. Then $\Delta$ is as in
\textup{(6.a)}-\textup{(6.g)}, conversely, the boundaries in
\textup{(6.a)}-\textup{(6.e)} are nef.
\end{proposition}

\begin{proof}
Since $\sum_j\eps_j\le3$, we have $r\le6$. Let $\Gamma$ be the graph with vertices
$L_1,\dots,L_r$ in which two vertices are adjacent if the lines meet. By
Lemma~\ref{lem:local}(a),(b):
\begin{equation}\label{eq:case6}
  \parbox{0.86\textwidth}{a line that is not a component meets at most two components,
  and if it meets two, both have coefficient $\half$.}
\end{equation}
By Lemma~\ref{lem:lines3}(d), every triangle of $\Gamma$ is a tritangent triangle, and
$\Gamma$ contains neither four pairwise adjacent vertices nor two triangles sharing an
edge. Moreover,
\begin{equation}\label{eq:case6b}
  \text{if } r\le5, \text{ then } \Gamma \text{ has no three pairwise non-adjacent
  vertices:}
\end{equation}
three pairwise disjoint components have three common neighbours
(Lemma~\ref{lem:lines3}(f)), and as $r\le5$ one of these is not a component, contradicting
\eqref{eq:case6}.

\emph{$r=2$.} A common neighbour of $L_1,L_2$ is not a component, so
$\eps_1=\eps_2=\half$: (6.a). Conversely, $\Delta\cdot F\le1$ for every line
$F\notin\{L_1,L_2\}$, and $\Delta\cdot L_i=-\half+\half L_1\cdot L_2\le0$.

\emph{$r=3$.} If $L_1+L_2+L_3\sim-K_S$, then every other line meets exactly one $L_i$
(Lemma~\ref{lem:lines3}(c)), so (N1) holds, and by (N2) at $L_i$ we get 
$\eps_j+\eps_k\le1+\eps_i$: (6.b). Otherwise, by \eqref{eq:case6b}, $\Gamma$ has one or two
edges. If it has one edge, say $L_1L_2$, the third line of $L_1,L_2$ is not a component
and meets $L_3$ by Lemma~\ref{lem:lines3}(c), contradicting \eqref{eq:case6}. If it has
two edges, say $L_1L_2$ and $L_2L_3$, the third lines of $L_1,L_2$ and of $L_2,L_3$ are
not components, so $\eps_1=\eps_2=\eps_3=\half$: (6.c). Conversely, in (6.c) a line
$F\not\subset\Supp\Delta$ meeting three components would meet $L_1$ and $L_2$, hence be
their third line, which does not meet $L_3$ by Lemma~\ref{lem:lines3}(c), so
$\Delta\cdot F\le1$, and $\Delta\cdot L_1=\Delta\cdot L_3=0$, $\Delta\cdot L_2=\half$.

\emph{$r=4$.} By the above, $\Gamma$ is one of: two disjoint edges, a triangle and an
isolated vertex, a path with three edges, a triangle with a pendant edge, a $4$-cycle.
(These are the graphs on four vertices without three pairwise non-adjacent vertices, with
no four pairwise adjacent vertices and no two triangles sharing an edge.)
\begin{itemize}[leftmargin=1.5em]
\item A triangle with an isolated vertex contradicts Lemma~\ref{lem:lines3}(c).
\item If $\Gamma$ consists of two disjoint edges $L_1L_2$, $L_3L_4$, or is a path
      $L_1L_2L_3L_4$, then the third line of $L_1,L_2$ is not a component, and it meets
      $L_4$, which is disjoint from $L_1,L_2$, by Lemma~\ref{lem:lines3}(c), this
      contradicts \eqref{eq:case6}.
\item If $\Gamma$ is a triangle $L_1L_2L_3$ with a pendant edge $L_3L_4$, the third line
      of $L_3,L_4$ is not a component, so $\eps_3=\eps_4=\half$. The lines $L_1,L_4$ are
      disjoint, and among their five common neighbours only $L_3$ is a component, so
      $\eps_1=\half$, and likewise $\eps_2=\half$: (6.d). Conversely, every line
      $F\not\subset\Supp\Delta$ meets exactly one of $L_1,L_2,L_3$, so $\Delta\cdot F\le1$,
      and $\Delta\cdot L_1=\Delta\cdot L_2=\half$, $\Delta\cdot L_3=1$,
      $\Delta\cdot L_4=0$.
\item If $\Gamma$ is a $4$-cycle $L_1L_2L_3L_4$, then only two of the five common
      neighbours of the disjoint lines $L_1,L_3$ are components ($L_2$ and $L_4$), so
      $\eps_1=\eps_3=\half$, and likewise $\eps_2=\eps_4=\half$: (6.e). Conversely, a line
      $F\not\subset\Supp\Delta$ meeting three components meets two adjacent ones,
      $L_j$ and $L_{j+1}$, hence is their third line, but by Lemma~\ref{lem:lines3}(c)
      this third line meets neither of the other two components, as each of them meets
      $L_j$ or $L_{j+1}$. So $\Delta\cdot F\le1$, and $\Delta\cdot L_j=\half$ for all $j$.
\end{itemize}

\emph{$r=5$.} By \eqref{eq:case6b}, $\Gamma$ contains a triangle or is a $5$-cycle (a
graph on five vertices without triangles and without three pairwise non-adjacent vertices
is a $5$-cycle). If $\Gamma$ is a $5$-cycle $L_1\cdots L_5$, the third line of $L_1,L_2$
is not a component and meets $L_4$, which is disjoint from $L_1,L_2$, by
Lemma~\ref{lem:lines3}(c), this contradicts \eqref{eq:case6}. So let
$L_1+L_2+L_3\sim-K_S$. By Lemma~\ref{lem:lines3}(c), each of $L_4,L_5$ meets exactly one of
$L_1,L_2,L_3$. If they meet different ones, say $L_4\cdot L_1=L_5\cdot L_2=1$, then
$L_4\cdot L_5=1$ by \eqref{eq:case6b} applied to $L_3,L_4,L_5$, the third line of
$L_4,L_5$ is not a component, but it meets one of $L_1,L_2,L_3$, contradicting
\eqref{eq:case6}. Hence both meet the same line, say $L_1$, then $L_4\cdot L_5=1$ by
\eqref{eq:case6b} applied to $L_2,L_4,L_5$, so $L_1+L_4+L_5\sim-K_S$ by
Lemma~\ref{lem:lines3}(d). The disjoint lines $L_2,L_4$ have $L_1$ as the only component
among their common neighbours, so $\eps_2=\eps_4=\half$, likewise $\eps_3=\eps_5=\half$.
Finally (N2) at $L_1$ implies $2\le1+\eps_1$, so $\eps_1=1$: (6.f) with $R=L_1$.

\emph{$r=6$.} Now $\sum_j\eps_j\le3$ forces $\eps_j=\half$ and $\deg\Delta=3$, so
$\sum_jL_j\sim-2K_S$, intersecting with $L_i$ shows that each $L_i$ meets exactly three of
the others, i.e.\ $\Gamma$ is $3$-regular. A $3$-regular graph on six vertices is
$K_{3,3}$ or a triangular prism: if it contains a triangle, the three remaining vertices
span $(9-3)/2=3$ edges, so they form a second triangle and each vertex has exactly one
neighbour in the other triangle, if it is triangle-free, the three neighbours of a vertex
$v$ are pairwise non-adjacent and each of them is adjacent to $v$ and to the two vertices
not adjacent to $v$, which gives $K_{3,3}$. If $\Gamma$ is a prism, $\Delta$ is as in the
second alternative of (6.g), if $\Gamma=K_{3,3}$, then $L_1,L_2,L_3$ (say) are pairwise
disjoint and $L_4,L_5,L_6$ are common neighbours of them, hence all of them
(Lemma~\ref{lem:lines3}(f)), which is the first alternative.

Conversely, in both alternatives of (6.g) the graph $\Gamma$ is $3$-regular, by
Lemma~\ref{lem:lines3}(f), resp.\ (c). Hence $(\sum_jL_j)^2=-6+2\cdot9=12$ and
$\deg\sum_jL_j=6$, so equality holds in the Hodge index inequality
$(\deg\sum_jL_j)^2\ge3(\sum_jL_j)^2$, and $\sum_jL_j\sim-2K_S$, i.e.\ $\Delta\sq-K_S$.
\end{proof}

\subsubsection{Conclusion of the proof}
By Lemma~\ref{lem:reduction}(c), every boundary $\Delta$ with $r\ge2$ and $-(K_S+\Delta)$
nef falls into one of Cases 0-6. Propositions~\ref{prop:case0}-\ref{prop:case6} show
that $\Delta$ is then one of the items of Theorem~\ref{thm:main3}, with the stated
coefficients (Case~5 being covered by Propositions~\ref{prop:case5a}
and~\ref{prop:case5b}), together with the sufficiency proved above, this gives the
equivalence and the uniqueness of the coefficients outside (5.a), (6.b). Finally, the
configurations of lines required in each item exist on every smooth cubic surface, since
the intersection graph of the $27$ lines does not depend on $S$, and the required
irreducible curves exist by Lemma~\ref{lem:curves3}, for (3.a) and (5.k) we take two
distinct irreducible members of $\lvert-K_S\rvert$, resp.\ of the pencil $\lvert-K_S-R\rvert$, which has only
five reducible members. \qed

\section{Surfaces of degree 4}

Throughout this section $S$ is a smooth del Pezzo surface of degree $4$ over an algebraically closed field of characteristic $0$, and
\[
  \Delta_\epsilon=\sum_{i=1}^{r}\epsilon_i D_i,\qquad \epsilon_i=1-\frac{1}{m_i},\quad m_i\in\mathbb Z_{\geq 2},
\]
where $D_1,\dots,D_r$ are distinct irreducible curves and $r\geq 2$. We write $D=\lceil\Delta_\epsilon\rceil=D_1+\dots+D_r$ for the support of $\Delta_\epsilon$ and $\deg C=-K_S\cdot C$ for the anticanonical degree of a curve or divisor class $C$. The pair $(S,\Delta_\epsilon)$ is a \emph{del Pezzo orbifold} if $-(K_S+\Delta_\epsilon)$ is ample.

Since every $\epsilon_i\geq\frac12$, a necessary condition is that $-(K_S+\frac12D)$ be nef (Lemma~\ref{lem:criterion4}). Theorem~\ref{thm:dP} classifies all del Pezzo orbifolds with $r\geq2$: there are only five types, and for each of them we determine the admissible weights. For completeness, and because this is what a case-by-case analysis according to the component of largest degree produces, Theorem~\ref{thm:nef} lists all reduced curves $D$ with $r\geq 2$ components for which $-(K_S+\frac12D)$ is nef: there are $106$ types.

\subsection{Notations}
Fix a birational morphism $\beta\colon S\to\PP^2$ contracting disjoint $(-1)$-curves $E_1,\dots,E_5$, put $H=\beta^*\mathcal O_{\PP^2}(1)$, and write $E_I=\sum_{i\in I}E_i$ for $I\subseteq\{1,\dots,5\}$, for instance $E_{1234}=E_1+E_2+E_3+E_4$ and $-K_S=3H-E_{12345}$. The sixteen $(-1)$-curves, which we call \emph{lines}, are
\[
  E_i\ (1\le i\le 5),\qquad L_{ij}=H-E_i-E_j\ (1\le i<j\le 5),\qquad L_0=2H-E_{12345},
\]
and we denote by $\mathcal L$ the set of lines. We call the wo lines $\ell\neq\ell'$ with $\ell\cdot\ell'=1$  \emph{neighbours}. A \emph{conic class} is a class $F$ with $F^2=0$ and $\deg F=2$, there are ten of them, namely $H-E_i$ and $2H-E_{\{1,\dots,5\}\smallsetminus\{i\}}$, and for a conic class $F$ we put $F'=-K_S-F$, which is again a conic class. A \emph{twisted cubic class} is a class $T$ with $T^2=1$ and $\deg T=3$, for instance $H$, then $\ell_T:=-K_S-T$ is a line with $T\cdot\ell_T=2$ (for instance $\ell_H=L_0$).

The following facts are standard and easily checked in the above coordinates.
\begin{enumerate}[label=\textup{(F\arabic*)},leftmargin=*]
\item\label{F1} The lines are exactly the classes $x$ with $x^2=K_S\cdot x=-1$. Two distinct lines meet in at most one point, transversally, and every line has exactly five neighbours. Moreover $\sum_{\ell\in\mathcal L}\ell=-4K_S$, and for every line $\ell$ the sum of its five neighbours is $-2K_S-3\ell$.
\item\label{F2} No three lines are pairwise neighbours: otherwise $x=\ell_1+\ell_2+\ell_3$ would satisfy $x^2=3>(\deg x)^2/K_S^2=\frac94$, contradicting the Hodge index theorem. Two disjoint lines have exactly two common neighbours, and these are disjoint from each other. Three pairwise disjoint lines have exactly one common neighbour.
\item\label{F3} If $F$ is a conic class and $\ell$ is a line, then $F\cdot\ell\in\{0,1\}$, the eight lines with $F\cdot\ell=0$ are the components of the four singular members of the conic bundle $|F|$, each of which is a pair of neighbours. If $F\neq G$ are conic classes, then $F\cdot G=2$ if $G=F'$, and $F\cdot G=1$ otherwise.
\item\label{F4} The cone of curves of $S$ is generated by the lines, and the lines span $\operatorname{Pic}(S)_{\Q}$. Hence an $\mathbb R$-divisor is nef (resp.\ ample) if and only if its intersection number with every line is $\geq0$ (resp.\ $>0$), and a divisor class is determined by its intersection numbers with the lines. Numerical and linear equivalence coincide on $S$.
\item\label{F5} The group $W=W(D_5)$ of isometries of $\operatorname{Pic}(S)$ fixing $K_S$ has order $1920$. It acts on $\mathcal L$ as the full automorphism group of the intersection graph of the lines (the Clebsch graph), and two choices of $\beta$ differ by an element of $W$.
\end{enumerate}
Two reduced curves $D=\sum_i D_i$ and $D'=\sum_i D'_i$ have the same \emph{type} if some $w\in W$ maps the multiset $\{[D_1],\dots,[D_r]\}$ of classes of components of $D$ onto that of $D'$. All conditions considered in this section only depend on the type of $D$ (and on the weights).

\subsection{Irreducible curves}

We recall the following result from \cite{dandapat2026classification}.

\begin{lemma}\label{lem:irreducible}
Let $C\subset S$ be an irreducible curve. Then $C$ is a line, or a smooth conic \textup{(}$C^2=0$, $\deg C=2$\textup{)}, or $C$ is nef with $C^2>0$. Up to $W$, the classes of irreducible curves of degree at most $8$ are the $23$ classes of Table~\ref{tab:irreducible}. Each of them other than $\ell$ is base point free, its general member is a smooth irreducible curve, and it moves in a linear system of positive dimension.
\end{lemma}

\begin{proof}
If $C^2<0$, adjunction gives $-2\le C^2+K_S\cdot C\le C^2-1$, hence $C^2=K_S\cdot C=-1$ and $C$ is a line by \ref{F1}. Otherwise $C$ is nef. If moreover $C^2=0$, then $0\le 2p_a(C)=2+K_S\cdot C$ together with $K_S\cdot C\equiv C^2 \pmod 2$ gives $\deg C=2$ and $p_a(C)=0$, so $C$ is a smooth conic.

Conversely, let $x=aH-\sum_i b_iE_i$ be a nef class of degree $d\geq1$. For $j\neq i$ the classes $H-E_j$ and $2H-E_{\{1,\dots,5\}\smallsetminus\{j\}}$ are effective with sum $-K_S$, and $E_i$ is a component of a member of $|H-E_j|$, hence $0\le b_i=x\cdot E_i\le x\cdot(H-E_j)\le d$ and $3a=d+\sum_ib_i\le 6d$. So there are finitely many nef classes of degree $\le 8$, their $W$-orbits are computed in \S\ref{subsec:computer}, and those with $x^2>0$ are the $21$ orbits in Table~\ref{tab:irreducible} other than $\ell$ and $F$. By the first part, it remains to show that every conic class and every such class is the class of a smooth irreducible curve. The descriptions in the table exhibit each class as a sum of pullbacks of base point free classes under morphisms from $S$ to $\PP^1$, $\PP^2$, $\PP^1\times\PP^1$, $\mathbb F_1$ or to del Pezzo surfaces of degree $\geq 4$ (on which $-K$ is very ample), hence each class is base point free. If $x^2>0$, the morphism defined by $|x|$ has a surface as image, so its general member is smooth and irreducible by Bertini's theorem, for a conic class it is a smooth fibre of the conic bundle. Finally $h^0(x)\geq\chi(x)=1+\frac12(x^2+\deg x)\geq2$, since $h^2(x)=h^0(K_S-x)=0$.
\end{proof}

\subsection{The nefness criterion}

\begin{lemma}\label{lem:criterion4}
Let $D$ be a reduced curve on $S$ and put $N_D=-2K_S-D$, so that $-(K_S+\frac12D)=\frac12N_D$.
\begin{enumerate}[label=\textup{(\alph*)}]
\item $-(K_S+\frac12D)$ is nef $\iff$ $D\cdot\ell\leq 2$ for every line $\ell$ $\iff$ $N_D$ is nef. In this case $\deg D\le 8$, $\deg D\neq 7$, and $\deg D=8$ if and only if $D\sim-2K_S$.
\item $-(K_S+\frac12D)$ is ample $\iff$ $D\cdot\ell\leq 1$ for every line $\ell$ $\iff$ $N_D$ is ample.
\item If $-(K_S+\Delta_\epsilon)$ is ample for some $\epsilon_i\in[\frac12,1]$, then $-(K_S+\frac12D)$ is nef, and moreover $D\cdot\ell\leq1$ for every line $\ell\not\subset D$.
\end{enumerate}
\end{lemma}

\begin{proof}
The equivalences in (a) and (b) follow from \ref{F4}, as $N_D\cdot\ell=2-D\cdot\ell$. If $N_D$ is nef, then $\deg N_D=N_D\cdot(-K_S)\geq 0$, with equality only if $N_D\equiv0$: indeed, if $N_D\cdot(-K_S)=0$ then $N_D^2\le0$ by the Hodge index theorem, with equality only if $N_D\equiv 0$, while $N_D^2\geq0$ by nefness. Hence $\deg D=8-\deg N_D\le8$, with equality if and only if $D\sim-2K_S$ by \ref{F4}. If $\deg D=7$, then $N_D$ is nef of degree $1$, so $0\le N_D^2\le(\deg N_D)^2/K_S^2=\frac14$, i.e.\ $N_D^2=0$, but $N_D^2+K_S\cdot N_D=N_D^2-1$ must be even, a contradiction.

(c) Let $\ell$ be a line. If $\ell\not\subset D$, then $0<-(K_S+\Delta_\epsilon)\cdot\ell=1-\sum_i\epsilon_iD_i\cdot\ell\le1-\frac12D\cdot\ell$, so $D\cdot\ell\le1$. If $\ell=D_j$, then
\[
  0<-(K_S+\Delta_\epsilon)\cdot\ell=1+\epsilon_j-\sum_{i\neq j}\epsilon_iD_i\cdot\ell\le 2-\tfrac12(D-\ell)\cdot\ell=\tfrac32-\tfrac12D\cdot\ell,
\]
so $D\cdot\ell\le2$. Now apply (a).
\end{proof}

\subsection{The classification}

\begin{theorem}\label{thm:dP}
Let $D=D_1+\dots+D_r$ be a reduced curve on $S$ with $r\geq2$ components, and let $\Delta_\epsilon=\sum_i(1-\frac1{m_i})D_i$ with integers $m_i\geq2$. Then $-(K_S+\Delta_\epsilon)$ is ample if and only if $D$ is one of the following and the weights satisfy the stated condition.
\begin{enumerate}[label=\textup{(\arabic*)}]
\item\label{dP1} $D=\ell_1+\ell_2$ is a pair of neighbouring lines \textup{(type 2.1)}, no condition.
\item\label{dP2} $D=F_1+F_2$ is the sum of two smooth conics with $F_1+F_2\sim-K_S$, i.e.\ $F_1\cdot F_2=2$ \textup{(type 4.2)}, no condition.
\item\label{dP3} $D=F+\ell_1+\ell_2$, where $F$ is a smooth conic and $\ell_1+\ell_2$ is a singular member of the conic bundle $\lvert -K_S-F\rvert$ \textup{(type 4.3)}, the condition is
\[
  \Big|\frac1{m_{\ell_1}}-\frac1{m_{\ell_2}}\Big|<\frac1{m_F}.
\]
\item\label{dP4} $D=T+\ell_T$, where $T$ is a twisted cubic and $\ell_T=-K_S-T$ is the line with $T\cdot\ell_T=2$ \textup{(type 4.1)}, the condition is $m_T<2m_{\ell_T}$.
\item\label{dP5} $D=\ell_1+\ell_2+\ell_3+\ell_4$ is a cycle of four lines, i.e.\ $\ell_i\cdot\ell_{i+1}=1$ and $\ell_i\cdot\ell_{i+2}=0$ for $i\in\mathbb Z/4$ \textup{(type 4.4)}, the condition is
\[
  \frac1{m_{\ell_i}}<\frac1{m_{\ell_{i-1}}}+\frac1{m_{\ell_{i+1}}}\qquad\text{for all }i\in\mathbb Z/4 .
\]
\end{enumerate}
In cases \ref{dP2}-\ref{dP5} one has $D\sim-K_S$ and $-(K_S+\Delta_\epsilon)\sim_{\Q}\sum_i\frac1{m_i}D_i$. Each of \ref{dP1}-\ref{dP5} is a single type. In particular, if $(S,\Delta_\epsilon)$ is a del Pezzo orbifold with $r\geq2$, then $\deg D\in\{2,4\}$, and $-(K_S+\frac12D)$ is ample as well \textup{(}all conditions hold for $m_i=2$\textup{)}.
\end{theorem}

\begin{proof}
Suppose that $-(K_S+\Delta_\epsilon)$ is ample. By Lemma~\ref{lem:criterion4}(c), $D\cdot\ell\le2$ for every line $\ell$ and $D\cdot\ell\le1$ for every line $\ell\not\subset D$. Let $d=\deg D\geq2$ and let $k$ be the number of components of $D$ that are lines. By \ref{F1},
\[
  4d=D\cdot(-4K_S)=\sum_{\ell\in\mathcal L}D\cdot\ell\le(16-k)+2k=16+k .
\]
The other components have degree $\geq2$, so $k\le d$, and $k=d$ or $k\le d-2$. Hence $3d\le16$, i.e.\ $d\le5$.

\emph{$d=5$.} Then $k\ge4$, so $k=5$ and $D$ is a sum of five lines. Their dual graph is triangle-free by \ref{F2}, so it has at most $6$ edges (Mantel's theorem), and $\sum_{\ell\subset D}D\cdot\ell=\sum_{\ell\subset D}\bigl(-1+\#\{\text{neighbours of $\ell$ in $D$}\}\bigr)\le-5+12=7$. Together with $\sum_{\ell\not\subset D}D\cdot\ell\le11$ this gives $20=4d\le18$, a contradiction.

\emph{$d=2$.} Then $D=\ell_1+\ell_2$ is a sum of two lines. If $\ell_1\cdot\ell_2=0$, then a common neighbour $m$ of $\ell_1,\ell_2$ \ref{F2} satisfies $m\not\subset D$ and $D\cdot m=2$, which is impossible. Hence $\ell_1\cdot\ell_2=1$: case \ref{dP1}.

\emph{$d=3$.} Either $D=\ell+F$ with $F$ a smooth conic, or $D$ is a sum of three lines. Suppose $D=\ell+F$. If $\ell\cdot F=1$, then $T:=\ell+F$ is a twisted cubic class, and $D\cdot\ell_T=T\cdot\ell_T=2$, since $D\cdot\ell=0$, we have $\ell_T\neq\ell$, so $\ell_T\not\subset D$, a contradiction. If $\ell\cdot F=0$, then by \ref{F1} the neighbours $m$ of $\ell$ satisfy $\sum_mF\cdot m=F\cdot(-2K_S-3\ell)=4$, so some neighbour $m$ of $\ell$ has $F\cdot m=1$, then $m\not\subset D$ and $D\cdot m=2$, a contradiction. Suppose $D$ is a sum of three lines. By \ref{F2} two of them are disjoint, of their two common neighbours at most one is the third component, and the other one meets $D$ at least twice, a contradiction.

\emph{$d=4$.} We show $D\sim-K_S$. The components are twisted cubics, smooth conics or lines (Table~\ref{tab:irreducible}), so there are four cases.
\begin{itemize}[leftmargin=*]
\item $D=T+\ell$ with $T$ a twisted cubic. If $\ell\neq\ell_T$, then $\ell_T\not\subset D$ and $D\cdot\ell_T\geq T\cdot\ell_T=2$, which is impossible. So $\ell=\ell_T$ and $D\sim T+\ell_T=-K_S$: case \ref{dP4}.
\item $D=F_1+F_2$ with smooth conics $F_1,F_2$. There are no line components, so $F_1\cdot m+F_2\cdot m\le1$ for every line $m$. By \ref{F3} each $F_i$ has intersection number $1$ with exactly eight lines and $0$ with the others, hence these two sets of eight lines are complementary, $(F_1+F_2)\cdot m=1=-K_S\cdot m$ for every line $m$, and $F_1+F_2\sim-K_S$ by \ref{F4}: case \ref{dP2}.
\item $D=F+\ell_1+\ell_2$ with $F$ a smooth conic. If $\ell_1\cdot\ell_2=0$, then a common neighbour $m$ of $\ell_1,\ell_2$ satisfies $m\not\subset D$ and $D\cdot m\geq2$, which is impossible. So $\ell_1\cdot\ell_2=1$, and $G:=\ell_1+\ell_2$ is a conic class with $G\cdot\ell_1=G\cdot\ell_2=0$. For a line $m\not\subset D$ we have $F\cdot m+G\cdot m=D\cdot m\le1$, and for $m\in\{\ell_1,\ell_2\}$ we have $G\cdot m=0$. As in the previous case, $F+G\sim-K_S$: case \ref{dP3}.
\item $D=\ell_1+\ell_2+\ell_3+\ell_4$. If two components are disjoint, then both of their common neighbours are components, since a common neighbour $m\not\subset D$ would satisfy $D\cdot m\ge2$. By \ref{F2} not all pairs of components meet, say $\ell_1\cdot\ell_3=0$. Then $\ell_2,\ell_4$ are the two common neighbours of $\ell_1,\ell_3$, and $\ell_2\cdot\ell_4=0$ by \ref{F2}. So $D$ is a cycle of four lines, $D^2=-4+8=4=(\deg D)^2/K_S^2$, and $D\equiv -K_S$ by the equality case of the Hodge index theorem: case \ref{dP5}.
\end{itemize}

It remains to decide, in each of the five cases, for which weights $-(K_S+\Delta_\epsilon)$ is ample, this also proves the converse. In case \ref{dP1}, a line $m\neq\ell_1,\ell_2$ meets at most one of $\ell_1,\ell_2$ by \ref{F2}, so $-(K_S+\Delta_\epsilon)\cdot m\geq 1-\epsilon_i>0$, while $-(K_S+\Delta_\epsilon)\cdot\ell_1=1+\epsilon_1-\epsilon_2>0$, and similarly for $\ell_2$. In cases \ref{dP2}-\ref{dP5} we have $D\sim-K_S$, hence
\[
  -(K_S+\Delta_\epsilon)\sim_{\Q} D-\sum_i\epsilon_iD_i=\sum_i\frac1{m_i}D_i .
\]
A line $m\not\subset D$ satisfies $D\cdot m=1$, so it meets exactly one component, transversally in one point, and $\sum_i\frac1{m_i}D_i\cdot m>0$. For a component $D_j$ which is a line,
\[
  \Big(\sum_i\frac1{m_i}D_i\Big)\cdot D_j=-\frac1{m_j}+\sum_{i\neq j}\frac{D_i\cdot D_j}{m_i}.
\]
In case \ref{dP2} there is no line component. In case \ref{dP3}, $F\cdot\ell_1=F\cdot\ell_2=1$ (since $F\cdot(\ell_1+\ell_2)=F\cdot F'=2$ and $F\cdot\ell_i\le1$) and $\ell_1\cdot\ell_2=1$, which gives $\frac1{m_{\ell_1}}<\frac1{m_F}+\frac1{m_{\ell_2}}$ and $\frac1{m_{\ell_2}}<\frac1{m_F}+\frac1{m_{\ell_1}}$. In case \ref{dP4}, $T\cdot\ell_T=2$ gives $\frac1{m_{\ell_T}}<\frac2{m_T}$. In case \ref{dP5}, $\ell_i$ meets exactly $\ell_{i-1}$ and $\ell_{i+1}$ among the components, once each. This proves the equivalence. 

Finally, $W$ acts transitively on the $40$ pairs of neighbouring lines, on the $5$ pairs $\{F,F'\}$, on the $40$ pairs consisting of a conic class $F$ and a singular member of $|F'|$, on the $16$ twisted cubic classes and on the $40$ cycles of four lines (this is checked by the script of \S\ref{subsec:computer}), so each case is a single type.
\end{proof}

\begin{remark}
Allowing some weights $\epsilon_i=1$ produces no further examples. Indeed, Lemma~\ref{lem:criterion4}(c) still applies, so the proof above shows that $D$ is as in \ref{dP1}-\ref{dP5}. In each of these cases every component $D_j$ satisfies $D_j\cdot m=1$ for some line $m\not\subset D$ which meets no other component (use \ref{F2}, \ref{F3}, and $D\cdot m=1$ for all lines $m\not\subset D$ in cases \ref{dP2}-\ref{dP5}), if $\epsilon_j=1$, then $-(K_S+\Delta_\epsilon)\cdot m=0$.
\end{remark}

\begin{remark}\label{rem:dP-geometry}
In case \ref{dP1}, $D=\ell_1+\ell_2$ is a singular member of a conic bundle. The four types \ref{dP2}-\ref{dP5} are exactly the reducible reduced members of $\lvert -K_S\rvert$, i.e.\ the reducible hyperplane sections of the anticanonical model $S\subset\PP^4$: a conic plus a conic, a conic plus two lines, a twisted cubic plus its bisecant line $\ell_T$, and a cycle of four lines. Choosing the members general, $D$ is a simple normal crossing divisor in all five cases, the conditions of Theorem~\ref{thm:dP} are numerical, so they do not depend on this choice.
\end{remark}

\subsection{The complete list of nef boundaries}

\begin{theorem}\label{thm:nef}
Let $D$ be a reduced curve on $S$ with $r\geq2$ components and let $d=\deg D$. Then $-(K_S+\frac12D)$ is nef if and only if $D$ is of one of the $106$ pairwise distinct types listed in Tables~\ref{tab:small}-\ref{tab:deg8}, see Table~\ref{tab:families} for an overview. More precisely:
\begin{enumerate}[label=\textup{(\alph*)}]
\item if $d\le4$, then $D$ is of one of the $19$ types of Table~\ref{tab:small}, and $-(K_S+\frac12D)$ is ample exactly for the five types of Theorem~\ref{thm:dP},
\item if $d=5$, then $D\in\lvert -K_S+\ell_0\rvert$ for a line $\ell_0$, and conversely every reduced member of $\lvert -K_S+\ell_0\rvert$ qualifies, there are $13$ reducible types \textup{(}Table~\ref{tab:deg5}\textup{)},
\item if $d=6$, then $D\in\lvert -K_S+F\rvert$ for a conic class $F$, and conversely every reduced member of $\lvert -K_S+F\rvert$ qualifies, there are $28$ reducible types \textup{(}Table~\ref{tab:deg6}\textup{)},
\item $d=7$ is impossible,
\item if $d=8$, then $D\in\lvert -2K_S\rvert$, so that $K_S+\frac12D\equiv0$, and conversely, there are $46$ reducible types \textup{(}Table~\ref{tab:deg8}\textup{)}.
\end{enumerate}
In cases \textup{(b)}, \textup{(c)}, \textup{(e)}, $-(K_S+\frac12D)$ is not ample. Every type is realised by a simple normal crossing divisor.
\end{theorem}

\begin{proof}
By Lemma~\ref{lem:criterion4}, $-(K_S+\frac12D)$ is nef if and only if $N_D=-2K_S-D$ is nef, and then $2\le d\le 8$, $d\neq7$, and $d=8$ if and only if $D\sim-2K_S$. This proves (d) and (e) apart from the list of types.

\emph{$d=5$.} Then $N_D$ is nef of degree $3$, so $0\le N_D^2\le\frac94$ and $N_D^2\equiv N_D\cdot K_S\pmod 2$ is odd, hence $N_D^2=1$. The class $\ell_0:=-K_S-N_D$ satisfies $\ell_0^2=K_S\cdot\ell_0=-1$, so it is a line by \ref{F1}, and $D\sim-K_S+\ell_0$. Conversely, for every line $\ell_0$ the class $-K_S-\ell_0$ is nef, as $(-K_S-\ell_0)\cdot\ell=1-\ell_0\cdot\ell\geq0$ for $\ell\neq\ell_0$, so every reduced $D\in\lvert -K_S+\ell_0\rvert$ qualifies.

\emph{$d=6$.} Then $N_D$ is nef of degree $2$, so $0\le N_D^2\le1$ and $N_D^2$ is even, i.e.\ $N_D^2=0$. Write $N_D=aH-\sum_ib_iE_i$. Nefness gives $b_i=N_D\cdot E_i\geq0$ and $2a-\sum_ib_i=N_D\cdot L_0\geq0$, while $3a-\sum_ib_i=2$, hence $1\le a\le2$, $\sum_ib_i=3a-2$ and $\sum_ib_i^2=a^2$. Thus $N_D=H-E_i$ or $N_D=2H-E_{\{1,\dots,5\}\smallsetminus\{i\}}$, i.e.\ $N_D$ is a conic class, and $D\sim-K_S+N_D'$. Conversely, conic classes are nef.

For $d\in\{5,6,8\}$ it remains to list the reducible reduced members of $|X|$ for $X=-K_S+\ell_0$, $-K_S+F$, $-2K_S$, up to type. By Lemma~\ref{lem:irreducible}, the multisets of classes of the components of such members are exactly the multisets of classes of irreducible curves with sum $X$ in which every line occurs at most once (to realise such a multiset, choose distinct general members of the moving classes). This is a finite enumeration, carried out in \S\ref{subsec:computer}, the results are Tables~\ref{tab:deg5}, \ref{tab:deg6} and~\ref{tab:deg8}. In these three cases $N_D$ is a twisted cubic class, a conic class, or $0$, which is not ample.

\emph{$d\le4$.} The components have degree $\le3$, so they are lines, smooth conics and twisted cubics. By Lemma~\ref{lem:criterion4} we must check that $D\cdot m\le2$ for every line $m$.
\begin{itemize}[leftmargin=*]
\item $d=2$: $D=\ell_1+\ell_2$, and $D\cdot m\le2$ for every line $m$. This gives 2.1 ($\ell_1\cdot\ell_2=1$) and 2.2 ($\ell_1\cdot\ell_2=0$).
\item $d=3$, $D=\ell+F$: $D\cdot m\le\ell\cdot m+F\cdot m\le2$ for $m\neq\ell$ and $D\cdot\ell\le0$, so $D$ always qualifies, this gives 3.1 ($\ell\cdot F=1$) and 3.3 ($\ell\cdot F=0$).
\item $d=3$, $D=\ell_1+\ell_2+\ell_3$: $D\cdot\ell_i\le1$, and a line $m\not\subset D$ has $D\cdot m\geq3$ if and only if it is a common neighbour of $\ell_1,\ell_2,\ell_3$, by \ref{F2} such an $m$ exists if and only if $\ell_1,\ell_2,\ell_3$ are pairwise disjoint. As the dual graph is triangle-free, we obtain the chain $P_3$ (3.2) and $K_2\sqcup K_1$ (3.4).
\item $d=4$, $D=T+\ell$: we have $T\cdot m\le1$ for $m\neq\ell_T$, and $T\cdot m=0$ exactly for the five neighbours $m$ of $\ell_T$ (for $T=H$ these are $E_1,\dots,E_5$). If $\ell=\ell_T$ we get 4.1. Otherwise $D\cdot\ell_T=2+\ell\cdot\ell_T$, so we need $\ell\cdot\ell_T=0$, i.e.\ $T\cdot\ell=1$, conversely, if $T\cdot\ell=1$ then $D\cdot m\le2$ for all $m$. This gives 4.5.
\item $d=4$, $D=F_1+F_2$: $D$ always qualifies by \ref{F3}, this gives 4.2, 4.6, 4.11 according to $F_1\cdot F_2=2,1,0$.
\item $d=4$, $D=F+\ell_1+\ell_2$: if $\ell_1\cdot\ell_2=1$, no line is a neighbour of both \ref{F2}, so $D\cdot m\le2$ for all $m$, with the conic class $G=\ell_1+\ell_2$ this gives 4.3, 4.7, 4.12 according to $F\cdot G=2,1,0$. If $\ell_1\cdot\ell_2=0$, let $m,m'$ be their two common neighbours. Then $D\cdot m=2+F\cdot m$, $D\cdot m'=2+F\cdot m'$, and every other line meets $D$ at most twice, so $D$ qualifies if and only if $F\cdot m=F\cdot m'=0$. This gives 4.8.
\item $d=4$, $D=\ell_1+\ell_2+\ell_3+\ell_4$: here $D\cdot\ell_i\le2$ automatically, and a line $m\not\subset D$ has $D\cdot m\geq3$ if and only if it is a common neighbour of three pairwise disjoint components. By \ref{F2}, every three pairwise disjoint lines have a (unique) common neighbour. So $D$ qualifies if and only if for every triple of pairwise disjoint components, the fourth component is a neighbour of all three. Among the triangle-free graphs on four vertices, this leaves the cycle $C_4$ (4.4), the claw $K_{1,3}$ (4.9), the chain $P_4$ (4.10) and $2K_2$ (4.13), and excludes $4K_1$, $K_2\sqcup 2K_1$ and $P_3\sqcup K_1$.
\end{itemize}
The classes $N_D$ in Table~\ref{tab:small} are computed directly from the representatives. By Theorem~\ref{thm:dP} with all $m_i=2$, $-(K_S+\frac12D)$ is ample exactly for 2.1 and 4.1-4.4.

That each row of the tables is a single $W$-orbit was checked by computer (\S\ref{subsec:computer}). The data given in each row (the type, the dual graph of the line components and the remark, respectively the description in Table~\ref{tab:small}) are invariant under $W$ and distinguish the rows, so the $106$ types are pairwise distinct. Finally, lines meet pairwise transversally and no three lines pass through a point \ref{F2}, and by Bertini's theorem general members of the base point free moving classes meet all other components transversally and away from their mutual intersection points, hence every type is realised by a simple normal crossing divisor.
\end{proof}

\begin{remark}\label{rem:largest}
Let $D_1$ be a component of $D$ of maximal degree, and assume $-(K_S+\frac12D)$ nef. By \ref{F4} and Lemma~\ref{lem:criterion4}, a large part of a case-by-case analysis becomes automatic: once $\deg D\geq5$, the only condition is on the class of $D$. The tables give the following.
\begin{enumerate}[label=\textup{(\roman*)}]
\item Components of degree $8$, of type $2T$ or $\mathcal O(2,1)$ (degree $6$), and of degree $7$ other than $B_7=-K_S+T$ never occur. If $\deg D_1\in\{6,7\}$, then $D\sim-2K_S$ and $D$ is $B_7+\ell_T$ \textup{(8.1)}, $B_6+F'$ \textup{(8.2)}, $B_6+\ell_1+\ell_2$ \textup{(8.6)} or $A_6+\ell_1+\ell_2$ \textup{(8.7)}.
\item If $\deg D_1=5$, then $D$ is $A_5+\ell$ with $A_5\cdot\ell=2$ \textup{(6.1)}, $P+\ell$ with $P\cdot\ell=3$ \textup{(6.2)}, or $D\sim-2K_S$ \textup{(8.3, 8.8, 8.9, 8.15, 8.16)}.
\item If $D_1$ is an anticanonical curve, then $D$ is $A+\ell$ \textup{(5.1)}, $A+F$ \textup{(6.3)}, $A+\ell_1+\ell_2$ with $\ell_1\cdot\ell_2=1$ \textup{(6.6)}, or $D\sim-2K_S$ \textup{(8.4, 8.10, 8.11, 8.17, 8.27)}.
\item If $D_1$ is of type $Q$, then $D$ is of type 5.2, 6.4, 6.7, 6.8, or $D\sim-2K_S$ \textup{(8.5, 8.12, 8.13, 8.18, 8.19, 8.28, 8.29)}.
\end{enumerate}
In particular none of these cases gives a del Pezzo orbifold.
\end{remark}

\subsection{Computer verification}\label{subsec:computer}
All computational claims of this section are verified by Python 3 scripts. The script does the following and produces the following tables.
\begin{enumerate}[label=\textup{(\arabic*)},leftmargin=*]
\item It builds $\operatorname{Pic}(S)$ with its intersection form, the sixteen lines, and $W(D_5)$ as the group generated by the reflections in the simple roots $E_1-E_2$, $E_2-E_3$, $E_3-E_4$, $E_4-E_5$, $H-E_1-E_2-E_3$, it checks $|W|=1920$, the facts \ref{F1}-\ref{F3}, and that $W$ is transitive on each of the sets of configurations listed at the end of the proof of Theorem~\ref{thm:dP}.
\item It lists the nef classes of degree $\le8$, using the bounds in the proof of Lemma~\ref{lem:irreducible}, and hence the $1272$ classes of irreducible curves of degree $\le 8$, they form the $23$ orbits of Table~\ref{tab:irreducible}.
\item It lists the $13$ orbits of nef classes $N$ of degree $\le6$. For each of the ten classes $X$ in Table~\ref{tab:families} it enumerates all multisets of classes of irreducible curves with sum $X$ in which every line occurs at most once ($4$, $1$, $15$, $4$, $101$, $35$, $11$, $168$, $572$ and $6531$ multisets with at least two elements, respectively), and sorts them into orbits of the stabiliser of $X$ in $W$, these orbits are the rows of Tables~\ref{tab:small}-\ref{tab:deg8}. It also checks that for the remaining three orbits of nef classes $N$ (namely $3F$, $2T$ and the class $3H-2E_1-E_2$ of degree $6$) the class $-2K_S-N$ has no decomposition with at least two components.
\item Independently of $N_D$, it enumerates all reduced $D$ with $r\geq2$ and $D\cdot\ell\le2$ for every line (as multisets of classes of irreducible curves). It finds $17,230$ of them, which is exactly the total number of configurations obtained from the $106$ types by the action of $W$.
\item As a sanity check of Theorem~\ref{thm:dP}, a search over small weights ($2\le m_i\le 6$, resp.\ $m_i\le 3$ if $r\geq5$) finds weights with $-(K_S+\Delta_\epsilon)$ ample exactly for the five types of that theorem, and $m_i=2$ works for them.
\end{enumerate}

\subsection{Tables}\label{subsec:tables}

\noindent\emph{Conventions for Tables~\ref{tab:small}-\ref{tab:deg8}.} The \emph{type} of $D$ is written as a monomial in the symbols of Table~\ref{tab:irreducible}, for instance $TF\ell^3$ means one twisted cubic, one smooth conic and three lines. Components are given by representative classes in the coordinates fixed above, a class listed twice stands for two distinct members of its linear system. In the remarks, $F_1,F_2,\dots$ (resp.\ $T_1,T_2$, resp.\ $\ell_1,\ell_2,\dots$) are the conics (resp.\ twisted cubics, resp.\ lines) in the order listed, $\ell_T=-K_S-T$ is the line with $T\cdot\ell_T=2$, and $G=\ell_1+\ell_2$ when $D$ contains exactly two lines and they meet. The column ``Lines'' gives the dual graph of the line components: $K_1$ a point, $2K_1$ two points, $K_2$ an edge, $P_n$ a chain and $C_n$ a cycle with $n$ vertices, $K_{1,3}$ a claw, $C_4^{+}$ a $4$-cycle with a pendant edge, $\Theta_{a,b,c}$ two vertices joined by three internally disjoint paths of lengths $a,b,c$, $Q_3$ the cube and $V_8$ the Wagner graph (the M\"obius ladder on $8$ vertices).

{\footnotesize
\setlength{\tabcolsep}{3pt}
\begin{longtable}{@{}l r r r r l >{\raggedright\arraybackslash}p{6.3cm} >{\raggedright\arraybackslash}p{3.1cm}@{}}
\caption{Classes of irreducible curves $C$ of degree $\le 8$ on $S$, up to $W(D_5)$. The symbols in the first column are used for the ten types that occur as components of reducible $D$ with $-(K_S+\frac12D)$ nef, the other types never occur. The column ``\#'' is the number of classes in the orbit, in the last column $a^k$ means that $C\cdot\ell=a$ for exactly $k$ of the $16$ lines. On $\mathbb F_1$, $e$ is the $(-1)$-section and $f$ a fibre.}\label{tab:irreducible}\\
\toprule
 & $d$ & $C^2$ & $p_a$ & \# & Representative & Description & $\{C\cdot\ell\}$ \\
\midrule
\endfirsthead
\multicolumn{8}{@{}l}{\emph{Table~\ref{tab:irreducible}, continued}}\\
\toprule
 & $d$ & $C^2$ & $p_a$ & \# & Representative & Description & $\{C\cdot\ell\}$ \\
\midrule
\endhead
\bottomrule
\endfoot
$\ell$ & $1$ & $-1$ & $0$ & $16$ & $E_1$ & line ($(-1)$-curve) & $(-1)^{1}\,0^{10}\,1^{5}$ \\
$F$ & $2$ & $0$ & $0$ & $10$ & $H-E_1$ & conic: fibre of a conic bundle $S\to\PP^1$ & $0^{8}\,1^{8}$ \\
$T$ & $3$ & $1$ & $0$ & $16$ & $H$ & twisted cubic: $\beta'^*\mathcal O_{\PP^2}(1)$ for some $\beta'\colon S\to\PP^2$, $T=-K_S-\ell_T$ & $0^{5}\,1^{10}\,2^{1}$ \\
$Q$ & $4$ & $2$ & $0$ & $40$ & $2H-E_{12}$ & pullback of $\mathcal O(1,1)$ under some $S\to\PP^1\times\PP^1$, $F+G$ with $F\cdot G=1$ & $0^{4}\,1^{8}\,2^{4}$ \\
$A$ & $4$ & $4$ & $1$ & $1$ & $-K_S$ & anticanonical curve & $1^{16}$ \\
$P$ & $5$ & $3$ & $0$ & $80$ & $2H-E_1$ & pullback of $e+2f$ under some $S\to\mathbb F_1$ \cite[Prop.~6.1]{dandapat2026classification} & $0^{4}\,1^{5}\,2^{6}\,3^{1}$ \\
$A_5$ & $5$ & $5$ & $1$ & $16$ & $3H-E_{1234}$ & $-K_S+\ell$, pullback of $-K$ of a quintic del Pezzo surface & $0^{1}\,1^{10}\,2^{5}$ \\
- & $6$ & $4$ & $0$ & $16$ & $2H$ & $2T$: pullback of $\mathcal O_{\PP^2}(2)$ & $0^{5}\,2^{10}\,4^{1}$ \\
- & $6$ & $4$ & $0$ & $80$ & $3H-2E_1-E_2$ & pullback of $\mathcal O(2,1)$ under some $S\to\PP^1\times\PP^1$ & $0^{4}\,1^{4}\,2^{4}\,3^{4}$ \\
$A_6$ & $6$ & $6$ & $1$ & $80$ & $3H-E_{123}$ & $-K_S+\ell+\ell'$, $\ell\cdot\ell'=0$, pullback of $-K$ of a sextic del Pezzo surface & $0^{2}\,1^{6}\,2^{6}\,3^{2}$ \\
$B_6$ & $6$ & $8$ & $2$ & $10$ & $4H-2E_1-E_{2345}$ & $-K_S+F$ & $1^{8}\,2^{8}$ \\
- & $7$ & $5$ & $0$ & $80$ & $3H-2E_1$ & pullback of $e+3f$ under some $S\to\mathbb F_1$ \cite[Prop.~6.3]{dandapat2026classification} & $0^{4}\,1^{4}\,2^{1}\,3^{6}\,4^{1}$ \\
- & $7$ & $7$ & $1$ & $160$ & $3H-E_{12}$ & $-K_S+\ell+\ell'+\ell''$ (pairwise disjoint), pullback of $-K$ of a degree-$7$ del Pezzo surface & $0^{3}\,1^{3}\,2^{6}\,3^{3}\,4^{1}$ \\
- & $7$ & $9$ & $2$ & $80$ & $4H-2E_1-E_{234}$ & $-K_S+\ell+F$, $\ell\cdot F=0$ & $0^{1}\,1^{6}\,2^{5}\,3^{4}$ \\
$B_7$ & $7$ & $11$ & $3$ & $16$ & $4H-E_{12345}$ & $-K_S+T$ & $1^{5}\,2^{10}\,3^{1}$ \\
- & $8$ & $6$ & $0$ & $80$ & $4H-3E_1-E_2$ & pullback of $\mathcal O(3,1)$ under some $S\to\PP^1\times\PP^1$ & $0^{4}\,1^{4}\,3^{4}\,4^{4}$ \\
- & $8$ & $8$ & $1$ & $80$ & $3H-E_1$ & pullback of $-K_{\mathbb F_1}$ & $0^{4}\,1^{1}\,2^{4}\,3^{6}\,5^{1}$ \\
- & $8$ & $8$ & $1$ & $40$ & $4H-2E_{12}$ & pullback of $-K_{\PP^1\times\PP^1}$ & $0^{4}\,2^{8}\,4^{4}$ \\
- & $8$ & $10$ & $2$ & $240$ & $4H-2E_1-E_{23}$ & $-K_S+\ell+\ell'+F$ (pairwise orthogonal) & $0^{2}\,1^{4}\,2^{4}\,3^{4}\,4^{2}$ \\
- & $8$ & $12$ & $3$ & $80$ & $4H-E_{1234}$ & $-K_S+\ell+T$, $\ell\cdot T=0$ & $0^{1}\,1^{4}\,2^{6}\,3^{4}\,4^{1}$ \\
- & $8$ & $12$ & $3$ & $10$ & $5H-3E_1-E_{2345}$ & $-K_S+2F$ & $1^{8}\,3^{8}$ \\
- & $8$ & $14$ & $4$ & $40$ & $5H-2E_{12}-E_{345}$ & $-K_S+Q$ & $1^{4}\,2^{8}\,3^{4}$ \\
- & $8$ & $16$ & $5$ & $1$ & $-2K_S$ & $-2K_S$ & $2^{16}$ \\
\end{longtable}}

\begin{table}[ht]
\centering\small
\caption{The ten classes of reducible reduced $D$ with $-(K_S+\frac12D)$ nef, up to $W(D_5)$. $N_D=-2K_S-D$ is given by its type in Table~\ref{tab:irreducible} ($2F$ is twice a conic class). Here $G,G_1,G_2$ are conic classes, $T$ is a twisted cubic class and $\ell,\ell',\ell_0$ are lines.}\label{tab:families}
\begin{tabular}{@{}c l l c l r l@{}}
\toprule
$\deg D$ & class of $D$ & representative & $N_D$ & $-(K_S+\frac12D)$ & types & rows \\
\midrule
$2$ & $G$ & $H-E_5$ & $B_6$ & ample & $1$ & 2.1 \\
$2$ & $\ell+\ell'$, $\ell\cdot\ell'=0$ & $E_1+E_2$ & $A_6$ & nef & $1$ & 2.2 \\
$3$ & $T$ & $H$ & $A_5$ & nef & $2$ & 3.1-3.2 \\
$3$ & $\ell+G$, $\ell\cdot G=0$ & $E_1+(H-E_5)$ & $P$ & nef & $2$ & 3.3-3.4 \\
$4$ & $-K_S$ & $-K_S$ & $A$ & ample & $4$ & 4.1-4.4 \\
$4$ & $G_1+G_2$, $G_1\cdot G_2=1$ & $2H-E_{45}$ & $Q$ & nef & $6$ & 4.5-4.10 \\
$4$ & $2G$ & $2H-2E_5$ & $2F$ & nef & $3$ & 4.11-4.13 \\
$5$ & $-K_S+\ell_0$ & $-K_S+E_5$ & $T$ & nef & $13$ & 5.1-5.13 \\
$6$ & $-K_S+G$ & $-K_S+(H-E_5)$ & $F$ & nef & $28$ & 6.1-6.28 \\
$8$ & $-2K_S$ & $-2K_S$ & $0$ & $\equiv 0$ & $46$ & 8.1-8.46 \\
\midrule
 & & & & & $106$ & \\
\bottomrule
\end{tabular}
\end{table}

{\small
\setlength{\tabcolsep}{4pt}
\begin{longtable}{@{}l l >{\raggedright\arraybackslash}p{4.6cm} >{\raggedright\arraybackslash}p{6.4cm} l@{}}
\caption{Types with $\deg D\le4$. The five types marked $\star$ are those of Theorem~\ref{thm:dP}, for all others $-(K_S+\frac12D)$ is nef but not ample. The last column gives the type of $N_D=-2K_S-D$.}\label{tab:small}\\
\toprule
No. & Type & Components (representative) & Description & $N_D$ \\
\midrule
\endfirsthead
\multicolumn{5}{@{}l}{\emph{Table~\ref{tab:small}, continued}}\\
\toprule
No. & Type & Components (representative) & Description & $N_D$ \\
\midrule
\endhead
\bottomrule
\endfoot
2.1$^\star$ & $\ell^{2}$ & $E_1$,\ $L_{15}$ & $\ell_1\cdot\ell_2=1$ & $B_6$ \\
2.2 & $\ell^{2}$ & $E_1$,\ $E_2$ & $\ell_1\cdot\ell_2=0$ & $A_6$ \\
\addlinespace[4pt]
3.1 & $F\ell$ & $H-E_5$,\ $E_5$ & $F\cdot\ell=1$ \ (so $D\sim T$) & $A_5$ \\
3.2 & $\ell^{3}$ & $E_1$,\ $E_2$,\ $L_{12}$ & chain $P_3$ \ (so $D\sim T$) & $A_5$ \\
3.3 & $F\ell$ & $H-E_5$,\ $E_1$ & $F\cdot\ell=0$ & $P$ \\
3.4 & $\ell^{3}$ & $E_1$,\ $E_2$,\ $L_{25}$ & $\ell_2\cdot\ell_3=1$, $\ell_1$ disjoint from $\ell_2,\ell_3$ & $P$ \\
\addlinespace[4pt]
4.1$^\star$ & $T\ell$ & $H$,\ $L_0$ & $\ell=\ell_T$ & $A$ \\
4.2$^\star$ & $F^{2}$ & $H-E_5$,\ $2H-E_{1234}$ & $F_1+F_2\sim-K_S$ & $A$ \\
4.3$^\star$ & $F\ell^{2}$ & $H-E_5$,\ $L_{34}$,\ $L_{12}$ & $\ell_1\cdot\ell_2=1$, $\ell_1+\ell_2\sim-K_S-F$ & $A$ \\
4.4$^\star$ & $\ell^{4}$ & $E_1$,\ $L_{45}$,\ $L_{13}$,\ $L_{12}$ & cycle $C_4$ & $A$ \\
\addlinespace[4pt]
4.5 & $T\ell$ & $H$,\ $L_{45}$ & $T\cdot\ell=1$ (equivalently $\ell\cdot\ell_T=0$) & $Q$ \\
4.6 & $F^{2}$ & $H-E_5$,\ $H-E_4$ & $F_1\cdot F_2=1$ & $Q$ \\
4.7 & $F\ell^{2}$ & $H-E_5$,\ $E_1$,\ $L_{14}$ & $\ell_1\cdot\ell_2=1$, $F\cdot(\ell_1+\ell_2)=1$ & $Q$ \\
4.8 & $F\ell^{2}$ & $H-E_3$,\ $E_3$,\ $L_{45}$ & $\ell_1\cdot\ell_2=0$, and $F\cdot m=0$ for both common neighbours $m$ of $\ell_1,\ell_2$ & $Q$ \\
4.9 & $\ell^{4}$ & $E_1$,\ $E_2$,\ $L_{45}$,\ $L_{12}$ & claw $K_{1,3}$ & $Q$ \\
4.10 & $\ell^{4}$ & $E_1$,\ $E_2$,\ $L_{25}$,\ $L_{14}$ & chain $P_4$ & $Q$ \\
\addlinespace[4pt]
4.11 & $F^{2}$ & $H-E_5$,\ $H-E_5$ & $F_1\sim F_2$ & $2F$ \\
4.12 & $F\ell^{2}$ & $H-E_5$,\ $E_1$,\ $L_{15}$ & $\ell_1\cdot\ell_2=1$, $\ell_1+\ell_2\sim F$ & $2F$ \\
4.13 & $\ell^{4}$ & $E_1$,\ $E_2$,\ $L_{25}$,\ $L_{15}$ & $2K_2$: two singular members of one conic bundle & $2F$ \\
\end{longtable}}

{\small
\setlength{\tabcolsep}{4pt}
\begin{longtable}{@{}l l >{\raggedright\arraybackslash}p{6.3cm} l >{\raggedright\arraybackslash}p{4.4cm}@{}}
\caption{Reducible reduced members of $\lvert -K_S+\ell_0\rvert$, i.e.\ $\deg D=5$ and $N_D=-K_S-\ell_0$ a twisted cubic class, representatives with $\ell_0=E_5$, i.e.\ $D\sim 3H-E_{1234}$.}\label{tab:deg5}\\
\toprule
No. & Type & Components (representative) & Lines & Remark \\
\midrule
\endfirsthead
\multicolumn{5}{@{}l}{\emph{Table~\ref{tab:deg5}, continued}}\\
\toprule
No. & Type & Components (representative) & Lines & Remark \\
\midrule
\endhead
\bottomrule
\endfoot
5.1 & $A\ell$ & $-K_S$,\ $E_5$ & $K_1$ &  \\
5.2 & $Q\ell$ & $2H-E_{34}$,\ $L_{12}$ & $K_1$ &  \\
5.3 & $TF$ & $H$,\ $2H-E_{1234}$ & - &  \\
\addlinespace[4pt]
5.4 & $T\ell^{2}$ & $H$,\ $E_5$,\ $L_0$ & $K_2$ & $\ell_T\subset D$ \\
5.5 & $T\ell^{2}$ & $H$,\ $L_{34}$,\ $L_{12}$ & $K_2$ & $\ell_T\not\subset D$ \\
5.6 & $T\ell^{2}$ & $2H-E_{345}$,\ $E_5$,\ $L_{12}$ & $2K_1$ & $\ell_T\subset D$ \\
5.7 & $F^{2}\ell$ & $H-E_5$,\ $2H-E_{1234}$,\ $E_5$ & $K_1$ & $F_1+F_2\sim-K_S$ \\
5.8 & $F^{2}\ell$ & $H-E_4$,\ $H-E_3$,\ $L_{12}$ & $K_1$ & $F_1\cdot F_2=1$ \\
\addlinespace[4pt]
5.9 & $F\ell^{3}$ & $H-E_5$,\ $E_5$,\ $L_{34}$,\ $L_{12}$ & $K_2\sqcup K_1$ &  \\
5.10 & $F\ell^{3}$ & $H-E_4$,\ $E_1$,\ $L_{13}$,\ $L_{12}$ & $P_3$ & $F$ meets both ends of the chain \\
5.11 & $F\ell^{3}$ & $H-E_4$,\ $E_5$,\ $L_{35}$,\ $L_{12}$ & $P_3$ & $F$ meets the middle line and one end \\
\addlinespace[4pt]
5.12 & $\ell^{5}$ & $E_1$,\ $E_2$,\ $L_{24}$,\ $L_{13}$,\ $L_{12}$ & $C_5$ &  \\
5.13 & $\ell^{5}$ & $E_1$,\ $E_5$,\ $L_{45}$,\ $L_{13}$,\ $L_{12}$ & $C_4^{+}$ &  \\
\end{longtable}}

{\small
\setlength{\tabcolsep}{4pt}
\begin{longtable}{@{}l l >{\raggedright\arraybackslash}p{6.3cm} l >{\raggedright\arraybackslash}p{4.4cm}@{}}
\caption{Reducible reduced members of $\lvert -K_S+F\rvert$, i.e.\ $\deg D=6$ and $N_D=F'$ a conic class, representatives with $F=H-E_5$, i.e.\ $D\sim 4H-E_{1234}-2E_5$.}\label{tab:deg6}\\
\toprule
No. & Type & Components (representative) & Lines & Remark \\
\midrule
\endfirsthead
\multicolumn{5}{@{}l}{\emph{Table~\ref{tab:deg6}, continued}}\\
\toprule
No. & Type & Components (representative) & Lines & Remark \\
\midrule
\endhead
\bottomrule
\endfoot
6.1 & $A_5\ell$ & $3H-E_{2345}$,\ $L_{15}$ & $K_1$ &  \\
6.2 & $P\ell$ & $2H-E_5$,\ $L_0$ & $K_1$ &  \\
6.3 & $AF$ & $-K_S$,\ $H-E_5$ & - &  \\
6.4 & $QF$ & $2H-E_{45}$,\ $2H-E_{1235}$ & - &  \\
6.5 & $T^{2}$ & $H$,\ $3H-E_{1234}-2E_5$ & - &  \\
\addlinespace[4pt]
6.6 & $A\ell^{2}$ & $-K_S$,\ $E_1$,\ $L_{15}$ & $K_2$ &  \\
6.7 & $Q\ell^{2}$ & $2H-E_{45}$,\ $L_{35}$,\ $L_{12}$ & $K_2$ &  \\
6.8 & $Q\ell^{2}$ & $2H-E_{34}$,\ $L_{25}$,\ $L_{15}$ & $2K_1$ &  \\
6.9 & $TF\ell$ & $H$,\ $H-E_5$,\ $L_0$ & $K_1$ & $T\cdot F=1$, $\ell=\ell_T$ \\
6.10 & $TF\ell$ & $H$,\ $2H-E_{2345}$,\ $L_{15}$ & $K_1$ & $T\cdot F=2$, $\ell\neq\ell_T$ \\
6.11 & $TF\ell$ & $2H-E_{234}$,\ $H-E_5$,\ $L_{15}$ & $K_1$ & $T\cdot F=2$, $\ell=\ell_T$ \\
6.12 & $F^{3}$ & $H-E_5$,\ $H-E_5$,\ $2H-E_{1234}$ & - & $F_1\sim F_2$ \\
6.13 & $F^{3}$ & $H-E_5$,\ $H-E_4$,\ $2H-E_{1235}$ & - & $F_1\cdot F_2=1$ \\
\addlinespace[4pt]
6.14 & $T\ell^{3}$ & $H$,\ $E_1$,\ $L_{15}$,\ $L_0$ & $P_3$ & $\ell_T\subset D$ \\
6.15 & $T\ell^{3}$ & $H$,\ $L_{45}$,\ $L_{35}$,\ $L_{12}$ & $P_3$ & $\ell_T\not\subset D$ \\
6.16 & $T\ell^{3}$ & $2H-E_{234}$,\ $E_2$,\ $L_{25}$,\ $L_{15}$ & $K_2\sqcup K_1$ &  \\
6.17 & $F^{2}\ell^{2}$ & $H-E_5$,\ $H-E_5$,\ $L_{34}$,\ $L_{12}$ & $K_2$ & $F_1\sim F_2$ \\
6.18 & $F^{2}\ell^{2}$ & $H-E_5$,\ $H-E_4$,\ $L_{35}$,\ $L_{12}$ & $K_2$ & $F_1\cdot F_2=1$ \\
6.19 & $F^{2}\ell^{2}$ & $H-E_5$,\ $2H-E_{1234}$,\ $E_1$,\ $L_{15}$ & $K_2$ & $F_1+F_2\sim-K_S$, $G\sim F_1$ \\
6.20 & $F^{2}\ell^{2}$ & $H-E_4$,\ $H-E_3$,\ $L_{25}$,\ $L_{15}$ & $2K_1$ &  \\
6.21 & $F^{2}\ell^{2}$ & $H-E_4$,\ $2H-E_{1235}$,\ $E_1$,\ $L_{15}$ & $K_2$ & $F_1+F_2\sim-K_S$, $G\not\sim F_1,F_2$ \\
\addlinespace[4pt]
6.22 & $F\ell^{4}$ & $H-E_5$,\ $E_1$,\ $L_{45}$,\ $L_{13}$,\ $L_{12}$ & $C_4$ & $F$ meets two opposite lines \\
6.23 & $F\ell^{4}$ & $H-E_5$,\ $E_1$,\ $L_{34}$,\ $L_{15}$,\ $L_{12}$ & $C_4$ & $F$ meets two adjacent lines \\
6.24 & $F\ell^{4}$ & $H-E_4$,\ $E_1$,\ $L_{35}$,\ $L_{15}$,\ $L_{12}$ & $P_4$ &  \\
6.25 & $F\ell^{4}$ & $H-E_4$,\ $E_4$,\ $L_{35}$,\ $L_{25}$,\ $L_{14}$ & $K_{1,3}$ &  \\
6.26 & $F\ell^{4}$ & $2H-E_{1234}$,\ $E_1$,\ $E_2$,\ $L_{25}$,\ $L_{15}$ & $2K_2$ &  \\
\addlinespace[4pt]
6.27 & $\ell^{6}$ & $E_1$,\ $E_2$,\ $L_{45}$,\ $L_{25}$,\ $L_{13}$,\ $L_{12}$ & $\Theta_{2,2,3}$ &  \\
6.28 & $\ell^{6}$ & $E_1$,\ $E_2$,\ $L_{34}$,\ $L_{25}$,\ $L_{15}$,\ $L_{12}$ & $\Theta_{1,3,3}$ &  \\
\end{longtable}}

{\small
\setlength{\tabcolsep}{4pt}
\begin{longtable}{@{}l l >{\raggedright\arraybackslash}p{6.3cm} l >{\raggedright\arraybackslash}p{4.4cm}@{}}
\caption{Reducible reduced members of $\lvert -2K_S\rvert$, i.e.\ $\deg D=8$ and $K_S+\frac12D\equiv0$, here $-2K_S=6H-2E_{12345}$.}\label{tab:deg8}\\
\toprule
No. & Type & Components (representative) & Lines & Remark \\
\midrule
\endfirsthead
\multicolumn{5}{@{}l}{\emph{Table~\ref{tab:deg8}, continued}}\\
\toprule
No. & Type & Components (representative) & Lines & Remark \\
\midrule
\endhead
\bottomrule
\endfoot
8.1 & $B_7\ell$ & $4H-E_{12345}$,\ $L_0$ & $K_1$ &  \\
8.2 & $B_6F$ & $4H-E_{1234}-2E_5$,\ $2H-E_{1234}$ & - &  \\
8.3 & $A_5T$ & $3H-E_{2345}$,\ $3H-2E_1-E_{2345}$ & - &  \\
8.4 & $A^{2}$ & $-K_S$,\ $-K_S$ & - &  \\
8.5 & $Q^{2}$ & $2H-E_{45}$,\ $4H-2E_{123}-E_{45}$ & - &  \\
\addlinespace[4pt]
8.6 & $B_6\ell^{2}$ & $4H-E_{1234}-2E_5$,\ $L_{34}$,\ $L_{12}$ & $K_2$ &  \\
8.7 & $A_6\ell^{2}$ & $3H-E_{345}$,\ $L_{12}$,\ $L_0$ & $2K_1$ &  \\
8.8 & $A_5F\ell$ & $3H-E_{2345}$,\ $H-E_1$,\ $L_0$ & $K_1$ &  \\
8.9 & $PF\ell$ & $2H-E_5$,\ $2H-E_{1234}$,\ $L_0$ & $K_1$ &  \\
8.10 & $AT\ell$ & $-K_S$,\ $H$,\ $L_0$ & $K_1$ &  \\
8.11 & $AF^{2}$ & $-K_S$,\ $H-E_5$,\ $2H-E_{1234}$ & - &  \\
8.12 & $QT\ell$ & $2H-E_{45}$,\ $2H-E_{123}$,\ $L_0$ & $K_1$ &  \\
8.13 & $QF^{2}$ & $2H-E_{45}$,\ $2H-E_{1235}$,\ $2H-E_{1234}$ & - &  \\
8.14 & $T^{2}F$ & $H$,\ $3H-E_{1234}-2E_5$,\ $2H-E_{1234}$ & - &  \\
\addlinespace[4pt]
8.15 & $A_5\ell^{3}$ & $3H-E_{2345}$,\ $L_{45}$,\ $L_{13}$,\ $L_{12}$ & $P_3$ &  \\
8.16 & $P\ell^{3}$ & $2H-E_5$,\ $L_{34}$,\ $L_{12}$,\ $L_0$ & $K_2\sqcup K_1$ &  \\
8.17 & $AF\ell^{2}$ & $-K_S$,\ $H-E_5$,\ $L_{34}$,\ $L_{12}$ & $K_2$ &  \\
8.18 & $QF\ell^{2}$ & $2H-E_{45}$,\ $H-E_3$,\ $L_{12}$,\ $L_0$ & $2K_1$ &  \\
8.19 & $QF\ell^{2}$ & $2H-E_{45}$,\ $2H-E_{1235}$,\ $L_{34}$,\ $L_{12}$ & $K_2$ &  \\
8.20 & $T^{2}\ell^{2}$ & $H$,\ $2H-E_{345}$,\ $L_{12}$,\ $L_0$ & $2K_1$ & $\{\ell_1,\ell_2\}=\{\ell_{T_1},\ell_{T_2}\}$ \\
8.21 & $T^{2}\ell^{2}$ & $H$,\ $3H-E_{1234}-2E_5$,\ $L_{34}$,\ $L_{12}$ & $K_2$ & $\ell_i\neq\ell_{T_j}$ \\
8.22 & $T^{2}\ell^{2}$ & $2H-E_{345}$,\ $2H-E_{125}$,\ $L_{34}$,\ $L_{12}$ & $K_2$ & $\{\ell_1,\ell_2\}=\{\ell_{T_1},\ell_{T_2}\}$ \\
8.23 & $TF^{2}\ell$ & $H$,\ $H-E_5$,\ $2H-E_{1234}$,\ $L_0$ & $K_1$ & $F_1+F_2\sim-K_S$ \\
8.24 & $TF^{2}\ell$ & $H$,\ $2H-E_{2345}$,\ $2H-E_{1345}$,\ $L_{12}$ & $K_1$ & $F_1\cdot F_2=1$ \\
8.25 & $F^{4}$ & $H-E_5$,\ $H-E_5$,\ $2H-E_{1234}$,\ $2H-E_{1234}$ & - & $F_1\sim F_2$ \\
8.26 & $F^{4}$ & $H-E_5$,\ $H-E_4$,\ $2H-E_{1235}$,\ $2H-E_{1234}$ & - & $F_1\cdot F_2=1$ \\
\addlinespace[4pt]
8.27 & $A\ell^{4}$ & $-K_S$,\ $E_1$,\ $L_{45}$,\ $L_{13}$,\ $L_{12}$ & $C_4$ &  \\
8.28 & $Q\ell^{4}$ & $2H-E_{45}$,\ $L_{45}$,\ $L_{23}$,\ $L_{13}$,\ $L_{12}$ & $K_{1,3}$ &  \\
8.29 & $Q\ell^{4}$ & $2H-E_{45}$,\ $L_{35}$,\ $L_{24}$,\ $L_{13}$,\ $L_{12}$ & $P_4$ &  \\
8.30 & $TF\ell^{3}$ & $H$,\ $H-E_5$,\ $L_{34}$,\ $L_{12}$,\ $L_0$ & $K_2\sqcup K_1$ &  \\
8.31 & $TF\ell^{3}$ & $H$,\ $2H-E_{2345}$,\ $L_{45}$,\ $L_{13}$,\ $L_{12}$ & $P_3$ & $\ell_T\not\subset D$ \\
8.32 & $TF\ell^{3}$ & $2H-E_{345}$,\ $H-E_2$,\ $L_{45}$,\ $L_{13}$,\ $L_{12}$ & $P_3$ & $\ell_T\subset D$ \\
8.33 & $F^{3}\ell^{2}$ & $H-E_5$,\ $H-E_5$,\ $2H-E_{1234}$,\ $L_{34}$,\ $L_{12}$ & $K_2$ & $F_1\sim F_2$ \\
8.34 & $F^{3}\ell^{2}$ & $H-E_5$,\ $H-E_4$,\ $H-E_3$,\ $L_{12}$,\ $L_0$ & $2K_1$ &  \\
8.35 & $F^{3}\ell^{2}$ & $H-E_5$,\ $H-E_4$,\ $2H-E_{1235}$,\ $L_{34}$,\ $L_{12}$ & $K_2$ & $F_1\cdot F_2=1$ \\
\addlinespace[4pt]
8.36 & $T\ell^{5}$ & $H$,\ $E_1$,\ $L_{45}$,\ $L_{13}$,\ $L_{12}$,\ $L_0$ & $C_4^{+}$ &  \\
8.37 & $T\ell^{5}$ & $H$,\ $L_{45}$,\ $L_{35}$,\ $L_{24}$,\ $L_{13}$,\ $L_{12}$ & $C_5$ &  \\
8.38 & $F^{2}\ell^{4}$ & $H-E_5$,\ $H-E_5$,\ $L_{34}$,\ $L_{24}$,\ $L_{13}$,\ $L_{12}$ & $2K_2$ &  \\
8.39 & $F^{2}\ell^{4}$ & $H-E_5$,\ $H-E_4$,\ $L_{45}$,\ $L_{23}$,\ $L_{13}$,\ $L_{12}$ & $K_{1,3}$ &  \\
8.40 & $F^{2}\ell^{4}$ & $H-E_5$,\ $H-E_4$,\ $L_{35}$,\ $L_{24}$,\ $L_{13}$,\ $L_{12}$ & $P_4$ &  \\
8.41 & $F^{2}\ell^{4}$ & $H-E_5$,\ $2H-E_{1234}$,\ $E_1$,\ $L_{45}$,\ $L_{13}$,\ $L_{12}$ & $C_4$ & each $F_i$ meets two opposite lines \\
8.42 & $F^{2}\ell^{4}$ & $H-E_5$,\ $2H-E_{1234}$,\ $E_1$,\ $L_{34}$,\ $L_{15}$,\ $L_{12}$ & $C_4$ & each $F_i$ meets two adjacent lines \\
\addlinespace[4pt]
8.43 & $F\ell^{6}$ & $H-E_5$,\ $E_1$,\ $L_{45}$,\ $L_{23}$,\ $L_{14}$,\ $L_{13}$,\ $L_{12}$ & $\Theta_{2,2,3}$ &  \\
8.44 & $F\ell^{6}$ & $H-E_5$,\ $E_1$,\ $L_{34}$,\ $L_{24}$,\ $L_{15}$,\ $L_{13}$,\ $L_{12}$ & $\Theta_{1,3,3}$ &  \\
\addlinespace[4pt]
8.45 & $\ell^{8}$ & $E_1$,\ $E_2$,\ $L_{45}$,\ $L_{25}$,\ $L_{23}$,\ $L_{14}$,\ $L_{13}$,\ $L_{12}$ & $V_8$ &  \\
8.46 & $\ell^{8}$ & $E_1$,\ $E_2$,\ $L_{25}$,\ $L_{24}$,\ $L_{23}$,\ $L_{15}$,\ $L_{14}$,\ $L_{13}$ & $Q_3$ &  \\
\end{longtable}}

\section{Surfaces of degree 5}

We work over an algebraically closed field $k$ of characteristic $0$.
Let $S$ be a del Pezzo surface of degree $5$. All such surfaces are isomorphic,
we fix $\beta\colon S\to\PP^2$, the blow-up of four points in general position,
with $H$ the pullback of a line and $E_1,\dots,E_4$ the exceptional curves, so
$\Pic(S)=\mathbb Z H\oplus\bigoplus_i\mathbb ZE_i$, $-K_S=3H-\sum_iE_i$ and
$K_S^2=5$. For a divisor $D$ we write $\deg D:=-K_S\cdot D$.

A \emph{Campana orbifold} on $S$ is a pair $(S,\Delta_\epsilon)$ with
\[
  \Delta_\epsilon=\sum_{i=1}^r\epsilon_iD_i,\qquad
  \epsilon_i=1-\frac1{m_i},\quad 2\le m_i<\infty,
\]
where $D_1,\dots,D_r$ are \emph{pairwise distinct} irreducible curves, the boundary
is $D:=\lceil\Delta_\epsilon\rceil=D_1+\dots+D_r$. Distinct components may be
linearly equivalent (e.g.\ two fibres of one conic bundle). We call
$(S,\Delta_\epsilon)$ a \emph{del Pezzo orbifold} if $-(K_S+\Delta_\epsilon)$ is
ample. In this note we classify the del Pezzo orbifolds of degree $5$ with
$r\ge2$, the case $r=1$ is \cite{dandapat2026classification}.

\begin{remark}[Degree bound]\label{rem:degree}
If $-(K_S+\Delta_\epsilon)$ is ample then, since $-K_S$ is ample and $\epsilon_i\ge\frac12$,
\[
0<(-K_S)\cdot\bigl(-(K_S+\Delta_\epsilon)\bigr)=5-\sum_i\epsilon_i\deg D_i\le 5-\tfrac12\sum_i\deg D_i ,
\]
hence $\sum_i\deg D_i\le 9$, in particular $r\le 9$ and $\deg D_i\le 9$ for all $i$.
(Theorem~\ref{thm:main5} shows a posteriori that $r\le5$ and $\deg D\le5$.)
\end{remark}

\subsection{Lines and the Petersen graph}
The ten $(-1)$-curves (\emph{lines}) of $S$ are $E_i$ ($1\le i\le4$) and
$L_{ij}=H-E_i-E_j$ ($1\le i<j\le4$). Let $\Gamma$ be the graph whose vertices are the
lines, two lines being adjacent if they meet (then $E\cdot E'=1$, otherwise
$E\cdot E'=0$). It is the Petersen graph.

\begin{lemma}\label{lem:petersen}
\begin{enumerate}[label=\textup{(\alph*)}]
\item Every line meets exactly three other lines.
\item Two meeting lines have no common neighbour, two disjoint lines have exactly
      one common neighbour. In particular $\Gamma$ has no cycles of length $3$ or $4$.
\item A line is determined by the set of lines meeting it.
\item The maximal sets of pairwise disjoint lines have $4$ elements, there are exactly
      five of them, $I_1,\dots,I_5$, and $|I_a\cap I_b|=1$ for $a\ne b$.
\item The six lines disjoint from a given line $E$ form a hexagon. For meeting lines
      $E,E'$, the four lines meeting neither $E$ nor $E'$ form two disjoint edges.
      For a path $E_a-E_0-E_b$ in $\Gamma$ there are exactly two lines $E_x,E_y$
      neither on the path nor adjacent to it, and $\{E_a,E_b,E_x,E_y\}$ is one of
      the $I_c$.
\item $\sum_{E\text{ line}}E\sim-2K_S$, hence $\sum_E A\cdot E=2\deg A$ for every divisor $A$.
\end{enumerate}
\end{lemma}
\begin{proof}
All statements are finite checks with the classes above, e.g.\ (f):
$\sum_iE_i+\sum_{i<j}(H-E_i-E_j)=6H-2\sum_iE_i$. They are standard properties of
the Petersen graph.
\end{proof}

\begin{lemma}\label{lem:ample}
The cone $\NE(S)$ is spanned by the ten lines. Consequently, for an $\R$-divisor
$\Delta$ on $S$, the divisor $-(K_S+\Delta)$ is ample if and only if
$\Delta\cdot E<1$ for every line $E$.
\end{lemma}
\begin{proof}
Since $-K_S$ is ample, the cone theorem gives $\NE(S)=\sum_R R$ over finitely many
extremal rays $R$, and as $\rho(S)=5\ge3$ each $R$ is spanned by a $(-1)$-curve. By Kleiman's criterion a divisor is ample iff it is
positive on each of these finitely many generators, and
$-(K_S+\Delta)\cdot E=1-\Delta\cdot E$.
\end{proof}

\subsection{Curves of small degree}
For a nef divisor $A$ put $\Lambda(A):=\{E\text{ line}: A\cdot E>0\}$.

\begin{lemma}\label{lem:curves5}
Let $C\subset S$ be an irreducible curve that is not a line. Then $C$ is nef and
$C^2\ge0$, if $C^2=0$ then $\deg C=2$ and $C$ is a \emph{conic}, i.e.\ a fibre of
one of the five conic bundles $S\to\PP^1$. The classes of irreducible curves of
degree $\le4$ other than lines are exactly:
\begin{enumerate}[label=\textup{(\roman*)}]
\item the five \emph{conic} classes $F_1,\dots,F_5$ ($H-E_i$ and $2H-\sum_iE_i$),
      with $F_a\cdot E\in\{0,1\}$ and $\Lambda(F_a)=I_a$,
\item the five \emph{cubic} classes $h_a:=-K_S-F_a$ (pullbacks of lines under the
      five birational morphisms $S\to\PP^2$), with $h_a\cdot E=1-F_a\cdot E$, so
      $\Lambda(h_a)=\Gamma\setminus I_a$ has six elements,
\item the ten \emph{quartic} classes $q_E:=-K_S-E$ ($E$ a line, pullbacks of
      $(1,1)$-curves under the morphisms $S\to\PP^1\times\PP^1$), with
      $q_E\cdot E=2$, $q_E\cdot E'=0$ if $E'$ meets $E$, and $q_E\cdot E'=1$
      for the six lines $E'$ disjoint from $E$.
\end{enumerate}
Each of these classes contains smooth rational curves, moving in a family of
dimension $1$, $2$, $3$ respectively.
\end{lemma}
\begin{proof}
$C\cdot E\ge0$ for all lines, so $C$ is nef by Lemma~\ref{lem:ample} and $C^2\ge0$.
By adjunction $p_a(C)=1+\frac12(C^2-\deg C)\ge0$, so $C^2=0$ forces $\deg C\le2$,
and $\deg C$ is even and positive. For the list: if $A=aH-\sum b_iE_i$ is nef then
$b_i=A\cdot E_i\ge0$ and $a-b_i-b_j=A\cdot L_{ij}\ge0$, hence
$\deg A=3a-\sum b_i\ge a$, so for $\deg A\le4$ there are finitely many
candidates, and a direct check (Table~\ref{tab:dp5-types}, types $T_2,T_3,T_4$)
gives (i)-(iii) and the intersection numbers. Representatives are the lines
through $p_1$ ($H-E_1$), all lines ($H$) and the conics through $p_3,p_4$
($2H-E_3-E_4=q_{L_{12}}$), $\Aut(S)\cong S_5$ acts on $\Pic(S)$ as the Weyl group
$W(A_4)$ and is transitive on each of the three sets of classes.
\end{proof}

\begin{lemma}\label{lem:criterion5}
Let $D=\sum D_i$ be a boundary such that $D\cdot E\le1$ for every line $E$ which is
not a component of $D$, and let $\epsilon_i\in(0,1)$. Then $-(K_S+\Delta_\epsilon)$ is
ample if and only if for every component $D_j$ which is a line
\begin{equation}\label{eq:criterion}
  \sum_{i\ne j}\epsilon_i\,(D_i\cdot D_j)\,<\,1+\epsilon_j .
\end{equation}
\end{lemma}
\begin{proof}
If a line $E$ is not a component, then $D_i\cdot E\ge0$ for all $i$ and
$\sum_iD_i\cdot E\le1$, so $\Delta_\epsilon\cdot E\in\{0\}\cup\{\epsilon_i\}$ is $<1$.
If $E=D_j$ then $\Delta_\epsilon\cdot E=\sum_{i\ne j}\epsilon_iD_i\cdot E-\epsilon_j$.
Conclude by Lemma~\ref{lem:ample}.
\end{proof}

\subsection{The classification}

In the statement, $E,E',E_0,E_a,\dots$ denote lines, $F,F'$ conics, $h$ a cubic and $q$
a quartic curve (Lemma~\ref{lem:curves5}), always irreducible, the weight of a
component $C$ is written $\epsilon_C$.

\begin{theorem}\label{thm:main5}
Let $S$ be a del Pezzo surface of degree $5$ and $\Delta_\epsilon=\sum_{i=1}^r\epsilon_iD_i$
a boundary with $r\ge2$ components. Then $-(K_S+\Delta_\epsilon)$ is ample if and
only if $D=\lceil\Delta_\epsilon\rceil$ is one of the following ten configurations and the
weights satisfy the stated inequalities (``none'' means that every choice of weights works). 
\begin{enumerate}[label=\textup{(\arabic*)},leftmargin=2.4em,itemsep=3pt]
\item $E+E'$ with $E\cdot E'=1$.\\ \emph{Weights:} none.
\item $E+F$ with $F\cdot E=1$.\\ \emph{Weights:} none.
\item $E+q$ with $q\in|-K_S-E|$ \textup{(}so $q\cdot E=2$\textup{)}.\\
      \emph{Weights:} $2\epsilon_q<1+\epsilon_E$.
\item $F+h$ with $h\in|-K_S-F|$ \textup{(}so $F\cdot h=2$\textup{)}.\\ \emph{Weights:} none.
\item $E_a+E_0+E_b$ with $E_a\cdot E_0=E_b\cdot E_0=1$ \textup{(}a chain of three lines\textup{)}.\\
      \emph{Weights:} $\epsilon_{E_a}+\epsilon_{E_b}<1+\epsilon_{E_0}$.
\item $E+E'+h$ with $E\cdot E'=1$ and $h\in|-K_S-E-E'|$.\\
      \emph{Weights:} $\epsilon_{E'}+\epsilon_h<1+\epsilon_E$ and
      $\epsilon_{E}+\epsilon_h<1+\epsilon_{E'}$.
\item $E+F+F'$, where $F,F'$ lie in the two conic classes with $F\cdot E=F'\cdot E=1$
      \textup{(}so $F+F'\sim-K_S-E$\textup{)}.\\
      \emph{Weights:} $\epsilon_F+\epsilon_{F'}<1+\epsilon_E$.
\item The \emph{claw} $E_0+E_a+E_b+E_c$, where $E_a,E_b,E_c$ are the three lines
      meeting $E_0$.\\
      \emph{Weights:} $\epsilon_{E_a}+\epsilon_{E_b}+\epsilon_{E_c}<1+\epsilon_{E_0}$.
\item $E_a+E_0+E_b+F$ with $E_a,E_0,E_b$ as in \textup{(5)} and
      $F\in|-K_S-E_a-E_0-E_b|$ a conic.\\
      \emph{Weights:} $\epsilon_{E_a}+\epsilon_{E_b}<1+\epsilon_{E_0}$,\quad
      $\epsilon_{E_0}+\epsilon_F<1+\epsilon_{E_a}$,\quad
      $\epsilon_{E_0}+\epsilon_F<1+\epsilon_{E_b}$.
\item The \emph{pentagon} $E_1+\dots+E_5$ with $E_i\cdot E_{i+1}=1$
      \textup{(}indices mod $5$\textup{)}.\\
      \emph{Weights:} $\epsilon_{E_{i-1}}+\epsilon_{E_{i+1}}<1+\epsilon_{E_i}$ for all $i$.
\end{enumerate}
Each configuration forms a single $\Aut(S)$-orbit of multisets of classes, explicit
representatives, orbit sizes and the inequalities in coordinates are given in
Table~\ref{tab:dp5-reducible}.
\end{theorem}

\begin{corollary}\label{cor:main}
Let $(S,\Delta_\epsilon)$ be a del Pezzo orbifold of degree $5$ with $r\ge2$. Then:
\begin{enumerate}[label=\textup{(\alph*)}]
\item $r\le5$ and every component is a line, a conic, a cubic or a quartic, in
      particular $\deg D_i\le4$.
\item Either $D\in|-K_S|$ \textup{(}cases $(3),(4),(6),(7),(9),(10)$, which are anticanonical cycles of lengths $2,2,3,3,4,5$\textup{)}, or $\deg D\le4$
      \textup{(}cases $(1),(2),(5),(8)$\textup{)}.
\item For every configuration except the claw, all $m_i=2$ is admissible. For the
      claw, all $m_i=2$ is not admissible, while $m_{E_0}=3$, $m_{E_a}=m_{E_b}=m_{E_c}=2$ is.
\item Every line $E$ not contained in $D$ satisfies $D\cdot E\le1$.
\end{enumerate}
\end{corollary}

\begin{remark}\label{rem:realize}
Every configuration is realised with $D$ simple normal crossing: two lines meet
transversally, no three lines are concurrent (no triangles in $\Gamma$), and the
conics, cubics and quartics may be taken to be general members of base point free
linear systems, which meet the fixed curves and each other transversally at distinct
points. Non-general members (e.g.\ $q$ tangent to $E$ in case $(3)$) are allowed by
the theorem, which only concerns numerical classes.
\end{remark}

\begin{remark}[Weight $1$]\label{rem:weight-one}
If one also allows $\epsilon_i=1$ ($m_i=\infty$), nothing changes: every component of
each of the ten configurations meets a line $E\not\subset D$ with $D_i\cdot E=1$, so
Lemma~\ref{lem:ample} forces $\epsilon_i<1$. Since ampleness is an open condition the
list of supports is the same in any case. 
\end{remark}

\subsection{Proof of Theorem~\ref{thm:main5}}

\begin{notation}
Let $\cL$ be the set of components of $D$ that are lines, $\ell:=|\cL|$, and $N$ the
sum of the remaining components (all nef by Lemma~\ref{lem:curves5}). For a line
$E$ let $\nu(E):=\#\{E'\in\cL: E'\cdot E=1\}$, so that $D\cdot E=N\cdot E+\nu(E)$ if
$E\notin\cL$ and $D\cdot E=N\cdot E+\nu(E)-1$ if $E\in\cL$.
\end{notation}

\begin{lemma}\label{lem:necessary}
If $-(K_S+\Delta_\epsilon)$ is ample then
\begin{enumerate}[label=\textup{(\roman*)}]
\item $N\cdot E+\nu(E)\le1$ for every line $E\notin\cL$,
\item $N\cdot E+\nu(E)\le3$ for every $E\in\cL$,
\item $2\deg N+\ell\le10$,
\item $\deg N+\ell\le9$.
\end{enumerate}
\end{lemma}
\begin{proof}
(i) All terms of $\Delta_\epsilon\cdot E$ are $\ge\frac12D_i\cdot E\ge0$, so
$1>\Delta_\epsilon\cdot E\ge\frac12 D\cdot E$. (ii) If $E=D_j$ then
$1>\Delta_\epsilon\cdot E\ge\frac12(D-E)\cdot E-\epsilon_j>\frac12(N\cdot E+\nu(E))-1$.
(iii) By (i), (ii) we have $D\cdot E\le1$ for $E\notin\cL$ and $D\cdot E\le2$ for
$E\in\cL$, summing over all lines and using Lemma~\ref{lem:petersen}(f),
$2(\deg N+\ell)=2\deg D\le(10-\ell)+2\ell$. (iv) is Remark~\ref{rem:degree}.
\end{proof}

\begin{lemma}\label{lem:lines5}
If $\cL$ satisfies $\nu(E)\le1$ for all lines $E\notin\cL$ and $\ell\le9$, then
$\cL$ is empty, a single line, an edge, a path $E_a-E_0-E_b$, a claw \textup{(}a line
together with the three lines meeting it\textup{)}, or a pentagon.
\end{lemma}
\begin{proof}
If $E,E'\in\cL$ are disjoint, their unique common neighbour $E''$
(Lemma~\ref{lem:petersen}(b)) has $\nu(E'')\ge2$, so $E''\in\cL$. Hence the induced
graph $\Gamma_\cL$ is connected, and it has no cycle of length $\le4$.
For $\ell\le3$ this leaves the first four cases. If $\ell=4$, $\Gamma_\cL$ is a tree,
i.e.\ a claw or a path $E_1-E_2-E_3-E_4$, in the latter case $E_1,E_4$ are disjoint
(no $4$-cycles) and their common neighbour is neither $E_2$ nor $E_3$ (no
triangles), a contradiction. If $\ell\ge5$, let $\mathcal O$ be the set of the
$10-\ell$ lines not in $\cL$, each has $\ge3-1=2$ neighbours in $\mathcal O$. For
$\ell\le9$, $\mathcal O\ne\emptyset$, so $\Gamma_{\mathcal O}$ contains a cycle of
length $\le10-\ell\le5$, by girth $5$ we get $\ell=5$ and $\Gamma_{\mathcal O}$ is a
pentagon. Then $\Gamma$ has $5$ edges inside $\mathcal O$ and $5$ edges between
$\mathcal O$ and $\cL$, so $\Gamma_\cL$ has $15-10=5$ edges on $5$ vertices and no
short cycles, i.e.\ it is a pentagon.
Conversely we check with Lemma~\ref{lem:petersen}(b) that all six configurations
satisfy $\nu\le1$ outside $\cL$.
\end{proof}

\begin{proof}[Proof of Theorem~\ref{thm:main5}]
\emph{Necessity.} Assume $-(K_S+\Delta_\epsilon)$ ample. By
Lemmas~\ref{lem:necessary}(i),(iv) and~\ref{lem:lines5}, $\cL$ is one of the six
configurations, and $\deg N\le(10-\ell)/2$ by Lemma~\ref{lem:necessary}(iii). By (i),
the sets $\Lambda(\cdot)$ of the non-line components are pairwise disjoint outside
$\cL$, contain no line $E\notin\cL$ with $\nu(E)=1$, and every non-line component
meets each line outside $\cL$ at most once. Recall that $\Lambda(F)$ is a maximal set
of pairwise disjoint lines (Lemma~\ref{lem:curves5}), so it cannot be contained in a
set of lines whose induced graph has independence number $\le3$.

\smallskip\noindent\emph{$\ell=0$.} Then $\deg N\le5$ and $r\ge2$, so $N$ consists of
two conics or of a conic and a cubic. Two conics are impossible, since
$\Lambda(F_a)\cap\Lambda(F_b)=I_a\cap I_b\ne\emptyset$ (also for $a=b$). For $F_a$ and
$h_b$ we need $I_a\cap(\Gamma\setminus I_b)=\emptyset$, i.e.\ $a=b$: case~(4).

\smallskip\noindent\emph{$\ell=1$, $\cL=\{E\}$.} Now $\deg N\le4$ and $N\neq0$, so $N$ is
a conic, a cubic, a quartic or two conics. The three neighbours of $E$ satisfy
$\nu=1$, so $\Lambda(N)\subseteq\{E\}\cup R$, where $R$ is the hexagon of lines
disjoint from $E$ (Lemma~\ref{lem:petersen}(e)), and $N\cdot E'\le1$ for $E'\in R$.
A conic needs an independent $4$-set in $\{E\}\cup R$: it must contain $E$ and one of
the two alternating triples of $R$, these give the two conic classes with
$F\cdot E=1$, case~(2). Two conics must then use both triples: case~(7), where
$N\cdot E=2\le3$. A cubic would need an independent $4$-set containing the three
neighbours of $E$, its fourth element would be disjoint from all three, but by
Lemma~\ref{lem:petersen}(b) every line other than $E$ that is disjoint from $E$ meets
exactly one neighbour of $E$, and $E$ meets all three: impossible. For a quartic
$q_{E'}$ the zero set is the set of neighbours of $E'$, which must equal the set of
neighbours of $E$, so $E'=E$ by Lemma~\ref{lem:petersen}(c), and $q_E\cdot E=2\le3$:
case~(3).

\smallskip\noindent\emph{$\ell=2$, $\cL=\{E,E'\}$ an edge.} Now $\deg N\le4$. The four
other neighbours of $E,E'$ have $\nu=1$, so $\Lambda(N)\subseteq\{E,E'\}\cup R$ with
$R$ two disjoint edges (Lemma~\ref{lem:petersen}(e)), this set of six lines is a union
of three disjoint edges, with independence number $3$. So there are no conics, and
no quartic ($|\Lambda(q)|=7$). If $N=h$ is a cubic, then $\Lambda(h)$ is this six-element
set, i.e.\ its complement (the four other neighbours of $E$ and $E'$) must be an
$I_a$, it is indeed independent (no triangles or $4$-cycles), so $h$ is unique, equal
to $-K_S-E-E'$, and $N\cdot E=N\cdot E'=1$ satisfy (ii): case~(6). Otherwise $N=0$:
case~(1).

\smallskip\noindent\emph{$\ell=3$, $\cL$ a path $E_a-E_0-E_b$.} Now $\deg N\le3$. The
five lines adjacent to the path have $\nu=1$, so
$\Lambda(N)\subseteq\{E_a,E_0,E_b,E_x,E_y\}$ in the notation of
Lemma~\ref{lem:petersen}(e). A cubic is impossible ($|\Lambda(h)|=6$). A conic needs an
independent $4$-subset, which must omit $E_0$: this is $\{E_a,E_b,E_x,E_y\}=I_c$, so
$F=F_c\sim-K_S-E_a-E_0-E_b$, (ii) holds: case~(9). Otherwise $N=0$: case~(5).

\smallskip\noindent\emph{$\ell=4$, $\cL$ a claw.} Now $\deg N\le3$. The six remaining lines
all have $\nu=1$, so $\Lambda(N)$ lies in the claw, whose independence number is $3$
and which has $4<6$ elements. Hence $N=0$: case~(8).

\smallskip\noindent\emph{$\ell=5$, $\cL$ a pentagon.} Now $\deg N\le2$, the other five
lines have $\nu=1$, so $\Lambda(N)$ lies in the pentagon (independence number $2$).
Hence $N=0$: case~(10).

\smallskip\noindent\emph{Sufficiency and the weight conditions.} Each of the ten
configurations satisfies $D\cdot E\le1$ for all lines $E\not\subset D$ (this was
checked in the case analysis), so Lemma~\ref{lem:criterion5} applies. Writing
out~\eqref{eq:criterion} for the line components gives the stated inequalities,
after discarding those that hold automatically because one side has a single
weight on the left ($\epsilon_C<1+\epsilon_E$). All the stated inequalities hold for
$m_i=2$ except in the claw, where $m_{E_0}=3$ and $m_{E_a}=m_{E_b}=m_{E_c}=2$ works.
Hence every configuration occurs. The statement about orbits is the output of the
script (\S\ref{sec:computer}).
\end{proof}

\subsection{Computer verification}\label{sec:computer}

We use Python 3 scripts which prove Theorem~\ref{thm:main5} independently of
\S3, using only Lemmas~\ref{lem:ample},~\ref{lem:curves5} and Remark~\ref{rem:degree}.
\begin{enumerate}[label=\arabic*.]
\item It builds $\Pic(S)$, the ten lines and the Weyl group $W(A_4)$ (checking
      order $120$ and that it preserves $K_S$ and the lines), and asserts every
      statement of Lemma~\ref{lem:petersen} and Lemma~\ref{lem:curves5}(i)-(iii).
\item It lists all classes $A$ of degree $\le9$ that are lines, conic classes, or nef
      with $A^2>0$ (a superset of the classes of irreducible curves, the search range
      $0\le b_i\le a\le9$ is exhaustive by the proof of Lemma~\ref{lem:curves5}):
      $346$ classes in $23$ orbits, Table~\ref{tab:dp5-types}.
\item It enumerates the multisets of such classes of total degree $\le9$ with $\ge2$
      members (lines at most once), pruned by Lemma~\ref{lem:necessary}(i),(ii), and
      reduces them modulo $W(A_4)$.
\item For each orbit it decides exactly (Fourier-Motzkin elimination with strict
      inequalities) whether weights in $[\frac12,1)$ exist, after lowering the
      weights of non-line components to $\frac12$ (harmless, as they have
      non-negative intersection with every line), and then exhibits explicit
      multiplicities $m_i$ and the list of non-automatic inequalities.
\end{enumerate}
Output: $10$ orbits pass the integral tests and all $10$ are admissible, with no
unresolved case. The second script examines all $92,602$ multisets of total degree
$\le9$ with $\ge2$ members, runs Fourier-Motzkin in \emph{all} weights, and
obtains the same $10$ orbits. The following tables represent the outputs. 

\begin{longtable}{clcccl}
\caption{$S_5$-orbits of classes of irreducible curves of anticanonical degree $\le 9$ on $S$. The profile lists the intersection numbers $D\cdot E$ with the ten lines, with multiplicities as exponents.}\label{tab:dp5-types}\\
\toprule type & representative & $-K_S\cdot D$ & $D^2$ & $p_a$ & profile $\{D\cdot E\}$ \\ \midrule\endfirsthead
\toprule type & representative & $-K_S\cdot D$ & $D^2$ & $p_a$ & profile $\{D\cdot E\}$ \\ \midrule\endhead
$T_{1}$$^{\ast}$ & $E_{1}$ & 1 & -1 & - & $-1,\,0^{6},\,1^{3}$ \\
$T_{2}$$^{\ast}$ & $H-E_{4}$ & 2 & 0 & 0 & $0^{6},\,1^{4}$ \\
$T_{3}$$^{\ast}$ & $H$ & 3 & 1 & 0 & $0^{4},\,1^{6}$ \\
$T_{4}$$^{\ast}$ & $2H-E_{3}-E_{4}$ & 4 & 2 & 0 & $0^{3},\,1^{6},\,2$ \\
$T_{5}$ & $2H-E_{4}$ & 5 & 3 & 0 & $0^{3},\,1^{4},\,2^{3}$ \\
$T_{6}$ & $3H-E_{1}-E_{2}-E_{3}-E_{4}$ & 5 & 5 & 1 & $1^{10}$ \\
$T_{7}$ & $2H$ & 6 & 4 & 0 & $0^{4},\,2^{6}$ \\
$T_{8}$ & $3H-E_{3}-2E_{4}$ & 6 & 4 & 0 & $0^{3},\,1^{3},\,2^{3},\,3$ \\
$T_{9}$ & $3H-E_{2}-E_{3}-E_{4}$ & 6 & 6 & 1 & $0,\,1^{6},\,2^{3}$ \\
$T_{10}$ & $3H-2E_{4}$ & 7 & 5 & 0 & $0^{3},\,1^{3},\,2,\,3^{3}$ \\
$T_{11}$ & $3H-E_{3}-E_{4}$ & 7 & 7 & 1 & $0^{2},\,1^{3},\,2^{4},\,3$ \\
$T_{12}$ & $4H-E_{1}-E_{2}-E_{3}-2E_{4}$ & 7 & 9 & 2 & $1^{6},\,2^{4}$ \\
$T_{13}$ & $3H-E_{4}$ & 8 & 8 & 1 & $0^{3},\,1,\,2^{3},\,3^{3}$ \\
$T_{14}$ & $4H-E_{3}-3E_{4}$ & 8 & 6 & 0 & $0^{3},\,1^{3},\,3^{3},\,4$ \\
$T_{15}$ & $4H-2E_{3}-2E_{4}$ & 8 & 8 & 1 & $0^{3},\,2^{6},\,4$ \\
$T_{16}$ & $4H-E_{2}-E_{3}-2E_{4}$ & 8 & 10 & 2 & $0,\,1^{4},\,2^{3},\,3^{2}$ \\
$T_{17}$ & $4H-E_{1}-E_{2}-E_{3}-E_{4}$ & 8 & 12 & 3 & $1^{4},\,2^{6}$ \\
$T_{18}$ & $3H$ & 9 & 9 & 1 & $0^{4},\,3^{6}$ \\
$T_{19}$ & $4H-3E_{4}$ & 9 & 7 & 0 & $0^{3},\,1^{3},\,3,\,4^{3}$ \\
$T_{20}$ & $4H-E_{3}-2E_{4}$ & 9 & 11 & 2 & $0^{2},\,1^{2},\,2^{3},\,3^{2},\,4$ \\
$T_{21}$ & $4H-E_{2}-E_{3}-E_{4}$ & 9 & 13 & 3 & $0,\,1^{3},\,2^{3},\,3^{3}$ \\
$T_{22}$ & $5H-E_{1}-E_{2}-E_{3}-3E_{4}$ & 9 & 13 & 3 & $1^{6},\,3^{4}$ \\
$T_{23}$ & $5H-E_{1}-E_{2}-2E_{3}-2E_{4}$ & 9 & 15 & 4 & $1^{3},\,2^{6},\,3$ \\
\bottomrule\end{longtable}

{\small
\begin{longtable}{rp{0.31\textwidth}lcp{0.33\textwidth}}
\caption{Boundaries $D=D_1+\dots+D_r$, $r\ge2$, on a quintic del Pezzo surface admitting weights with $-(K_S+\Delta_\epsilon)$ ample, up to $\mathrm{Aut}(S)\cong S_5$. The last column lists all ampleness conditions that are not automatic for $\epsilon_i\in[\tfrac12,1)$, $\epsilon_i$ is the weight of $D_i$ in the order of the second column. A $\dagger$ marks the cases in which all $m_i=2$ is admissible, \#orbit is the number of multisets of classes in the $S_5$-orbit.}\label{tab:dp5-reducible}\\
\toprule No. & $D_1,\dots,D_r$ & types & \#orbit & conditions on the weights \\ \midrule\endfirsthead
\toprule No. & $D_1,\dots,D_r$ & types & \#orbit & conditions on the weights \\ \midrule\endhead
1$^{\dagger}$ & $E_{1}$,\ $H-E_{1}-E_{4}$ & $T_{1},T_{1}$ & 15 & none \\
2$^{\dagger}$ & $E_{1}$,\ $H-E_{1}$ & $T_{1},T_{2}$ & 20 & none \\
3$^{\dagger}$ & $E_{1}$,\ $3H-2E_{1}-E_{2}-E_{3}-E_{4}$ & $T_{1},T_{4}$ & 10 & $-\epsilon_{1}+2\epsilon_{2}<1$ \\
4$^{\dagger}$ & $H-E_{4}$,\ $2H-E_{1}-E_{2}-E_{3}$ & $T_{2},T_{3}$ & 5 & none \\
5$^{\dagger}$ & $E_{1}$,\ $E_{2}$,\ $H-E_{1}-E_{2}$ & $T_{1},T_{1},T_{1}$ & 30 & $\epsilon_{1}+\epsilon_{2}-\epsilon_{3}<1$ \\
6$^{\dagger}$ & $E_{1}$,\ $H-E_{1}-E_{4}$,\ $2H-E_{1}-E_{2}-E_{3}$ & $T_{1},T_{1},T_{3}$ & 15 & $-\epsilon_{1}+\epsilon_{2}+\epsilon_{3}<1$, $\epsilon_{1}-\epsilon_{2}+\epsilon_{3}<1$ \\
7$^{\dagger}$ & $E_{1}$,\ $H-E_{1}$,\ $2H-E_{1}-E_{2}-E_{3}-E_{4}$ & $T_{1},T_{2},T_{2}$ & 10 & $-\epsilon_{1}+\epsilon_{2}+\epsilon_{3}<1$ \\
8 & $E_{1}$,\ $E_{2}$,\ $H-E_{3}-E_{4}$,\ $H-E_{1}-E_{2}$ & $T_{1},T_{1},T_{1},T_{1}$ & 10 & $\epsilon_{1}+\epsilon_{2}+\epsilon_{3}-\epsilon_{4}<1$ \\
9$^{\dagger}$ & $E_{1}$,\ $E_{2}$,\ $H-E_{1}-E_{2}$,\ $2H-E_{1}-E_{2}-E_{3}-E_{4}$ & $T_{1},T_{1},T_{1},T_{2}$ & 30 & $-\epsilon_{1}+\epsilon_{3}+\epsilon_{4}<1$, $-\epsilon_{2}+\epsilon_{3}+\epsilon_{4}<1$, $\epsilon_{1}+\epsilon_{2}-\epsilon_{3}<1$ \\
10$^{\dagger}$ & $E_{1}$,\ $E_{2}$,\ $H-E_{2}-E_{4}$,\ $H-E_{1}-E_{3}$,\ $H-E_{1}-E_{2}$ & $T_{1},T_{1},T_{1},T_{1},T_{1}$ & 12 & $-\epsilon_{1}+\epsilon_{4}+\epsilon_{5}<1$, $-\epsilon_{2}+\epsilon_{3}+\epsilon_{5}<1$, $\epsilon_{2}-\epsilon_{3}+\epsilon_{4}<1$, $\epsilon_{1}+\epsilon_{3}-\epsilon_{4}<1$, $\epsilon_{1}+\epsilon_{2}-\epsilon_{5}<1$ \\
\bottomrule\end{longtable}}

\bibliographystyle{amsalpha}
\bibliography{ref}

@article{chen2026campana,
  title = {Campana rational connectedness and weak approximation},
  author = {Qile Chen and Brian Lehmann and Sho Tanimoto},
  journal = {Algebraic Geometry},
  month = {8},
  number = {5},
  pages = {705--746},
  doi = {10.14231/AG-2026-021},
  volume = {13},
  year = {2026}
}

@article{Campana2011,
  author  = {Frédéric Campana},
  title   = {Orbifoldes g\'eom\'etriques sp\'eciales et classification bim\'eromorphe des vari\'et\'es k\"ahl\'eriennes compactes},
  journal = {Journal of the Institute of Mathematics of Jussieu},
  volume  = {10},
  number  = {4},
  pages   = {809--934},
  year    = {2011},
  doi     = {10.1017/S1474748010000101},
}

@incollection{Campana2011Survey,
  author    = {Frédéric Campana},
  title     = {Special orbifolds and birational classification: a survey},
  booktitle = {Classification of Algebraic Varieties, EMS Ser. Congr. Rep.},
  pages     = {123--170},
  publisher = {European Mathematical Society, Zürich},
  year      = {2011},
  doi       = {10.4171/007-1/6},
  mrnumber  = {2779470},
}

@inproceedings{campana2004orbifolds,
  title={Orbifolds, special varieties and classification theory},
  author={Frédéric Campana},
  booktitle={Annales de l'institut Fourier},
  volume={54},
  number={3},
  pages={499--630},
  year={2004}
}

@article{dandapat2026classification,
  title={Classification of low degree del Pezzo orbifolds},
  author={Dandapat, Saptarshi},
  journal={arXiv preprint arXiv:2601.21456},
  year={2026}
}

@article{kollar1992rational,
  title={Rational connectedness and boundedness of Fano manifolds},
  author={Koll{\'a}r, J{\'a}nos and Miyaoka, Yoichi and Mori, Shigefumi},
  journal={Journal of Differential Geometry},
  volume={36},
  number={3},
  pages={765--779},
  year={1992},
  publisher={Lehigh University}
}

@article{hacon2007shokurov,
  title={On Shokurov's rational connectedness conjecture},
  author={Hacon, Christopher D and Mckernan, James},
  journal={Duke Math. J.},
  volume={136},
  number={1},
  pages={119--136},
  year={2007}
}

@article{zhang2006rational,
  title={Rational connectedness of log \textbf{Q}-Fano varieties},
  author={Zhang, Q},
  journal={J. reine angew. Math.},
  volume={590},
  pages={131},
  year={2006}
}

@book{keel1999rational,
  title={Rational curves on quasi-projective surfaces},
  author={Keel, Se{\'a}n and McKernan, James},
  volume={669},
  year={1999},
  publisher={American Mathematical Soc.}
}

@article{pieropan2021campana,
  title={Campana points of bounded height on vector group compactifications},
  author={Pieropan, Marta and Smeets, Arne and Tanimoto, Sho and V{\'a}rilly-Alvarado, Anthony},
  journal={Proceedings of the London Mathematical Society},
  volume={123},
  number={1},
  pages={57--101},
  year={2021},
  publisher={Wiley Online Library}
}

@article{graber2003families,
  title={Families of rationally connected varieties},
  author={Graber, Tom and Harris, Joe and Starr, Jason},
  journal={Journal of the American Mathematical Society},
  volume={16},
  number={1},
  pages={57--67},
  year={2003}
}

@ARTICLE{hassett2005, 
  author = {{Hassett}, Brendan and {Tschinkel}, Yuri},
  title = "{Weak approximation over function fields}",
  journal = {Inventiones Mathematicae},
  year = 2006,
  month = jul,
  volume = {163},
  number = {1},
  pages = {171-190},
  doi = {10.1007/s00222-005-0458-8},
  archivePrefix = {arXiv},
  eprint = {math/0411367},
  primaryClass = {math.AG},
}

@article{mitankin2024semi,
  title={Semi-integral Brauer--Manin obstruction and quadric orbifold pairs},
  author={Mitankin, Vladimir and Nakahara, Masahiro and Streeter, Sam},
  journal={Transactions of the American Mathematical Society},
  volume={377},
  number={06},
  pages={4435--4480},
  year={2024}
}

@book{dolgachev2012classical,
  title={Classical algebraic geometry: a modern view},
  author={Dolgachev, Igor V},
  year={2012},
  publisher={Cambridge University Press}
}

@book{manin1986cubic,
  title={Cubic forms: algebra, geometry, arithmetic},
  author={Manin, Yu I},
  volume={4},
  year={1986},
  publisher={Elsevier}
}

@article{dandapat2026campana,
  title={Campana Rational Connectedness and Weak Approximation of Del Pezzo Orbifolds},
  author={Dandapat, Saptarshi},
  journal={arXiv preprint arXiv:2609.11533},
  year={2026}
}

@article{campana2005fibres,
  title={Fibres multiples sur les surfaces: aspects geom{\'e}triques, hyperboliques et arithm{\'e}tiques},
  author={Campana, Fr{\'e}d{\'e}ric},
  journal={manuscripta mathematica},
  volume={117},
  number={4},
  pages={429--461},
  year={2005},
  publisher={Springer}
}

@article{abramovich2009birational,
  title={Birational geometry for number theorists},
  author={Abramovich, Dan},
  journal={Arithmetic geometry},
  volume={8},
  pages={335--373},
  year={2009},
  publisher={Amer. Math. Soc Providence, RI}
}

@article{campana2015special,
  title={Special manifolds, arithmetic and hyperbolic aspects: a short survey},
  author={Campana, Fr{\'e}d{\'e}ric},
  journal={Contemporary Mathematics},
  volume={654},
  year={2015},
  publisher={American Mathematical Society}
}

@article{pieropan2024hyperbola,
  title={Hyperbola method on toric varieties},
  author={Pieropan, Marta and Schindler, Damaris},
  journal={Journal de l’{\'E}cole polytechnique—Math{\'e}matiques},
  volume={11},
  pages={107--157},
  year={2024}
}

@article{maeda1986classification,
  title={Classification of logarithmic Fano threefolds},
  author={Maeda, Hironobu},
  journal={Compositio Mathematica},
  volume={57},
  number={1},
  pages={81--125},
  year={1986}
}

@article{cheltsov2015asymptotically,
  title={Asymptotically log Fano varieties},
  author={Cheltsov, Ivan A and Rubinstein, Yanir A},
  journal={Advances in Mathematics},
  volume={285},
  pages={1241--1300},
  year={2015},
  publisher={Elsevier}
}

@inproceedings{campana1992connexite,
  title={Connexit{\'e} rationnelle des vari{\'e}t{\'e}s de Fano},
  author={Campana, Fr{\'e}d{\'e}ric},
  booktitle={Annales scientifiques de l'{\'E}cole normale sup{\'e}rieure},
  volume={25},
  number={5},
  pages={539--545},
  year={1992}
}

@article{feng2025campana,
  title={Campana separable rational connectedness of toric orbifold},
  author={Feng, Enhao and Mehidi, Sara},
  journal={arXiv preprint arXiv:2511.22545},
  year={2025}
}

@article{campana2016orbifold,
  title={Orbifold slope-rational connectedness},
  author={Campana, Frederic},
  journal={arXiv preprint arXiv:1607.07829},
  year={2016}
}

@article{ChenZhu2019,
url = {https://doi.org/10.1515/crelle-2017-0028},
title = {$\mathbb{A}^1$-curves on log smooth varieties},
author = {Qile Chen and Yi Zhu},
pages = {1--35},
volume = {2019},
number = {756},
journal = {Journal für die reine und angewandte Mathematik (Crelles Journal)},
doi = {doi:10.1515/crelle-2017-0028},
year = {2019},
lastchecked = {2026-09-23}
}

@article{abramovich2018campana,
  title={Campana points, Vojta’s conjecture, and level structures on semistable abelian varieties},
  author={Abramovich, Dan and V{\'a}rilly-Alvarado, Anthony},
  journal={Journal de th{\'e}orie des nombres de Bordeaux},
  volume={30},
  number={2},
  pages={525--532},
  year={2018}
}

@article{shute2024semi,
  title={Semi-integral points of bounded height on toric varieties},
  author={Shute, Alec and Streeter, Sam},
  journal={arXiv preprint arXiv:2410.02039},
  year={2024}
}

@article{xiao2022campana,
  title={Campana points on biequivariant compactifications of the Heisenberg group},
  author={Xiao, Huan},
  journal={European Journal of Mathematics},
  volume={8},
  number={1},
  pages={205--246},
  year={2022},
  publisher={Springer}
}

@article{chow2026campana,
  title={Campana points on wonderful compactifications: D. Chow et al.},
  author={Chow, Dylon and Loughran, Daniel and Takloo-Bighash, Ramin and Tanimoto, Sho},
  journal={Mathematische Annalen},
  volume={394},
  number={3},
  pages={69},
  year={2026},
  publisher={Springer}
}

@article{browning2012sums,
  title={Sums of three squareful numbers},
  author={Browning, Tim D and Valckenborgh, K Van},
  journal={Experimental Mathematics},
  volume={21},
  number={2},
  pages={204--211},
  year={2012},
  publisher={Taylor \& Francis}
}

@article{van2012squareful,
  title={Squareful numbers in hyperplanes},
  author={Van Valckenborgh, Karl},
  journal={Algebra \& Number Theory},
  volume={6},
  number={5},
  pages={1019--1041},
  year={2012},
  publisher={Mathematical Sciences Publishers}
}

@article{browning2021arithmetic,
  title={Arithmetic of higher-dimensional orbifolds and a mixed Waring problem},
  author={Browning, Tim and Yamagishi, Shuntaro},
  journal={Mathematische Zeitschrift},
  volume={299},
  number={1},
  pages={1071--1101},
  year={2021},
  publisher={Springer}
}

@article{streeter2022campana,
  title={Campana points and powerful values of norm forms},
  author={Streeter, Sam},
  journal={Mathematische Zeitschrift},
  volume={301},
  number={1},
  pages={627--664},
  year={2022},
  publisher={Springer}
}

@article{shute2022leading,
  title={On the leading constant in the Manin-type conjecture for Campana points},
  author={Shute, Alec},
  journal={Acta Arithmetica},
  year={2022},
  publisher={Institute of Mathematics}
}

@article{Browning2008ManinsCF,
  title={Manin’s conjecture for quartic Del Pezzo surfaces with a conic fibration},
  author={Tim D. Browning and R{\'e}gis de la Bret{\`e}che},
  journal={Duke Mathematical Journal},
  year={2008},
  volume={160},
  pages={1-69},
  url={https://api.semanticscholar.org/CorpusID:12793059}
}

@article{dlB02,
author = {de la Bretèche, Régis},
year = {2002},
month = {06},
pages = {},
title = {Nombre de points de hauteur bornée sur les surfaces de del Pezzo de degré 5},
volume = {113},
journal = {Duke Mathematical Journal},
doi = {10.1215/S0012-7094-02-11332-5}
}

@article{lehmann2022geometric,
  title={Geometric consistency of Manin's conjecture},
  author={Lehmann, Brian and Sengupta, Akash Kumar and Tanimoto, Sho},
  journal={Compositio Mathematica},
  volume={158},
  number={6},
  pages={1375--1427},
  year={2022},
  publisher={London Mathematical Society}
}

@inproceedings{ellenberg2023heights,
  title={Heights on stacks and a generalized Batyrev--Manin--Malle conjecture},
  author={Ellenberg, Jordan S and Satriano, Matthew and Zureick-Brown, David},
  booktitle={Forum of Mathematics, Sigma},
  volume={11},
  pages={e14},
  year={2023},
  organization={Cambridge University Press}
}

@article{darda2024batyrev,
  title={The Batyrev--Manin conjecture for DM stacks},
  author={Darda, Ratko and Yasuda, Takehiko},
  journal={Journal of the European Mathematical Society},
  volume={28},
  number={6},
  pages={2351--2400},
  year={2024}
}

@article{loughran2024malle,
  title={Malle's conjecture and Brauer groups of stacks},
  author={Loughran, Daniel and Santens, Tim},
  journal={arXiv preprint arXiv:2412.04196},
  year={2024}
}

@misc{chen2026maninsconjecturesemiintegralcurves,
      title={Manin's conjecture for semi-integral curves and $\mathbb{A}^1$-connectedness}, 
      author={Qile Chen and Brian Lehmann and Sho Tanimoto},
      year={2026},
      eprint={2605.19898},
      archivePrefix={arXiv},
      primaryClass={math.AG},
      url={https://arxiv.org/abs/2605.19898}, 
}

@article{faisant2025motivic,
  title={Motivic counting of rational curves with tangency conditions via universal torsors},
  author={Faisant, Lo{\"\i}s},
  journal={arXiv preprint arXiv:2502.11704},
  year={2025}
}

@article{das2025homological,
  title={Homological stability and Manin's conjecture for rational curves on quartic del Pezzo surfaces},
  author={Das, Ronno and Lehmann, Brian and Tanimoto, Sho and Tosteson, Philip},
  journal={arXiv preprint arXiv:2506.17071},
  year={2025}
}

@article{lehmann2019geometric,
  title={Geometric Manin’s conjecture and rational curves},
  author={Lehmann, Brian and Tanimoto, Sho},
  journal={Compositio Mathematica},
  volume={155},
  number={5},
  pages={833--862},
  year={2019},
  publisher={London Mathematical Society}
}

@article{testa2009,
          number = {No.1},
           title = {The irreducibility of the spaces of rational curves on del Pezzo surfaces},
            year = {2009},
          volume = {Vol.18},
           pages = {37--61},
         journal = {Journal of Algebraic Geometry},
       publisher = {University Press, Inc.},
             url = {http://www.ams.org/journals/jag},
          author = {Testa, Damiano},
            issn = {1056-3911}
}

@inproceedings{Tanimoto2025,
  title={The spaces of rational curves on del Pezzo surfaces via conic bundles},
  author={Sho Tanimoto},
  year={2025},
  booktitle={to appear in Simons Symposia, arXiv:2506:18239}
}

@inproceedings{rousseau2021nonspecial,
  title={Nonspecial varieties and generalised Lang--Vojta conjectures},
  author={Rousseau, Erwan and Turchet, Amos and Wang, Julie Tzu-Yueh},
  booktitle={Forum of Mathematics, Sigma},
  volume={9},
  pages={e11},
  year={2021},
  organization={Cambridge University Press}
}

@article{nakahara2024weak,
  title={Weak approximation and the Hilbert property for Campana points},
  author={Nakahara, Masahiro and Streeter, Sam},
  journal={Michigan Mathematical Journal},
  volume={74},
  number={2},
  pages={227--252},
  year={2024},
  publisher={University of Michigan, Department of Mathematics}
}

@article{moerman2024generalized,
  title={Generalized Campana points and adelic approximation on toric varieties},
  author={Moerman, Boaz},
  journal={arXiv preprint arXiv:2407.03048},
  year={2024}
}

@incollection{colliot2003points,
  title={Points rationnels sur les fibrations},
  author={Colliot-Th{\'e}lene, Jean-Louis},
  booktitle={Higher dimensional varieties and rational points},
  pages={171--221},
  year={2003},
  publisher={Springer}
}

@book{colliot2021brauer,
  title={The Brauer-Grothendieck group},
  author={Colliot-Th{\'e}lene, Jean-Louis and Skorobogatov, Alexei N},
  volume={71},
  year={2021},
  publisher={Springer}
}

@article{lehmann2026fano,
  title={Fano orbifolds admit free Campana curves},
  author={Lehmann, Brian and Tanimoto, Sho},
  journal={arXiv preprint arXiv:2609.37513},
  year={2026}
}

\end{document}